\documentclass[a4paper,11pt]{article}
\usepackage[english]{babel}
\usepackage[tbtags]{amsmath}
\usepackage{amssymb}
\usepackage{amsfonts}
\usepackage{amsthm}
\usepackage{calrsfs}
\usepackage{array}
\usepackage{bbm}
\usepackage{stmaryrd}
\usepackage{upgreek}
\usepackage[svgnames]{xcolor}
\usepackage{hyperref}
\usepackage{mathtools}
\usepackage{cancel}
\definecolor{addblue}{RGB}{0,70,200}
\definecolor{delred}{RGB}{200,0,0}
\definecolor{notegray}{RGB}{110,110,110}
\newcommand{\add}[1]{\textcolor{addblue}{#1}}
\newcommand{\delword}[1]{\setbox0\hbox{#1}\mbox{\rlap{\rule[.55ex]{\wd0}{.45pt}}\box0}}
\def\delwords#1 #2\delend{\delword{#1}\ifx\relax#2\relax\else\space\delwords#2\delend\fi}
\newcommand{\del}[1]{\textcolor{delred}{\delwords#1 \delend}}

\newcommand{\note}[1]{\par\smallskip\noindent{\small\color{notegray}\emph{Note to author: #1}}\par\smallskip}

\def\pp{\frac{\p}{\sqrt{1-|y|^2}}}
\def\p{\rho_{\eta}}
\def\div{\, \mbox{div}\,}
\def\d{\displaystyle}

\def\ps{\phi(s)}

\numberwithin{equation}{section}
\def\ibint{\int_{B}}
\def\iint{\int_B}
\def\grad{\nabla}

\def\I{B(x_0,T(x_0)-t)}
\def\nn{\nu_0}
\def\L{M_{\eta}(w(s))}

\def\v{{\mathrm{d}}v}

\def\Box{\hfill\rule{2.5mm}{2.5mm}}
\def\t{{\mathrm{d}}\tau}

\def\y{{\mathrm{d}}y}
\def\e{\varepsilon}
\def\no{\nonumber}

\def\R{\mathbb R}

\def\R{{\mathbb {R}}}

\newcommand\ttttaa{s\ge     3-\ln( T^*(x))}
\newcommand\tttaa{s\ge     2-\ln( T^*(x))}
\newcommand\ttaa{s\ge     1-\ln( T^*(x))}
\newcommand\taa{s\ge     -\ln( T^*(x))}

\newcommand{\ds}{\displaystyle}

\renewcommand{\H}{{\mathcal H}}

\theoremstyle{plain}
\newtheorem{thm}{Theorem}
\newtheorem*{thm*}{Theorem}

\newtheorem{prop}{Proposition}[section]
\newtheorem{cors}{Corollary}[section]
\newtheorem{lem}[prop]{Lemma}

\theoremstyle{remark}
\newtheorem{nb}{Remark}[section]

\makeatletter
\def\blfootnote{\xdef\@thefnmark{}\@footnotetext}
\makeatother
\title{\bf The blow-up rate for a log non-scaling invariant semilinear wave equation in the conformal regime}

\date{}

\author{Mohamed  Ali  Hamza\\
{\it \small 
Imam Abdulrahman Bin Faisal University,
 Dammam, 34212, Saudi Arabia}}

\allowdisplaybreaks

\begin{document}
\maketitle

\begin{abstract}
We consider the blow-up behavior of solutions to the semilinear wave equation
\[
\partial_t^2 u - \Delta u = |u|^{p-1}u \ln^a(u^2+2),
\qquad (x,t)\in \mathbb{R}^n \times [0,T),
\]
in the conformal case \( p = p_c = 1 + \frac{4}{n-1} \), with $n\ge 2$ and $a<0$.
Previous results in \cite{HZjmaa2020, HZ2022} show that for \( a \in \mathbb{R} \), solutions in the subconformal regime \( p < p_c \) blow up with a Type~I rate
at any non-characteristic point.
The objective of this work is to extend this blow-up rate to the conformal regime under the assumption \( a < 0 \).
We establish an a priori upper bound for any blow-up solution and construct a Lyapunov functional in similarity variables.
The resulting functional exhibits only weak dissipation,
of size $O(1/s)$, which only yields time-averaged energy bounds of order $s$. We remove this loss with an additional energy identity for the conformal problem that controls the boundary flux, and conclude with interpolation and covering arguments.
To our knowledge, this is the first blow-up rate result at the conformal exponent for a wave equation whose leading nonlinearity is itself a logarithmic, non-homogeneous perturbation of the power.
\end{abstract}




\section{Introduction}

\subsection{ Motivation of the problem }
This paper is devoted to the study of blow-up solutions for the
following semilinear  wave equation:
\begin{equation}\label{gen}
\left\{
\begin{array}{l}
\partial_t^2 u =\Delta u +|u|^{p-1}u \ln^a (u^2+2),\qquad  (x,t)\in \R^n\times [0,T),\\
\\
u(x,0)=u_0(x)\in  H^{1}_{loc,u}(\R^n),\qquad 
\partial_tu(x,0)=u_1(x)\in  L^{2}_{loc,u}(\R^n),
\end{array}
\right.
\end{equation}
where $u(t):x\in{\R^n} \rightarrow u(x,t)\in{\R}.$ 
 We assume in addition  that  $n\ge 2$ and
\begin{equation}\label{pc}
p=p_c\equiv 1+\frac4{n-1}. 
\end{equation}
The spaces  $L^{2}_{loc,u}(\R^n)$ and $H^{1}_{loc,u}(\R^n)$ are  defined by
\begin{equation*}
L^{2}_{loc,u}(\R^n)=\{u:\R^n\rightarrow \R/ \sup_{d\in \R^n}(\int_{|x-d|\le 1}|u(x)|^2dx)<+\infty \},
\end{equation*}
and
\begin{equation*}
H^{1}_{loc,u}(\R^n)=\{u\in L^{2}_{loc,u}(\R^n),|\grad u|\in L^{2}_{loc,u}(\R^n) \}.
\end{equation*}

\medskip

Semilinear wave equations with a nonlinearity featuring a logarithmic factor have been used in various nonlinear physical models, such as nuclear physics, wave mechanics, optics, and geophysics (see, e.g., \cite{Bia75, Bia76}).

\bigskip

  The Cauchy problem for the  equation ({\ref{gen}}) is well-posed in
$H^{1}_{loc}(\R^n)\times L^{2}_{loc}(\R^n)$, locally in
time. This follows from the finite speed
of propagation and the well-posedness in $H^{1} (\R^n)\times L^{2}(\R^n)$,
valid whenever $ 1< p< p_S$, where $p_S\equiv 1+\frac4{n-2}.$
Combining the blow-up solutions for the associated ordinary differential equation of (\ref{gen})   and 
the finite speed of
propagation, we conclude that there exists a
 blow-up solution $u(t)$ of (\ref{gen})  which depends non-trivially on the space variable. 

\medskip

If $u$ is an arbitrary blow-up solution of   \eqref{gen}, we define
a 1-Lipschitz curve $\Gamma=\{(x,T(x))\}$
such that
the domain of definition of $u$ is written as
\begin{equation}\label{defdu}
D=\{(x,t)\;|\; t< T(x)\}.
\end{equation}
The justification of this fact can be found in Alinhac
\cite{Apndeta95}, where such a set is referred to as the \textit{maximal influence domain}. 
 In fact, for any $x_0\in \R^n$, we derive a solution
which is defined in $K_{x_0,t_0}$, the backward light cone
defined by
\[
K_{x_0,t_0}=\{(x,t)\in \R^n\times[0,t_0)\;|\;|x-x_0|<t_0-t\}.
\]

\medskip

Let $T(x_0)$ denote the supremum of all such $t_0$.
 This leads to two distinct scenarios:
 \begin{itemize}
 \item Global Existence: If $T(x_1) = \infty$ for some $x_1 \in \mathbb{R}^n$, then by the geometry of light cones, $T(x) = \infty$ for all $x \in \mathbb{R}^n$, and the solution exists globally in time.
 \item Finite-time Blow-up: If $T(x) < +\infty$ for all $x \in \mathbb{R}^n$, the domain $D$ is the union of all such maximal backward light cones of the type $K_{x_0,t_0}$. This union  defines a 1-Lipschitz function $x \mapsto T(x)$ as described in \eqref{defdu}.
 \end{itemize}
The time
$\bar T=\inf_{x\in {\R^n}}T(x)$ and $\Gamma$ are called the blow-up time and the blow-up graph of $u$.
A point $x_0$ is non characteristic
if  there are
\begin{equation}\label{nonchar}
\delta_0\in(0,1)\mbox{ and }t_0<T(x_0)\mbox{ such that }
u\;\;\mbox{is defined on }{\mathcal C}_{x_0, T(x_0), \delta_0}\cap \{t\ge t_0\},
\end{equation}
where ${\cal C}_{\bar x, \bar t, \bar \delta}=\{(x,t)\;|\; t< \bar
t-\bar \delta|x-\bar x|\}$. If not, $x_0$ is said to be characteristic.

\bigskip

In the case $a=0$, equation \eqref{gen} reduces to the semilinear wave equation with power nonlinearity:
\begin{equation}\label{NLW}
\partial^2_{t} u =\Delta  u+|u|^{p-1}u,   \,\,\,(x,t)\in \R^n \times [0,T),
\end{equation}
The blow-up rate of singular solutions to \eqref{NLW} was extensively studied by Merle and Zaag in \cite{MZajm03, MZimrn05, MZma05}, where they established that in the subconformal and conformal regimes ($1 < p \leq p_c$), the blow-up rate is governed by the dynamics of the associated ordinary differential equation (ODE) $u'' = u^{p}$.
 This specific behavior is classified as Type I blow-up (or the ODE regime). 
 Conversely, any solution that does not follow this rate is classified as Type II.

\medskip

 Specifically, if $u$ is a solution  of equation \eqref{NLW} with a blow-up graph $\Gamma: x \mapsto T(x)$, and $x_0$ is a non-characteristic point, then for all $t \in [\frac{3}{4}T(x_0), T(x_0)]$, the following estimate holds:
 \begin{eqnarray}\label{mzmz}
  0 < \varepsilon_{0}(p,n)\leq (T(x_{0})-t)^{\frac{2}{
   p-1}}\frac{\|u(t)\|_{L^{2}(\I )}}
{(T(x_{0})-t)^{\frac{n}2}}\\
  +(T(x_{0})-t)^{\frac{2}{ p-1}+1}\Big(\frac{\|\partial_{t} u(t)\|_{L^{2}(\I)}}
  {(T(x_{0})-t)^{\frac{n}2}}
  + \frac{\|\partial_x u(t)\|_{L^{2}(\I )}}{{(T(x_{0})-t)^{\frac{n}2}}}\Big)\leq K_0,\nonumber
  \end{eqnarray}
 where the constant $K_0$ depends only on  $p, n, T(x_{0}), \delta_{0}(x_{0})$, together with the
 norm of initial data
 in $H^{1}_{loc,u}(\R^n)\times L^{2}_{loc,u}(\R^n)$.

\medskip

The proof is given in the
framework of similarity variables first introduced in Antonini and Merle \cite{AM} and
defined for any $x_0 \in \R^n$ and $T_0 \in  [0, T (x_0))$ by
$$ y = \frac{x-x_0}{
T_0-t} ,\ \  s = -\ln(T_0- t),\quad 
\textrm{and} \qquad w_{x_0,T_0} (y, s)= (T_0-t)^{\frac2{p-1}}u(x, t).$$
From \eqref{NLW}, the  function $w_{x_0,T_0}$  (we write $w$ for
simplicity) satisfies the following equation for all $y\in B$ and
$s\ge  -\ln T_0$,
where $B\equiv B(0,1)$ denotes the unit ball in $\mathbb{R}^n$:
\begin{equation*}\label{w0}
\partial_{s}^2w=\frac{1}{\rho_{\alpha_0}}\div(\rho_{\alpha_0} \grad w-\rho_{\alpha_0}(y\cdot \grad w)y)
-\frac{2p+2}{(p-1)^2}w
-\frac{p+3}{p-1}\partial_s w
-2y\cdot \grad \partial_{s}w+
|w|^{p-1}w,
\end{equation*}
where $\rho_{\alpha_0} (y)=(1-|y|^2)^{\alpha_0}$, with
$\alpha_0=\alpha_0(p) =\frac{2}{p-1}-\frac{n-1}2$.
As established   in \cite{AM,MZajm03, MZimrn05, MZma05}, by multiplying \eqref{w0}
by $\partial_sw\rho_{\alpha_0}$,   and   integrating over the unit ball $B$, 
 we construct a Lyapunov functional $\tilde E_{\alpha_0}(w(s))$ 
that satisfies the following dissipative property:
\begin{equation*}
\frac{\mathrm{d}}{\mathrm{d}s} \tilde E_{\alpha_0}(w(s)) = 
\begin{cases} 
-2\alpha_0 \displaystyle\int \frac{(\partial_{s}w)^2 \rho_{\alpha_0}}{1-|y|^2} \, \mathrm{d}y, & \text{if } p < p_c, \\[15pt]
-\displaystyle\int_{\partial B} (\partial_{s}w)^2 \, \mathrm{d}\sigma, & \text{if } p = p_c.
\end{cases}
\end{equation*} 
 Establishing \eqref{mzmz} relies fundamentally on the construction of a Lyapunov functional and the analysis of its dissipative term.  Furthermore, the argument employs Sobolev interpolations, critical Gagliardo-Nirenberg estimates, and a geometric covering technique specifically adapted to the blow-up surface to deduce the required estimate.  It is worth noting that, unlike the sub-conformal case, the dissipation in the conformal regime degenerates at the boundary, necessitating refined arguments to address this additional difficulty.

\medskip

In the one-dimensional setting,  further refinements were established in \cite{MZjfa07, MZcmp08, MZajm11, MZisol10, CZcpam13}. These results were adapted to the radial case away from the origin in \cite{MZbsm11} and further extended to the Klein-Gordon equation and other damped lower-order perturbations in \cite{HZkg12}. In later works \cite{MZods, MZods14}, Merle and Zaag extended their analysis to higher dimensions under subconformal conditions, proving that the explicit self-similar solution is stable with respect to perturbations of the blow-up point and initial data.

\medskip

In the super-conformal regime, ($p_c < p < p_S$),  the understanding of blow-up behavior is far from complete. 
Nevertheless, significant progress has been made: 
an upper bound for the blow-up rate was first proved in \cite{KSVma14}, and  later improved in \cite{HZdcds13, HZ2025JFA}. 
 In  parallel to these results, Donninger and Schörkhuber \cite{DSdpde12, DStams14, DScmp16, DSaihp17, DDuke} employed spectral-theoretic techniques to rigorously demonstrate the stability of ODE-type blow-up solutions within the backward light cone under small perturbations of the initial data.

\medskip

In the context of critical and supercritical   case  ($p \ge p_S$), a distinct phenomenon known as Type-II blow-up can occur. This phenomenon is distinguished by a blow-up rate that exceeds the  scaling of the corresponding ordinary differential equation. 
For a thorough overview of the mathematical theory and recent developments regarding Type-II blow-up, readers can consult the extensive surveys provided in 
\cite{Collot2018,DKM2012, Krieger2020, KS2014, KST2009}
and the literature cited within those works.

\medskip

In a more general context, we introduced in \cite{HZjhde12,HZnonl12} 
a perturbative framework to determine the blow-up rate for the Klein–Gordon equation and, more generally, for nonlinear wave equations of the form
\begin{equation}\label{NLWP}
\partial_t^2 u =\Delta u+|u|^{p-1}u+f(u)+g(\partial_t u ),\,\,\,(x,t)\in \R^n \times [0,T),
 \end{equation}
under the assumptions  $|f(u)|\leq C(1+|u|^q)$ and $|g(v)|\leq C(1+|v|)$,   for   $1<q<p\le  p_c$. 
In fact, we proved a similar result to $\eqref{mzmz}$,
valid in the subconformal  and conformal case.
  Let us also mention that in  \cite{H1, omar1, omar2}, the
results obtained in  \cite{HZjhde12,HZnonl12}   were extended  to the  strongly perturbed equation  (\ref{NLWP})  with  $|f(u)|\leq C\big[1+|u|^p\ln^{-a}(u^2+2)\big]$,  for some $a > 1$, 
while the same assumption on $g$ was retained.

\medskip

The validity of these results relies on the fact that, although \eqref{NLWP} is not invariant under the scaling $u_\lambda(x,t) = \lambda^{\frac{2}{p-1}}u(\lambda x, \lambda t),$ it remains asymptotically equivalent to the pure power equation with nonlinearity  $|u|^{p-1}u$. Passing 
to similarity variables,  the  equation can be viewed as an asymptotically autonomous problem, namely as a perturbation of the scale-invariant case. 
In particular, the leading-order dynamics are governed by the ordinary differential equation $u''=u^p$, which yields the estimate \eqref{mzmz}, while the perturbative terms remain lower order.
  It is important to note that, in the conformal regime, the dissipation degenerates at the boundary even in the unperturbed case, and this property persists under perturbations. 
As a consequence, the conformal case is more delicate and requires refined arguments. 
In particular, the construction of an appropriate Lyapunov functional necessitates additional analysis (see \cite{HZnonl12,H1,omar2}).

 \bigskip

The blow-up dynamics associated to equation \eqref{gen} differ from those of the classical equations \eqref{NLW} and \eqref{NLWP}, since the asymptotic behavior of the nonlinearity at infinity is no longer governed by a pure power.
  Accordingly, in the subconformal and conformal regimes, it is natural to  expect that the blow-up rate is dictated by the solution of the  ordinary differential equation:
\begin{equation}\label{v}
v_T'' (t)=  |v_T (t)|^{p-1}v_T(t) \ln^{a}\big(v_T^2(t)  +2\big), \quad v(T)=\infty.
\end{equation}
Clearly, as $t \to T$, the solution satisfies 
 \begin{equation*}\label{equivv}
 v_T(t) \sim \kappa_{a} \psi_T(t), \text{ as } t \to\ T,\quad \textrm{ where}\quad  \kappa_{a} =  \left(\frac{2^{1-2a}(p+1)}{(p-1)^{2-a}} \right)^{\frac1{p-1}},
 \end{equation*}
and
 \begin{equation}\label{psi}
 \psi_T(t)=(T - t)^{-\frac{2}{p-1}} (-\ln (T - t))^{-\frac{a}{p-1}},
 \end{equation}
(see Lemma A.2  in \cite{HZjmaa2020} for the  proof of \eqref{psi}).

\bigskip

In the subconformal regime, we previously investigated the blow-up dynamics for the nonlinear wave equation \eqref{gen} (see \cite{HZjmaa2020, HZ2022}). In those studies, we established an analogue to the sharp estimate \eqref{mzmz}, demonstrating the persistence of Type I blow-up. This classification was further extended by Roy and Zaag \cite{RZ} to include nonlinearities with double-logarithmic perturbations of the form $|u|^{p-1}u \ln^a (\ln(10 + u^2))$ for $a \in \mathbb{R}$.

\medskip

Our current goal is to demonstrate that the blow-up rate in the conformal regime is still  Type I, if $a<0$. The central element of our proof is the construction of a Lyapunov functional in similarity variables.
However, the strategies we employed for the subconformal regime in \cite{HZjmaa2020, HZ2022}  fail in this setting. 
Consequently, the construction requires novel techniques to overcome the specific  difficulties inherent to the conformal case.


\subsection{Outline of the proof}
We return to the equation under consideration, namely \eqref{gen}, under the condition \eqref{pc}. 
For this purpose, we introduce the following similarity variables, defined for any
$x_0 \in \R^n$ and $T_0$ such that $0<T_0\le T(x_0)$, by
\begin{equation}\label{scaling}
y=\frac{x-x_0}{T_0-t}, \quad s=-\ln(T_0-t), \quad
u(x,t)=\psi_{T_0}(t)\, w_{x_0,T_0}(y,s).
\end{equation}
From (\ref{gen}), the  function $w_{x_0,T_0}$  (we write $w$ for
simplicity) satisfies the following equation for all $y\in B$ and
$s\ge  -\ln T_0$,
where $B\equiv B(0,1)$ stands for
the unit ball in $\R^n$ and throughout the paper:
\begin{align}\label{A}
\partial_{s}^2w&=\frac{1}{\rho}\div(\rho \grad w-\rho(y\cdot \grad w)y)
-\frac{2p_c+2}{(p_c-1)^2}w+\gamma(s)w\nonumber\\
&-\Big(\frac{p_c+3}{p_c-1}+2\alpha (s)\Big)\partial_s w
-2y\cdot\grad \partial_{s}w+\frac{1}{s^a}   |w|^{p_c-1}w\ln^a(\phi^2w^2+2),
\end{align}
with 
\begin{equation}\label{defrho}
\rho:=\rho (y,s)=(1-|y|^2)^{\alpha(s)},\end{equation}
\begin{equation}\label{alpha}
\alpha (s):=-\frac{a}{(p_c-1)s}>0,
\end{equation}
\begin{equation}
   \gamma(s):=\frac{a(p_c+5)}{(p_c-1)^2s}-\frac{a(p_c+a-1)}{(p_c-1)^2s^2},\label{defgamma}
\end{equation}
and 
\begin{equation}
\phi:=\ps = e^{\frac{2s}{p_c-1}}s^{-\frac{a}{p_c-1}}.\label{defphi}
\end{equation}
In the new set of variables $(y,s),$ the behavior of $u$ as $t \rightarrow T_0$
is equivalent to the behavior of $w$ as $s \rightarrow +\infty$.  Moreover, if $T_0=T(x_0)$, then we simply write $w_{x_0}$
instead of $w_{x_0,T(x_0)}$.

\bigskip

Throughout this paper, unless otherwise specified, we let $u$ be a solution to \eqref{gen} with blow-up graph $\Gamma = \{ (x, T(x)) : x \in \mathbb{R}^n \}$, and we assume $x_0$ is a non-characteristic point. By exploiting the time-translation invariance of the equation\eqref{gen}, we can shift the initial time to $t_0(x_0)$ such that the remaining lifespan $T(x_0) - t_0(x_0)$ is arbitrarily small. Consequently, without loss of generality, we may assume 
\begin{equation}\label{ss00}
T(x_0) \ll 1,
\end{equation}
 which implies $-\ln T(x_0) \gg 1$. In this setting, the original initial data at $t=0$ are replaced by the corresponding solution values at $t=t_0(x_0)$.

\begin{nb}
Under the self-similar change of variables \eqref{scaling}, it would be natural to replace $\psi_{T_0}(t)$ with the exact solution $v_{T_0}(t)$ of \eqref{v}. However, since $v_{T_0}(t)$ lacks an explicit expression, its use would render the subsequent computations intractable. We therefore employ the explicit equivalent $\psi_{T_0}(t)$ defined in \eqref{psi}, This approach follows the strategy previously employed in \cite{HZjmaa2020, HZ2022, RZ}  and \cite{HZarma22}  within the context of parabolic equation.
\end{nb}

Our main strategy relies on the construction of a Lyapunov functional in similarity variables. 
As previously noted, the conformal case poses significantly greater challenges than the subconformal regime, requiring a novel approach. This is precisely the difficulty addressed in the present work: the methodologies developed in \cite{HZjmaa2020, HZ2022} are insufficient to resolve the specific technical obstacles inherent to the conformal setting. Below, we briefly outline the method used in \cite{HZ2022} and discuss where it fails in the conformal case, thereby motivating the current study.

\medskip

In the subconformal regime ($p < p_c$),  the approach in \cite{HZ2022} begins with the introduction of a functional $g_0(s)$ associated with equation \eqref{gen} in similarity variables. This functional satisfies the following differential inequality:
\begin{equation}\label{f0}
\frac{d}{ds}g_0(s) \le -\alpha_0 \int_{B} (\partial_s w)^2 \frac{\rho_{\alpha_0}(y)}{1 - |y|^2} dy + \frac{C}{s}g_0(s),
\end{equation}
which yields a rough polynomial bound in $s$ for the solution in space-time averages. Furthermore, by utilizing elementary estimates for the nonlinear term, the analysis can be reduced to the pure power case up to an $\epsilon$-perturbation. Consequently, by employing the Sobolev embedding of $H^1(\mathbb{R}^n)$ into $L^{2^*}(\mathbb{R}^n)$ for $n \geq 3$ (or $L^q(\mathbb{R}^n)$ for any $q \geq 2$ if $n=1,2$), we establish improved bounds that remove the  time averaging. By exploiting these refined polynomial bounds and the specific structure of the nonlinear term, we construct a second functional $g_1(s)$ satisfying:$$\frac{d}{ds}g_1(s) \leq -\alpha_0 \int_{B} (\partial_s w)^2 \frac{\rho_{\alpha_0}(y)}{1 - |y|^2} dy + \frac{C}{s^{\frac32}}g_1(s).$$
From this ordinary differential inequality, we naturally derive a Lyapunov functional. With this functional established, the interpolation strategy from our previous work can be adapted in a straightforward manner.

\bigskip

We recall that in the conformal case ($p=p_c$), the parameter $\alpha_0(p_c)$ vanishes, rendering the construction of a functional satisfying the differential inequality \eqref{f0} inapplicable. Consequently, addressing this regime necessitates novel insights beyond those applied in the subconformal setting. In what follows, we detail the strategy developed in this article to overcome these difficulties and extend the results of \cite{HZ2022} to the conformal case.

\medskip

The first step consists of introducing a functional associated with equation \eqref{gen}  in similarity variables 
 which satisfies the following  differential inequality:
 $$
\frac{d}{ds}h_0(s)\le -c\eta \ibint (\partial_{s}w)^2\frac{\rho_{\eta}(y)}{1-|y|^2}\y
+ C\eta h_0(s), \qquad \forall \eta>0.$$
By utilizing this functional, we derive an exponential bound (in $s$) for the $H^1 \times L^2(B)$ norm of the solution of \eqref{A} in terms of space-time averages. Specifically, we obtain the estimates \eqref{feb19eta} and \eqref{feb19etaNL} below. Subsequently, we derive a first functional from a Pohozaev identity by multiplying equation \eqref{A1} by $y \cdot \nabla w \rho_{\eta} \sqrt{1-|y|^2}$. A second functional is then obtained by multiplying equation \eqref{A1} by $w \frac{\rho_{\eta}}{\sqrt{1-|y|^2}}$. By combining these two functionals, we establish a new estimate that controls the time-averaged $L^{p+1}(B \times [s,s+1])$ and $L^{2}(B \times [s,s+1])$ norms of $w$ and $\nabla_{\theta}w$, respectively, both weighted by the singular term $\frac{\rho_{\eta}}{\sqrt{1-|y|^2}}$. These estimates  presented in Proposition \ref{prop123} are  one of the primary novelties of this paper and provide significantly better control near the boundary of the ball.

\medskip

By exploiting these   exponential bounds, we construct a second functional $h_1(s)$ satisfying:$$\frac{d}{ds}h_1(s) \leq -\alpha(s) \int_{B} (\partial_s w)^2 \frac{\rho}{1 - |y|^2} dy + \frac{C}{s}h_1(s).$$
Consequently, we establish a polynomial bound (in terms of $s$) for the $H^1 \times L^2(B)$ norm of the solution to \eqref{A}, expressed in space-time averages.
Using this result, we derive a Lyapunov functional for equation \eqref{A}. Specifically, we construct a third functional, $h_2(s)$, which satisfies the following  inequality:
$$\frac{d}{ds}h_2(s) \leq -\alpha(s) \int_{B} (\partial_s w)^2 \frac{\rho}{1 - |y|^2} dy.$$
Although this functional provides the foundation for our optimal estimate, the weakness of the damping term presents a novel complication. Indeed, by using the estimate related to  $h_2(s)$, we improve upon the previous control and conclude that the time average of the  $H^1 \times L^2(B)$ norm of the solution to \eqref{A} is bounded by $C s$.

\medskip

This complication arises because the problem resides in the logarithmically subconformal regime. This specific obstacle is particular  to our study and represents a gap in the literature. To overcome this difficulty, 
 by expressing the equation \eqref{A} as a perturbation  from the conformal regime and multiplying the new equation by $\partial_sw$ over the ball, we construct a new energy functional $E_0(w(s),s)$. 
 A deeper understanding of the dynamics allows us to establish more precise control over the damping term; specifically, we obtain the following uniform estimate:
 \begin{equation}\label{introbord}
\int_{s}^{s+1}\int_{\partial B}(\partial_{s} w)^2{\mathrm{d}}\sigma {\mathrm{d}}\tau\leq
 C. 
 \end{equation}
By virtue of this improvement, we conclude that the time average of the $H^1 \times L^2(B)$ norm of the solution to \eqref{A} is bounded. With the Lyapunov functional $h_2(s)$ and the boundedness of these space-time averages at hand, the adaptation of the interpolation strategy from our previous work follows straightforwardly.

\subsection{Main results}
The equation $\eqref{A}$ will be studied in the space 
\begin{equation*}
\H =\{(w_{1},w_{2})| \int_{B} \Big(w_{2}^2+|\nabla w_{1}|^2(1-|y|^2)+w_{1}^2\Big)\rho {\mathrm{d}}y< +\infty\}. 
\end{equation*}
To state our main result, we start by introducing the following functionals:
\begin{align}
  E(w(s),s)=&\iint \Big(\frac{1}{2}(\partial_{s}w)^2+\frac{1}{2}(|\grad  w|^2-(y\cdot \grad w)^2)+\big(\frac{p_c+1}{(p_c-1)^2}-\frac12  \gamma (s)\big)w^2\nonumber\\
  &-e^{-\frac{2(p_c+1)s}{p_c-1}}s^{\frac{2a}{p_c-1}}   f(\phi w)\Big)\rho \y,\label{Ealpha}\\
J(w(s),s)=&-\frac{1}{s\sqrt{s}}\int_{B} w\partial_{s}w \rho \y +\frac{n }{2s\sqrt{s}}\int_{ B}w^2\rho  {\mathrm{d}}y,\label{Jalpha}\\
G(w(s),s)=&E(w(s),s)+J(w(s),s)\label{Lalpha},
\end{align}
where
\begin{equation}\label{defF}
 f(u)=\int_0^u|v|^{p_c-1}v \ln^{{a}}(v^2  + 2)\v.
\end{equation}
Moreover,    we define the functional
\begin{equation}\label{calLalpha}
L(w(s),s)=\exp\Big(\frac{p_c+3}{\sqrt{s}}\Big) G(w(s),s)
+\frac{\theta_1}{ s},
\end{equation}
where $\theta_1$ is a sufficiently large constant that will be determined later.

\medskip

We will show that
the functional  $L(w(s),s)$  has a decreasing property, provided
that $s$ is large enough. Clearly, by \eqref{Jalpha}, and \eqref{Lalpha}, the functional 
$L(w(s),s)$ is a small
perturbation of the “natural” energy $
  E(w(s),s)$.

\medskip

Here is  the statement of  our main theorem in this paper.
 \begin{thm}\label{t1}
 Let $u$ be a solution to \eqref{gen} with blow-up graph $\Gamma := \{ (x, T(x)) : x \in \R^n \}$.  Assume \eqref{pc}, $a<0$, and let $x_0$ be a non-characteristic point. After a translation in time we assume $T(x_0)\le T_*$, where $T_*\in(0,1)$ is a constant depending only on $n$ and $a$.
Then,  for all $T_0 \in (0, T(x_0)]$ and $s \ge 6 - \ln(T_0 )$, the following inequality holds:
\begin{equation*}
L(w(s+1),s+1)  -L(w(s),s) 
 \le
\frac{a(n-1)}{4(s+1)} \exp\Big(\frac{p_c+3}{\sqrt{s+1}}\Big)\int_{s}^{s+1}\! \iint \frac{(\partial_{s}w)^2 \rho}{1-|y|^2}\y{\mathrm{d}}\tau,
 \end{equation*}
where  $w=w_{x_0,T_0}$ is defined in \eqref{scaling}.
Moreover, 
we have 
 \begin{equation}\label{posi1D}
L(w(s),s)\geq 0,  \qquad \forall s \geq 6-\ln( T_0). 
\end{equation}
\end{thm}

\medskip

\begin{nb}
As established previously, owing to the time-translation invariance of \eqref{gen}, we may shift the initial  time  at any non-characteristic point $x_0$ within the backward light cone such that the residual lifespan, $T(x_0) - t_0(x_0)$, is arbitrarily small. Consequently, without loss of generality, we assume \eqref{ss00}, which implies $-\ln T(x_0) \gg 1$. In this setting, the solution values at $t=t_0(x_0)$ serve as the effective initial data for our analysis.
\end{nb}
\begin{nb}
It should be noted that our method is specifically adapted to the case $a < 0$ and is not applicable when $a > 0$.
This limitation is consistent with the shift in the problem's structure, which becomes ``log- superconformal'' in the $a>0$ regime.
\end{nb}

\begin{nb}
As explained before, in the conformal case with  pure nonlinearity, 
the dissipation of the associated Lyapunov functional is supported on the boundary. In contrast, Theorem \ref{t1} demonstrates that in our setting, dissipation occurs over the entire ball, suggesting that the problem is subconformal. Furthermore, this damping is preceded   by an equivalent to the term  $c/s$ for large $s$, characterizing it as a weak damping. Consequently, this problem presents a  case not previously seen in the literature, which can be described  and  denoted as ``log-subconformal''.
\end{nb}
\begin{nb}
Since the existence of a Lyapunov functional in similarity variables constitutes the most critical and delicate step of our analysis, we present this result as Theorem \ref{t1} to emphasize its foundational role in this work.
\end{nb}
\begin{nb}
We remark that our method is not applicable to characteristic points. The construction of the Lyapunov functional in similarity variables relies on a covering argument that fails in the characteristic case. Consequently, it remains an open question whether the results of Theorem \ref{t1} extend to such points $x_0$.
\end{nb}

\medskip

As previously noted, the existence of the Lyapunov functional $L(w(s),s)$, combined with a blow-up criterion for equation \eqref{A}, constitutes a crucial step in deriving the blow-up rate for equation \eqref{gen}. However, a significant challenge arises from the weak nature of the damping, as the dissipation term is preceded  by the term $\frac{a(n-1)}{4(s+1)} \exp\Big(\frac{p_c+3}{\sqrt{s+1}}\Big)=\mathcal{O}(\frac1{s}),$
 we derive the following   estimate: 
\begin{equation*}
\int_{s}^{s+1}\! \big( \|w(\tau )\|^2_{H^1(B)}+\| \partial_{s}w(\tau )\|_{L^2(B)}^2\big) \t \leq C s.
\end{equation*}
 To address this difficulty, 
we begin by expressing the equation \eqref{A} as a perturbation  of the conformal regime, namely the following form:
\begin{align}\label{Ac}
\partial_{s}^2w&=\div( \grad w-(y\cdot \grad w)y)-2\alpha(s)y\cdot \grad w
-\big(\frac{2p_c+2}{(p_c-1)^2}-\gamma(s)\big)w\\
&-\big(\frac{p_c+3}{p_c-1}+2\alpha(s)\big)\partial_s w
-2y\cdot\grad \partial_{s}w
+\frac{1}{s^a}   |w|^{p_c-1}w\ln^a(\phi^2w^2+2),\ \   \forall s\ge -\ln T_0. \nonumber
\end{align} 
It is  natural to   introduce the functional associated with equation \eqref{Ac}:
\begin{align}
  E_0(w(s),s)=&\iint
             \Big(\frac{1}{2}(\partial_{s}w)^2+\frac{1}{2}(|\grad
             w|^2-(y\cdot \grad w)^2)+\big(\frac{p_c+1}{(p_c-1)^2}-\frac12  \gamma (s)\big)w^2\nonumber\\
  &-e^{-\frac{2(p_c+1)s}{p_c-1}}s^{\frac{2a}{p_c-1}}   f(\phi w)\Big) \y.\label{E}
\end{align}
By analyzing the dynamics of this energy and observing that the functional $E_0(w(s),s)$ is a small perturbation of the Lyapunov functional $L(w(s),s)$, 
we
establish the estimate \eqref{introbord}.
With this key estimate in hand, we are in position to state and prove our main theorem
\begin{thm}\label{t2}
{\bf {(Blow-up rate for equation \eqref{gen})}}.
Let $u$ be a solution to \eqref{gen} with blow-up graph $\Gamma:\{x\mapsto T(x)\}$. Suppose $x_0$ is a non-characteristic point satisfying conditions \eqref{pc} and \eqref{ss00}. Then,

i)
 For all
 $s\ge 13-\ln T(x_0)$,
\begin{equation*}
0<\varepsilon_0\le \|w_{x_0}(s)\|_{H^{1}(B)}+ \|\partial_s
w_{x_0}(s)\|_{L^{2}(B)} \le K,
\end{equation*}
where $w_{x_0}=w_{x_0,T(x_0)}$ is defined in (\ref{scaling}), for some $\varepsilon_0=\varepsilon_0 (n,a)$
and \\
$K=K(n, a, T(x_0),\|(u(0),\partial_tu(0))\|_{
H^{1}\times
L^{2}(B(x_0,\frac{T(x_0)}{\delta_0(x_0)}) )})$, 
where   $\delta_0(x_0)$ is given by  \eqref{nonchar}, and the initial data are those at the shifted time $t_0(x_0)$.\\
ii)  For all
  $t\in [(1-\frac1{e^{13}})T(x_0),T(x_0))$,  we have
\begin{align}
&&0<\varepsilon_0\le \frac1{\psi_{T(x_0)}(t)}\frac{\|u(t)\|_{L^2(B(x_0,{T(x_0)-t}))}}{ {(T(x_0)-t)^{\frac{n}2}}}\label{main1}\\
&&+ \frac{1}{\psi_{T(x_0)}(t)}\Big
(\frac{\|\partial_tu(t)\|_{L^2(B(x_0,{T(x_0)-t}))}}{
{(T(x_0)-t)^{{\frac{n}2}-1}}}+
 \frac{\| \grad u(t)\|_{L^2(B(x_0,{T(x_0)-t}))}}{ {(T(x_0)-t)^{{\frac{n}2}-1}}}\Big )\le K,\nonumber
\end{align}
where 
 $\varepsilon_0$ and $K$ are introduced in item i), and where  $\psi_{T(x_0)}(t)$ is defined in \eqref{psi}.
\end{thm}

%

\begin{nb}
As noted earlier, for a non-characteristic point $x_0$ in the backward light cone, where the function  $u$ is defined, the time-translation invariance of \eqref{gen} allows us to shift the initial time to $t_0(x_0)$ such that the remaining lifespan, $T(x_0) - t_0(x_0)$, is arbitrarily small. Thus, without loss of generality, we may assume \eqref{ss00}, which implies $-\ln T(x_0) \gg 1$. In this configuration, the solution values at $t=t_0(x_0)$ serve as the effective initial data for our analysis
\end{nb}

\begin{nb}
This   work  should be considered as  a starting point  for the understanding of 
the superconformal range ($p>p_c$) related to the blow-up rate of the  solution of  equation \eqref{NLW} below.
Hoping to extend the validity of our argument to the conformal case  ($p=p_c$)
 in some forthcoming work, we may see the case $a > 0$  of \eqref{gen} as a further
step in the understanding of blow-up dynamics in the superconformal case related to
equation \eqref{NLW} below.
\end{nb}

\begin{nb}
Since we crucially need a covering technique in the  argument of the construction of the Lyapunov functional,
our method also breaks down  in the case  of a characteristic point  and we are not able to obtain the sharp estimate
as in the unperturbed case \eqref{NLW}. 
\end{nb}
\begin{nb}
As in  \cite{MZimrn05}  in the pure power nonlinearity  case   \eqref{NLW}, the proof of
Theorem \ref{t2} relies on four ideas: the existence of a Lyapunov functional,  interpolation in Sobolev spaces, some
critical Gagliardo-Nirenberg estimates and
a covering technique adapted to the geometric shape of the blow-up surface.
As we said before,  the first point where  we construct  a Lyapunov functional in similarity variables 
 is far from being trivial 
and represent  a  crucial  step. Consequently, 
we have chosen to present our main contribution as Theorem \eqref{t1}  and we 
 write a detailed  proof.  However,   for the other three points, the adaptation of 
the proof of
\cite{MZimrn05} given in the  pure power nonlinearity  case \eqref{NLW} is straightforward   except for a key argument, where we
bound the nonlinear term  $e^{-\frac{2(p_c+1)s}{p_c-1}}s^{\frac{2a}{p_c-1}} f(\phi w)$.
Therefore, 
 in order to avoid unnecessary repetition, we prove this step and  kindly  refer to \cite{MZajm03, MZimrn05, MZma05, HZjhde12, HZnonl12,H1, omar1, omar2}  for the rest of the proof.
\end{nb}

\bigskip

Finally, we situate our work within the growing body of research focused on non-scale invariant regimes, where general non-homogeneous nonlinearities are studied in the subcritical case across elliptic and parabolic frameworks. Significant progress has been made in establishing universal profiles and blow-up rates for these models, particularly through the development of methodologies that do not rely on  scaling symmetries; for a detailed treatment, we refer to \cite{Souplet2023DCDS, QuittnerSouplet2025, ChabiSouplet2024}.\\

To our knowledge, this is the first treatment of the conformal exponent for a nonlinearity of the form $|u|^{p-1}u\ln^a(u^2+2)$. The conformal case differs from the subconformal one because the dissipation of the pure-power functional is supported on $\partial B$; here it is spread over $B$ but is only of size $O(1/s)$. We expect the same strategy to apply to the double-logarithmic nonlinearities of \cite{RZ} and to mildly space-dependent nonlinearities at $p=p_c$, but we do not pursue this here.

\medskip

Throughout this paper,
$C$  denotes a  generic positive constant
 depending only on $n$  and $a,$  which may vary from line to line.
 In addition, we  will use $K_1,K_2,K_3...$  as    positive constants
 depending only on $n, a, \delta_0(x_0)$  and initial data,   which
 may also vary from line to line.  We write $A(s)\sim B(s)$ to indicate 
$\displaystyle{\lim_{s\to \infty}\frac{A(s)}{B(s)}=1}.$

Notice that in the rest of this  paper,  we  define 
\begin{equation}\label{24nov1}
 \nabla_r w=\frac{y\cdot \nabla w}{|y|^2} y\quad \ {\textrm{and}}\quad  \nabla_{\theta}w=
\nabla w-
\frac{y\cdot \nabla w}{|y|^2} y.
\end{equation}
Given  \eqref{24nov1}, we can write $\nabla w=\nabla_rw+\nabla_{\theta}w$ and
we have  the identities 
\begin{equation}\label{12nov5}
|y|^2|\grad w|^2-(y\cdot \grad w)^2=|y|^2|\nabla_{\theta}w|^2,
\end{equation}
and
\begin{equation}\label{12nov6}
|\grad w|^2-(y \cdot\grad w)^2=|\nabla_{\theta}w|^2+(1-|y|^2)|\nabla_{r}w|^2.
\end{equation}

\bigskip

This paper is organized as follows: In Section \ref{section2}, we establish exponential bounds for the time average (in $s$) of the $H^1 \times L^2(B)$ norm of the solution $(w, \partial_s w)$. These estimates are subsequently refined in Section \ref{section3} to obtain nearly linear (arbitrarily small super-linear) bounds.
In Section \ref{section4}, leveraging this nearly linear control, we demonstrate that the functional $L(w(s),s)$ defined in \eqref{calLalpha} serves as a Lyapunov functional for equation \eqref{A}, thereby yielding Theorem \ref{t1}. This result constitutes a crucial step in deriving the blow-up rate for equation \eqref{A}.However, as previously noted, the logarithmic sub-conforming nature of the problem results in weak dissipation. Consequently, by utilizing the framework of the preceding theorem in section \ref{section5}, we first  improve the previous estimates  for the time-average of the $H^1\times L^2(B)$  norm only to linear bounds. Then, by expressing equation \eqref{A} as a perturbation of the conformal regime and constructing a new energy functional $E_0(w(s), s)$, we  refine the control over the damping term, establish the boundedness of space-time averages, and conclude with the proof of Theorem \ref{t2}, which characterizes the sharp Type I blow-up rate for the non-scaling invariant semilinear wave equation.

\section{Exponential 
bounds for the time average of the $H^{1}\times L^2$ norm  and the  non linear term of $w$ with 
singular weight for  solution of equation 
(\ref{A})}
\label{section2}

This section presents the statements and proofs of Proposition \ref{prop21} and Proposition \ref{prop123}. It is divided into two subsections, with each one dedicated to
prove one of the propositions.
\begin{itemize}
\item  The first subsection is devoted to the proof of Proposition \ref{prop21}, which relies on a set of classical energy estimates. These estimates are derived by multiplying equation \eqref{A} by $w(1-|y|^2)^{\eta}$ and $\partial_s w(1-|y|^2)^{\eta}$, then integrating over the unit ball for any $\eta \in (0,1)$. By combining these energy estimates, a Lyapunov functional is constructed for $\eta \in (0, \eta_1)$, where $\eta_1$ is sufficiently small. This functional provides an exponential bound for the time average of the $H^{1}$ norm, which concludes the proof of Proposition \ref{prop21}.

\item The second subsection is devoted to  two new  estimates
  following from a Pohozaev multiplier. More precisely, we multiply
  equation \eqref{A} by $y\cdot\grad w(1-|y|^2)^{\eta}$ and
  $w(1-|y|^2)^{-\frac12+\eta}$,  for any $\eta\in (0,1)$, then integrate over $B$, in order to get
  some  new identities.  By combining the above energy estimates with  the  exponential bound 
on the time average of the $H^{1}$ norm (which was established
in  Proposition \ref{prop21}), we deduce the proof of  Proposition \ref{prop123}.
This is one of the novelties of this paper.
\end{itemize}

\medskip
In this section and for the remainder of the paper, unless stated otherwise, let $u$ be a solution to \eqref{gen} with blow-up graph $\Gamma: \{x \mapsto T(x)\}$, and assume $x_0$ is a non-characteristic point, under the conditions \eqref{pc}, and \eqref{ss00}.
For $T_0 \in (0, T(x_0)]$, and $x \in \mathbb{R}^n$ satisfying $|x - x_0| \le \frac{T_0}{\delta_0(x_0)}$ (where $\delta_0(x_0)$ is defined by \eqref{nonchar}), we denote the rescaled variable $w_{x, T^*(x)}$ from \eqref{scaling} simply by $w$, where
 \begin{equation}\label{18dec1}
T^*(x)=T_0-\delta_0(x_0)|x-x_0|.
\end{equation}

\medskip
Let $\eta > 0$. We can express the  equation for $w$ as a perturbation of the subconformal dynamics, namely:
\begin{eqnarray}\label{A1}
\partial^2_{s}w&=&\frac{1}{\rho_{\eta}} div(\rho_{\eta}\nabla w-\rho_{\eta} (y \cdot \nabla w)y)+(2\eta-2\alpha(s))  y\cdot \nabla w-\frac{2(p_{c}+1)}{(p_{c}-1)^2}w+\gamma (s) w \nonumber \\
&&-\Big(\frac{p_c+3}{p_c-1}+2\alpha(s)\Big)\partial_s w
-2y\cdot \grad \partial_{s}w+ \frac{1}{s^a}   |w|^{p_c-1}w\ln^a(\phi^2w^2+2).\end{eqnarray}
where 
$ \rho_{\eta}:=(1-|y|^2)^{\eta},$
and where $\alpha(s)$, $\gamma(s)$, and $\psi(s)$ are given respectively by \eqref{alpha}, \eqref{defgamma}, and \eqref{defphi}.

\medskip

As explained before, the first   subsection is  devoted to deriving   
 an exponential  estimate  for  solution of \eqref{A} valid  for $x$ near $x_0$  a non
characteristic point.  
  More precisely,  the aim of this subsection   is the following:
\begin{prop}\label{prop21}
\noindent  Let $\eta_0>0$  be  sufficiently  small, and fix $\eta \in (0, \eta_0)$.
 Consider $u $   a solution of ({\ref{gen}}) with
blow-up graph $\Gamma:\{x\mapsto T(x)\}$ and  $x_0$  a non
characteristic point under the condition \eqref{ss00}. Then,     for all $T_0 \in (0,T(x_{0})]$,  $s\geq -\ln T_0,$   and $x\in \R^n,$ such that $|x-x_0|\le \frac{e^{-s}}{\delta_0(x_0)}$,  
we have
\begin{equation}\label{feb19eta}
\int_{s}^{s+1}\!\ibint \big( w^2+|\grad w|^2 +(\partial_{s}w)^2 \frac{\rho_{\eta}}{1-|y|^2}
\big)\y \t \leq K_1 e^{\frac{\eta(p_c+3)s} 2},
\end{equation}
\begin{equation}\label{feb19etaNL}
\frac{1}{s^a}\int_{s+1}^{s}\! \int_{B}   |w|^{p_c+1}\ln^a(\phi^2w^2+2){\mathrm{d}}y{\mathrm{d}}\tau
\leq K_1 e^{\frac{\eta(p_c+3)s} 2},
\end{equation}
where $w=w_{x,T^*(x)}$ is defined in \eqref{scaling},
 and $\delta_{0}(x_{0})$ defined in \eqref{nonchar}, and where  $K_1$ depends on $ n, a,  \eta,  \delta_{0}(x_{0})$, 
 $T(x_{0})$,  and
$\|(u(0),\partial_tu(0))\|_{H^{1}\times
L^{2}(B(x_0,\frac{T(x_0)}{\delta_0(x_0)}) )}$.
\end{prop}
For convenience, we introduce the following
\begin{equation}\label{NN}
{\cal {N}}_{\eta}(w(s))= \!\int_{B}\Big(|\nabla  w(s)|^2+(\partial_sw(s))^2+w^2(s)  + \frac{1}{s^a}  |w(s)|^{p_c+1}\ln^a(\phi^2w^2(s)+2)\Big) \p{\mathrm{d}}y.
\end{equation}

As we mentioned before, our first novelty in this paper lays in the following proposition:
 \begin{prop}\label{prop123}
  For all $T_0 \in (0,T(x_{0})]$, $s\geq 3-\ln (T_0),$
and   $x\in \R^n,$ such that $|x-x_0|\le \frac{e^{-s}}{\delta_0(x_0)}$,  
we have
\begin{align}\label{w1eta}
&\frac{1}{s^a}\int_s^{s+1} \int_{B}   |w|^{p_c+1}\ln^a(\phi^2w^2+2)\frac{\p}{\sqrt{1-|y|^2}}{\mathrm{d}}y{\mathrm{d}}\tau
 \le C\int_{s-2}^{s+3}{\cal {N}}_{\eta}(w(\tau)){\mathrm{d}}\tau+Ce^{-2s},
\end{align}
and
\begin{align}\label{wtangeta}
&\int_s^{s+1} 
\int_{B}   |\nabla_{\theta}w|^2\frac{\p}{\sqrt{1-|y|^2}}{\mathrm{d}}y{\mathrm{d}}\tau
 \le C\int_{s-3}^{s+4}{\cal {N}}_{\eta}(w(\tau)){\mathrm{d}}\tau,
\end{align}
where   the constant $C$ depends only on $n, a$, and $\eta$. 
\end{prop}
Clearly, by  combining    Proposition \ref{prop21} and Proposition \ref{prop123},
we deduce the following crucial estimates:
\begin{cors}\label{corol1}
For all $T_0 \in (0,T(x_{0})]$, $s\geq 3-\ln (T_0),$  
and  $x\in \R^n,$ such that $|x-x_0|\le \frac{e^{-s}}{\delta_0(x_0)}$,  
we have
\begin{equation}\label{14sep1}
\frac{1}{s^a}\int_s^{s+1} 
\int_{B}   |w|^{p_c+1}\ln^a(\phi^2w^2+2)\frac{\p}{\sqrt{1-|y|^2}}{\mathrm{d}}y{\mathrm{d}}\tau
 \leq K_2 e^{\frac{\eta(p_c+3)s} 2},
\end{equation}
and
\begin{align}\label{14sep2}
&\int_s^{s+1} 
\int_{B}   |\nabla_{\theta}w|^2\frac{\p}{\sqrt{1-|y|^2}}{\mathrm{d}}y{\mathrm{d}}\tau \le  K_2 e^{\frac{\eta(p_c+3)s} 2}.
\end{align}
where  $K_2$ depends on $  a, n, \eta,  \delta_{0}(x_{0})$, 
 $T(x_{0})$, and
$\|(u(0),\partial_tu(0))\|_{H^{1}\times
L^{2}(B(x_0,\frac{T(x_0)}{\delta_0(x_0)}) )}$.
\end{cors}

\subsection{An exponential   bound for the solution of equation   (\ref{A1})}
As previously noted, we first introduce two classical energy estimates. These are obtained by multiplying equation \eqref{A1} by $\partial_sw\rho_{\eta}$ and $ w\rho_{ \eta}$, respectively, and then integrating. Indeed,
 we first introduce the following functionals:
\begin{eqnarray}
E_{\eta}(w(s),s)&=&\ibint \Big(\frac{1}{2}(\partial_{s}w)^2+\frac{1}{2}|\grad w|^2-\frac12(y\cdot \grad w)^2
+\frac{p_c+1}{(p_c-1)^2}w^2\Big) \rho_{\eta}\y\nonumber \\
&&- e^{-\frac{2(p_c+1)s}{p_c-1}}s^{\frac{2a}{p_c-1}} \ibint  f(\phi w) \rho_{\eta}\y-\frac{\gamma (s)}2\ibint w^2 \rho_{\eta}\y,\qquad  \label{Eeta}\\
I_{\eta}(w(s),s)&=&-(\eta-\alpha (s))\int_{B} w\partial_{s}w \rho_{\eta} dy+\frac{(n-2\alpha(s))(\eta-\alpha(s))}{2}\int_{ B}w^2\rho_{\eta} {\mathrm{d}}y ,\qquad \qquad \label{Ieta}\\
H_{\eta}(w(s),s)&=&E_{\eta}(w(s),s)+I_{\eta}(w(s),s),\label{Heta} \\
{\cal E}_{\eta}(w(s),s)&=&H_{\eta}(w(s),s)e^{-\frac{\eta(p_{c}+3)s}{2}}+\theta_2 e^{-\frac{\eta(p_{c}+3)s}{2}},\label{Geta}
\end{eqnarray}
where $f$  is defined by \eqref{defF}, and 
where $\theta_2=\theta_2(\eta)$ is a sufficiently large constant that will be determined later.

\medskip

In this subsection, we first demonstrate that ${\cal E}_{\eta}(w(s),s)$ is non-increasing in time. This result yields a rough (specifically, exponentially fast) estimate for ${ E}_{\eta}(w(s),s)$, as well as a bound on the time average of $\|(w,\partial_{s} w)\|_{H^{1}(B)\times L^2(B)}$. More precisely, we shall establish the following proposition:
\begin{prop}\label{prop2.2} There exists $\lambda_{0}>0$ such that ${\cal E}_{\eta}(w(s),s)$ defined in \eqref{Geta}
 satisfies for all $s'\ge s\geq  -\ln (T^*(x))$, 
\begin{eqnarray} 
{\cal E}_{\eta}(w(s'),s')-{\cal E}_{\eta}(w(s),s)&\leq &-\lambda_{0}\int_{s}^{s'}e^{\frac{-\eta(p_{c}+3)\tau}{2}}\int_{B}(\partial_{s} w)^2\frac{\rho_{\eta}}{1-|y|^2}{\mathrm{d}}y{\mathrm{d}}\tau\nonumber \\
&&-\lambda_{0}\int_{s}^{s'}e^{\frac{-\eta(p_{c}+3)\tau}{2}}\int_{B}w^2\rho_{\eta} {\mathrm{d}}y{\mathrm{d}}\tau\label{pGeta}\\
&&-\lambda_{0}\int_{s}^{s'}e^{\frac{-\eta(p_{c}+3)\tau}{2}}\frac1{\tau^a}\int_{B}|w|^{p_{c}+1}\ln^a(\phi^2w^2+2)\rho_{\eta} {\mathrm{d}}y{\mathrm{d}}\tau\no\\
&&-\lambda_{0}\int_{s}^{s'}e^{\frac{-\eta(p_{c}+3)\tau}{2}}\int_{B}(|\nabla w|^2-(y\cdot \nabla u)^2)\rho_{\eta} {\mathrm{d}}y{\mathrm{d}}\tau ,\no
\end{eqnarray} 
where $w=w_{x,T^*(x)}$ is defined in \eqref{scaling}.
Moreover, there exist $\eta_0>0$, 
 such that 
\begin{equation}\label{positiveeta}
{\cal E}_{\eta}(w(s),s)\geq 0, \quad 
\forall s\geq -\ln (T^*(x)),  \qquad \forall \eta \in(0, \eta_0).
\end{equation}
 \end{prop}

\bigskip
 
We begin with  the time derivative of $E_{\eta}(w(s),s)$ in the following lemma.
\begin{lem}\label{lem21}
For all $s\geq -\ln (T^*(x))$, we have
 \begin{eqnarray}\label{mai4}
\frac{d}{ds}(E_{\eta}(w(s),s))&=&
-2\eta\int_{B}(\partial_{s} w)^2\frac{\rho_{\eta}}{1-|y|^2}{\mathrm{d}}y +2(\eta-\alpha(s))\int_{B}\partial_{s} w 
 y\cdot \nabla w \rho_{\eta}{\mathrm{d}}y\no \\
&&+2(\eta- \alpha(s)) \ibint (\partial_{s}w)^2\rho_{\eta}\y+\Sigma_{0}(s),
\end{eqnarray}
where $\Sigma_{0}(s)$ satisfies
    \begin{equation}\label{mai5}
\Sigma_{0}(s)\leq  
\frac{C}{s^{a+1}}\ibint |w|^{p_c+1}\ln^a(\phi^2w^2+2)\rho_{\eta} \y+  
\frac{C}{s^{2}}\int_{B}w^{2}\rho_{\eta} {\mathrm{d}}y+C e^{-2s}.
\end{equation}
\end{lem}
\bigskip

{\it Proof}: 
Multiplying \eqref{A1} by $\partial_{s} w\rho_{\eta}$ and integrating by parts over $B$, we obtain \eqref{mai4}, where
\begin{align}\label{mail1}
\Sigma_0(s)=&
\underbrace{\big(\frac{2p_c+2}{p_c-1}-\frac{2a}{(p_c-1)s}\big)
e^{-\frac{2(p_c+1)s}{p_c-1}}s^{\frac{2a}{p_c-1}}\ibint\big( f_1(\phi w)+f_2(\phi w)\big)
\rho_{\eta} \y}_{\Sigma^1_{0}(s)}\no\\
&+\frac{a}{(p_c+1)s^{a+1}}\ibint |w|^{p_c+1}\ln^a(\phi^2w^2+2)
\rho_{\eta}
\y \underbrace{-{\frac{\gamma' (s)}2\ibint w^2\rho_{\eta}\y}}_{\Sigma_0^2(s)},
 \end{align}
where 
\begin{equation}
f_1(x)= -\frac{ 2a} {(p_c+1)^2}{| x|^{p_c+1}}\ln^{{a-1}}(x^2+2 ),\label{defF2}
\end{equation}
and 
\begin{equation}\label{defF123}
f_2(x)=f(x)
-\frac{ 1} {p_c+1}{| x|^{p_c+1}}\ln^{a}(x^2+2 )-
f_1(x).
\end{equation}
According to \eqref{equiv2} and \eqref{equiv3},  we have
\begin{equation}\label{mais01}
\Sigma_{0}^{1}(s) \le \frac{C}{s^{a+1}} \int_{B} |w|^{p_c+1} \ln^a(\phi^2w^2+2) \rho_{\eta} dy + C e^{-2s}.
\end{equation}
Furthermore, the definition of $\gamma(s)$ in \eqref{defgamma} yields $|\gamma'(s)| \le C s^{-2}$, which implies
\begin{equation}\label{mai2025}
\Sigma_{0}^{2}(s) \le \frac{C}{s^2} \int_{B} w^2 \rho_{\eta}  dy.
\end{equation}
Combining \eqref{mail1}, \eqref{mais01}, and \eqref{mai2025}, we obtain \eqref{mai5}, which completes the proof of Lemma \ref{lem21}.
\Box
       
 \medskip

 We are now going to prove the following estimate for the functional $I_{\eta}(w(s),s)$.
\begin{lem}\label{lem2.4}
For all  $s\geq -\ln (T^*(x))$, we have
 \begin{eqnarray}\label{6b}
\frac{d}{ds}(I_{\eta}(w(s),s))&=&\hspace{-0,3cm}-(\eta-\alpha (s))\int_{B}(\partial_{s} w)^2\rho_{\eta}{\mathrm{d}}y+(\eta-\alpha (s))
\int_{B}(|\nabla w|^2-(y \cdot \nabla w)^2)\rho_{\eta} {\mathrm{d}}y\nonumber \\
&&-2(\eta-\alpha (s))\int_{B}\partial_{s} w y \cdot \nabla w\rho_{\eta}{\mathrm{d}}y+2\eta\Big(\frac{p_{c}+1}{(p_{c}-1)^2}+\frac{n\eta}2\Big)\int_{B}w^2\rho_{\eta} {\mathrm{d}}y\no \\
&&+4\eta (\eta-\alpha(s))\int_{B}w \partial_{s}w\frac{|y|^2\rho_{\eta}}{1-|y|^2}{\mathrm{d}}y-2\eta (\eta-\alpha(s))^2\int_{B}w^2\frac{|y|^2\rho_{\eta}}{1-|y|^2}{\mathrm{d}}y\nonumber \\
&&- \frac{\eta}{s^a} \int_{B} |w|^{p_c+1} \ln^a(\phi^2w^2+2) \rho_{\eta} 
{\mathrm{d}}y+\Sigma_{1}(s),
\end{eqnarray}
where $\Sigma_{1}(s)$ satisfies
  \begin{equation}\label{7}
\Sigma_{1}(s)\leq  \frac{C}{s}\int_{B} \Big(
w^2+(\partial_sw)^2+\frac{1}{s^{a}} |w|^{p_c+1}\ln^a(\phi^2w^2+2)\Big)\rho_{\eta} \y.
\end{equation}
\end{lem}
{\it Proof}: Note that $I_{\eta}(w(s),s)$ is a differentiable function, by using equation $\eqref{A1}$ and integrating by part, for all $s\geq -\ln( T^*(x))$ we have  \eqref{6b}, where 
\begin{eqnarray}\label{20b}
\Sigma_1(s)&=&\underbrace{\Big( n\alpha^2(s)-2n\eta \alpha(s)-\eta \gamma(s) 
+\alpha (s)\gamma(s) -\alpha(s)\frac{2p_{c}+2}{(p_{c}-1)^2}\Big)\int_{B}w^2\rho_{\eta} {\mathrm{d}}y}_{\Sigma_{1}^{1}(s)}\no \\
&&+\underbrace{\alpha' (s)\int_{B} w\partial_{s}w \rho_{\eta} dy-\frac{\alpha'(s) (n+2\eta-4\alpha (s))}{2}\int_{ B}w^2\rho_{\eta} {\mathrm{d}}y}_{\Sigma_{1}^{2}(s)}\no\\
&&+
\frac{\alpha(s)}{s^a} \int_{B} |w|^{p_c+1} \ln^a(\phi^2w^2+2) \rho_{\eta} 
{\mathrm{d}}y.
\end{eqnarray}
By applying the bound $|\alpha(s)| + |\alpha'(s)| + |\gamma(s)| \le \frac{C}{s}$ and using  inequality $ab \le a^2 + b^2$, we infer
\begin{equation}\label{18b}
\Sigma_{1}^{1}(s) + \Sigma_{1}^{2}(s) \leq \frac{C}{s} \int_{B}\big( (\partial_s w)^2 + w^2 \big)\rho_{\eta} \, dy.
\end{equation}
Consequently, by combining \eqref{18b} and \eqref{20b}, we  obtain the estimate \eqref{7}. This completes the proof of Lemma \ref{lem2.4}.
\Box

\bigskip

From Lemmas \ref{lem21} and  \ref{lem2.4}, we are in position to prove Proposition \ref{prop2.2}.
\\

{\it Proof of Proposition  \ref{prop2.2}}:
Now, we combine Lemmas 2.2 and 2.3,  the definition of the expression of $H_{\eta}(w(s),s)$ given by \eqref{Heta} and some straightforward computations, we infer
 \begin{eqnarray}\label{mai4bn}
\frac{d}{ds}(H_{\eta}(w(s),s))&=&
-2\eta\int_{B}(\partial_{s} w)^2\frac{\rho_{\eta}}{1-|y|^2}{\mathrm{d}}y  +\frac{\eta(p_c+3)}{2}H_{\eta}(w(s),s)\no \\
&&-\big(\frac{\eta (p_c-1)}4- \alpha(s)\big) \ibint (\partial_{s}w)^2\rho_{\eta}\y\no\\
&&-\big(\frac{\eta (p_c-1)}4- \alpha(s)\big)
\int_{B}(|\nabla w|^2-(y \cdot \nabla w)^2)\rho_{\eta} {\mathrm{d}}y\nonumber \\
&&-\big(\frac{\eta(p_{c}+1)}{2(p_{c}-1)}+\frac{n\eta^2(p_c-1)}4\Big)\int_{B}w^2\rho_{\eta} {\mathrm{d}}y\no \\
&& -\frac{\eta (p_c-1)}{(2p_c+2)s^{a}}\ibint |w|^{p_c+1}\ln^a(\phi^2w^2+2)\rho_{\eta} \y\no\\
&&+\underbrace{4\eta (\eta-\alpha(s))\int_{B}w \partial_{s}w\frac{|y|^2\rho_{\eta}}{1-|y|^2}{\mathrm{d}}y-2\eta (\eta-\alpha(s))^2\int_{B}w^2\frac{|y|^2\rho_{\eta}}{1-|y|^2}{\mathrm{d}}y}_{\Sigma_2^1(s)}\nonumber \\
&&+\underbrace{\frac{\eta(p_c+3)}{2}
e^{-\frac{2(p_c+1)s}{p_c-1}}s^{\frac{2a}{p_c-1}}  \ibint  \Big(f_1(\phi w) +f_2(\phi w)\Big)\rho_{\eta}\y}_{\Sigma_2^2(s)}\no\\
&&+\underbrace{\frac{\eta (p_c+3) ( 2\eta \alpha(s)+n\alpha(s)+\gamma (s)-2\alpha^2(s))}{4}
\int_{ B}w^2\rho_{\eta} {\mathrm{d}}y}_{\Sigma_2^3(s)}\no\\
&&+\underbrace{\frac{\eta (p_c+3)(\eta-\alpha (s))}2\int_{B} w\partial_{s}w \rho_{\eta} dy}_{\Sigma_2^4(s)} +\Sigma^5_{2}(s),\qquad \qquad
\end{eqnarray}
where $\Sigma^5_{2}(s)=\Sigma_{0}(s)+\Sigma_{1}(s).$

The next step is to estimate  each of the last five terms.
First, by  using 
 the  Young's inequality, we write for all $\mu\in (0,1)$   
$$\Sigma_{2}^{1}(s)\leq 2\eta(1-\mu)\int_{B}(\partial_{s} w)^2\frac{|y|^2\rho_{\eta}}{1-|y|^2}{\mathrm{d}}y+\frac{2\eta(\eta-\alpha(s))^2\mu }{1-\mu}\int_{B}w^2\frac{|y|^{2}\rho_{\eta}}{1-|y|^2}{\mathrm{d}}y.$$

Let us recall from  \cite{omar2}. the following Hardy type inequality
\begin{equation}\label{AHardy}
 \int_{B}h^2\frac{|y|^2\rho_{\eta}}{1-|y|^2}{\mathrm{d}}y\leq \frac{1}{\eta^2}\int_{B}|\nabla h|^2(1-|y|^2)\rho_{\eta} {\mathrm{d}}y+\frac{n}{\eta}\int_{B} h^2\rho_{\eta} {\mathrm{d}}y.
\end{equation}
We apply the Hardy type inequality  \eqref{AHardy} to the second term  and we choose $\mu= \frac{p_{c}-1}{p_{c}+15}$, we conclude that
\begin{eqnarray}\label{16b}
\Sigma_{2}^{1}(s)&\leq &\frac{32\eta}{p_{c}+15}\int_{B}(\partial_{s} w)^2\frac{|y|^2\rho_{\eta}}{1-|y|^2}{\mathrm{d}}y +\frac{(\eta-\alpha(s))^2n(p_{c}-1)}{8}\int_{B} w^2\rho_{\eta}{\mathrm{d}}y\nonumber \\
&&+\frac{(\eta-\alpha(s))^2(p_{c}-1)}{8\eta}\int_{B}|\nabla w|^2(1-|y|^2)\rho_{\eta} {\mathrm{d}}y.
\end{eqnarray}
Note from  \eqref{equiv2} and \eqref{equiv3}, we write, 
\begin{equation}\label{Juin1}
\Sigma_{2}^{2}(s)\le
\frac{C}{s^{a+1}}\ibint |w|^{p_c+1}\ln^a(\phi^2w^2+2)\rho_{\eta} \y+  C e^{-2s}.
\end{equation}
Furthermore, by exploiting  the  fact that  $|\alpha(s)|+|\gamma(s)|\le \frac{C}{s},$ we get 
\begin{equation}\label{18bt}
\Sigma_{2}^{3}(s)+\Sigma_{2}^{4}(s)\leq \frac{32\eta}{p_{c}+15}\int_{B} (\partial_sw)^2\rho_{\eta} {\mathrm{d}}y+C (\eta^2+\frac1{s})\int_{B} w^2\rho_{\eta} {\mathrm{d}}y.
\end{equation}
By Lemma \ref{lemm:esth} and \eqref{ss00}, we have $-\ln T^*(x) \ge -\ln T(x_0) \ge \tilde{S}_0$, which leads us to conclude
\begin{equation}\label{17j}
C\eta^2\int_{B} w^2\rho_{\eta} {\mathrm{d}}y
 \le \frac{\eta(p_c-1)}{(4p_c+4)s^a}\ibint |w|^{p_c+1}\ln^a(\phi^2w^2+2)\rho_{\eta} \y+
 C(\eta), \qquad  
 \forall s \ge -\ln T^*(x).
\end{equation}
Additionally, using the definition of $\Sigma^5_{2}(s)$ and inequalities \eqref{mai5} and \eqref{7}, we find
 \begin{equation}\label{bnbn}
 \Sigma^5_{2}(s)\leq  \frac{C}{s}\int_{B} \Big(
w^2+(\partial_sw)^2+\frac{1}{s^{a}} |w|^{p_c+1}\ln^a(\phi^2w^2+2)\Big)\rho_{\eta} \y.
\end{equation}
Hence, by exploiting \eqref{mai4bn} and inequalities \eqref{16b}, \eqref{Juin1}, \eqref{18b}, \eqref{17j}, and \eqref{bnbn}, we infer
 \begin{eqnarray}\label{Hmai4}
\frac{d}{ds}(H_{\eta}(w(s),s))&\le&\frac{\eta(p_c+3)}{2}H_{\eta}(w(s),s)
-\frac{\eta (p_c-1)}{p_c+15}\int_{B}(\partial_{s} w)^2\frac{\rho_{\eta}}{1-|y|^2}{\mathrm{d}}y \no \\
&&-\big(\frac{\eta(p_c-1)}4-\frac{C_0}{s}\big) \ibint (\partial_{s}w)^2\rho_{\eta}\y\no\\
&&-\big(\frac{\eta(p_c-1)}8-\frac{C_0}{s}\big) \ibint (|\nabla w|^2-(y \cdot \nabla w)^2)\rho_{\eta}\y\no\\
&&-\big(\frac{\eta(p_{c}+1)}{2(p_{c}-1)}+\frac{n\eta^2(p_c+1)}2-\frac{C_0}{s}\Big)\int_{B}w^2\rho_{\eta} {\mathrm{d}}y\no\\
&&-(\frac{\eta (p_c-1)}{4p_c+4}-\frac{C_0}{s}) 
\frac1{s^{a}}
\int_{B}   |w|^{p_c+1}\ln^a(\phi^2w^2+2)\rho_{\eta}{\mathrm{d}}y
+C_0.\qquad\qquad
\end{eqnarray}
In view of $T^*(x) \le T(x_0)$ and the assumption in \eqref{ss00} that $T(x_0)$ is sufficiently small, $-\ln(T^*(x))$ is sufficiently large to ensure that, for all $s \ge -\ln(T^*(x))$,
 \begin{eqnarray}\label{Hmai4bis}
\frac{d}{ds}(H_{\eta}(w(s),s))&\le&\frac{\eta(p_c+3)}{2}H_{\eta}(w(s),s)
-\frac{\eta (p_c-1)}{p_c+15}\int_{B}(\partial_{s} w)^2\frac{\rho_{\eta}}{1-|y|^2}{\mathrm{d}}y \\
&&-\frac{\eta(p_c-1)}8 \ibint (|\nabla w|^2-(y \cdot \nabla w)^2)\rho_{\eta}\y-\frac{\eta(p_{c}+1)}{4(p_{c}-1)}\int_{B}w^2\rho_{\eta} {\mathrm{d}}y\no\\
&&-\frac{\eta (p_c-1)}{4(p_c+1) s^{a}}
\int_{B}   |w|^{p_c+1}\ln^a(\phi^2w^2+2)\rho_{\eta}{\mathrm{d}}y+C_0.\no
\end{eqnarray} 
By using the definition of ${\cal E}_{\eta}(w(s),s))$ given  in  \eqref{Geta} together with  the estimate \eqref{Hmai4bis}, we  easily show that,   for all $s\geq -\ln(T^*(x))$,  
 \begin{eqnarray}\label{24}
 \frac{d}{ds}({\cal E}_{\eta}(w(s),s))\! &\!\!\leq &\!\!\!- \lambda_0
e^{\frac{-\eta(p_{c}+3)s}{2}}\int_{B}(\partial_{s} w)^2\frac{\rho_{\eta}}{1-|y|^2}{\mathrm{d}}y
-\lambda_0
e^{\frac{-\eta(p_{c}+3)s}{2}}\int_{B}(|\nabla w|^2-(y\cdot \nabla w)^2)\rho_{\eta} {\mathrm{d}}y\no\\
&&
-\lambda_0 e^{\frac{-\eta(p_{c}+3)s}{2}}\ibint w^2\rho_{\eta}\y-\frac{\lambda_0e^{\frac{-\eta(p_{c}+3)s}{2}}}{ s^a}
\int_{B}   |w|^{p_c+1}\ln^a(\phi^2w^2+2)\rho_{\eta}{\mathrm{d}}y\no\\
&&+\big(C_0-\frac{\theta_2 \eta (p_{c}+3)}{2}\big)e^{\frac{-\eta(p_{c}+3)s}{2}}.
\end{eqnarray}
where  $\lambda_{0}=\eta\min(\frac{ p_{c}-1}{p_{c}+15},\frac{ p_{c}-1}{8(p_{c}+1)},\frac{ p_{c}-1}{8}, \frac{ p_{c}+1}{4(p_c-1)})$.  Clearly, if we choose  $\theta_2 =\frac{2C_0}{\eta  (p_{c}+3)}$,  multiply the    inequality \eqref{24}
and   integrate in time between $s'$ and $s''$, for any $s''\ge s'\geq 
-\ln(T^*(x))$,  we obtain
\eqref{pGeta}.

\medskip

 We  prove    \eqref{positiveeta} here. 
  The argument is the
same as in the corresponding part in \cite{
HZjhde12, HZnonl12,H1, omar1, omar2}. We write the
proof for completeness. Arguing by contradiction, we assume that
there exists   
 $s_1 \geq -\ln (T^*(x))$ such that ${\cal E}_{\eta}(w(s_1),s_1)<0$. Since the energy ${\cal E}_{\eta}(w(s),s)$ decreases in time, we
have ${\cal E}_{\eta}(w(1+s_1),1+s_1)<0$.

 Consider now for $\delta>0$ the function
$\widetilde{w}^{\delta}(y,s)=w_{x,T^{*}(x)-\delta }(y,s)$. From
(\ref{scaling}), we see that for all $(y,s)\in B\times
[1+s_1,+\infty)$
\begin{equation}\label{blowupbefore}
 \widetilde{w}^{\delta}(y,s)=\frac{\phi (-\ln (\delta+e^{-s}))}{\ps}
w( \frac{y}{1+\delta e^s},-\ln (\delta+e^{-s})),
\end{equation}
where $\ps$ defined in \eqref{defphi}. Then, we make the following 3 observations: 
\begin{itemize}
\item  (A) Note that $\widetilde{w}^{\delta}$ is defined in  $ B\times [1+s_1,+\infty)$,
whenever $\delta>0$ is small enough such that
$-\ln(\delta+e^{-1-s_1})\ge s_1.$
\item (B) By construction, $\widetilde{w}^{\delta}$
is also a solution of equation (\ref{A1}).
\item (C) For $\delta$ small enough, we have
${\cal E}_{\eta}(\widetilde{w}^{\delta}(1+s_1),1+s_1)<0$ by continuity of the function
$\delta \mapsto {\cal E}_{\eta}(\widetilde{w}^{\delta}(1+s_1),1+s_1)$.
\end{itemize}
Now, we fix $\delta=\delta_1>0$ such that (A), (B) and (C) hold. Since
 ${\cal E}_{\eta}(\widetilde{w}^{\delta_1}(s),s)$ is decreasing  in time, we have
 \begin{equation}
\label{255}
\liminf_{s\rightarrow +\infty}{\cal E}_{\eta}(\widetilde{w}^{\delta_1}(s),s)\le {\cal E}_{\eta}(\widetilde{w}^{\delta_1}(1+s_1),1+s_1)<0,
 \end{equation}
on the one hand. 
  by using  the following  basic inequality 
$ab\le \frac12( a^2+b^2)$,  we have 
\begin{equation}\label{c55}
I_{\eta}(\widetilde{w}^{\delta_1},s)
\ge- \frac{\eta-\alpha(s)}{2(n-2\alpha(s))}
\ibint(\partial_s\widetilde{w}^{\delta_1})^2\rho_{\eta} {\mathrm{d}}y.
\end{equation}
By  \eqref{c55},  for sufficiently small $  \eta_0>0$, and  by using \eqref{ss00},  we deduce that  for all $\eta \in (0, \eta_0)$,
 for all $s\geq -\ln(T^*(x))$
\begin{equation}
H_{\eta}(\widetilde{w}^{\delta_1}(s),s)\ge 
-e^{-\frac{2(p_c+1)s}{p_c-1}}s^{\frac{2a}{p_c-1}}  \ibint  f(\ps \widetilde{w}^{\delta_1})\rho_{\eta} \y. \nonumber
\end{equation}
Therefore, by  the expression of ${\cal E}_{\eta}$ given by \eqref{Geta}, we have
\begin{equation}
{\cal E}_{\eta}(\widetilde{w}^{\delta_1}(s),s)\ge 
- e^{-\frac{\eta(p_{c}+3)s}{2}} e^{-\frac{2(p_c+1)s}{p_c-1}}s^{\frac{2a}{p_c-1}}  \ibint  f(\ps \widetilde{w}^{\delta_1})\rho_{\eta} \y
. \nonumber
\end{equation}
Due to  \eqref{equiv6},  and the  fact that $0\le \rho_{\eta}\le 1$, we infer,
\begin{align}\label{N0}
{\cal E}_{\eta}(\widetilde{w}^{\delta_1}(s),s)\ge
-C e^{-\frac{\eta(p_{c}+3)s}{2}} e^{-\frac{2(p_c+1)s}{p_c-1}}s^{\frac{2a}{p_c-1}}  \ibint  |\phi(s) \widetilde{w}^{\delta_1}|^{ p_c+1}  \y
-Ce^{-s}.
\end{align}
 Notice that,
after a change of variables defined in \eqref{blowupbefore}, we find that
\begin{equation*}
  \ibint  |\phi(s) \widetilde{w}^{\delta_1}|^{ p_c+1 } \y
=(1+\delta_1e^{s})^n\phi^{ p_c+1}\big(-\ln (\delta_1+e^{-s})\big)  \ibint 
|w(z,-\ln
(\delta_1+e^{-s}))|^{ p_c+1} {\mathrm{d}}z.
\end{equation*}
Since we have $-\ln (\delta_1+e^{-s})\rightarrow -\ln \delta_1$
as $s\rightarrow +\infty$, then
$\phi (-\ln (\delta_1+e^{-s})\rightarrow \phi(-\ln (\delta_1))$. Moreover, by exploiting the fact that 
$ p_c<1+\frac4{n-2}$.  Then    $\|w(s)\|_{L^{p_c+1}(B)}$ is
locally bounded, by a continuity argument, it follows that the
former integral remains bounded and
\begin{align}\label{N0b}
{\cal E}_{\eta}(\widetilde{w}^{\delta_1}(s),s)\ge
-C{(\delta_1+e^{-s})^n}e^{-s}s^{\frac{2a}{p_c-1}} 
e^{-\frac{\eta(p_{c}+3)s}{2}}  
-Ce^{-s}
\rightarrow 0.
\end{align}
as $s\rightarrow +\infty$. So, it follows that
\begin{equation}\label{d2}
\liminf_{s\rightarrow +\infty}{\cal E}_{\eta}(\widetilde{w}^{\delta_1}(s),s)\ge 0 .
\end{equation}

\noindent
  From \eqref{255}, this is a contradiction. Thus \eqref{positiveeta} holds.
 This concludes  the  proof of Proposition \ref{prop2.2}.
\Box

According to  Proposition \ref{prop2.2}, we obtain the following corollary which summarizes the principal properties of $H_{\eta}(w(s),s)$.
\begin{cors}\label{cor01}{\bf (Estimate on $H_{\eta}(w(s),s)$).}
For all $\eta \in (0,\eta_0)$,  
for all $T_0 \in (0,T(x_{0})]$, 
  $x\in \R^n,$ such that $|x-x_0|\le \frac{T_0}{\delta_0(x_0)}$,  and $s\geq -\ln (T^*(x)),$  
we have
\begin{equation}
-C\leq H_{\eta}(w(s),s) \leq \Big(\theta_2 +H_{\eta}(w(\tilde s_0),\tilde s_0) \Big)e^{\frac{\eta(p_{c}+3)s}{2}},
\end{equation}
\begin{equation}\label{cor01A}
\int_{s}^{s+1}\int_{B}(\partial_{s} w)^2\frac{\rho_{\eta}}{1-|y|^2}{\mathrm{d}}y{\mathrm{d}}\tau\leq C\Big(\theta_2 +H_{\eta}(w(\tilde s_0),\tilde s_0 ) \Big)e^{\frac{\eta(p_{c}+3)s}{2}}, 
\end{equation}
\begin{equation}  \label{cor01B}
\int_{s}^{s+1}
\int_{B}(w^2+|\nabla w|^2-(y\cdot \nabla w)^2)\rho_{\eta} {\mathrm{d}}y{\mathrm{d}}\tau\leq C\Big(\theta_2 +H_{\eta}(w(\tilde s_0),\tilde s_0 ) \Big)e^{\frac{\eta(p_{c}+3)s}{2}}, 
\end{equation}
\begin{equation} \label{cor01C}
\frac{1}{ s^a} \int_{s}^{s+1} 
\int_{B}   |w|^{p_c+1}\ln^a(\phi^2w^2+2)\rho_{\eta} {\mathrm{d}}y{\mathrm{d}}\tau \leq C\Big(\theta_2 +H_{\eta}(w(\tilde s_0),\tilde s_0 ) \Big)e^{\frac{\eta(p_{c}+3)s}{2}}, 
\end{equation}
where $w=w_{x,T^*(x)}$ is defined in \eqref{scaling}, and where   $\tilde s_0= -\ln (T^*(x))$.
\end{cors}
\begin{nb}
Using the definition \eqref{scaling}, we write easily 
$$C\theta_2 +CH_{\eta}(w(\tilde s_{0}),\tilde s_{0}) \leq \tilde K,$$
where $ \tilde K= \tilde K\Big(\eta,T_{0},\|(u(0),\partial_{t}u(0))\|_{H^{1}\times L^{2}(B(x_{0},\frac{T_{0}(x_{0})}{\delta_{0}(x_{0})}))}\Big)$ and $\delta_{0}(x_{0})\in (0,1)$ is defined in \eqref {nonchar}.
\end{nb}

\medskip

From Corollary \ref{cor01}, we are in position to prove the   Proposition \ref{prop21} \\

{\it Proof of  Proposition \ref{prop21} }:
Let $\eta \in (0,\eta_1)$. 
Note that the estimate on the space-time $L^{2}$ norm of $\partial_{s} w$ was already proved in \eqref{cor01A}.
 Moreover, by using  the estimates \eqref{cor01B},  and \eqref{cor01C},  we  prove  easily  the boundedness of 
$H^{1}_{loc,u}(\R^n)$ norm of  $ w$,  and the nonlinear norm
 with the ball   $B(0,\frac12 )$.
Thanks to the covering technique (we refer the reader to Merle
 and Zaag \cite{MZimrn05} (pure power case) and Hamza and Zaag in  Lemma 2.8 in \cite{HZjhde12}),  we easily extend this estimate from $B(0,\frac12 )$ to $B$. This concludes the proof
  of the Proposition \ref{prop21}.

\Box

\subsection{Proof of  Proposition \ref{prop123}} 

In this subsection, we assume $\eta\in (0,\eta_0)$, and 
we start by introducing the   three crucial   new
  functionals  defined by the following:
\begin{equation}\label{Po}
M_{\eta}(w(s))=\int_{B}\Big(y\cdot \nabla w \partial_sw+(y\cdot \nabla w)^2
\Big)\p {\mathrm{d}}y,
\end{equation}
\begin{equation}\label{Po1bis}
 J_{\eta}(w(s))= -\displaystyle\int_{B}w(\partial_s w+2y\cdot \grad w ) \frac{\p}{\sqrt{1-|y|^2}}{\mathrm{d}}y- \frac{ n}2 \displaystyle\int_{B}w^2\frac{\p}{\sqrt{1-|y|^2}}
 {\mathrm{d}}y,
\end{equation}
and
\begin{equation}\label{Po1bis1}
 {\cal{L}}_{\eta}(w(s))=M_{\frac12+\eta}(w(s))+(\frac12+\eta){{J}}_{\eta}(w(s)).
\end{equation}

\bigskip

First, we shall show that the above three functionals are well-defined. 
From    the fact that $(w,\partial_sw) \in  H^1(B)\times L^2(B)$, 
and  the basic inequality
$\big|M_{\eta}(w(s))\big|\le C  \int_{B} \big(|\nabla w|^2 +(\partial_sw)^2\big)
 {\mathrm{d}}y$, we conclude that the functional $M_{\eta}(w(s))$ is well-defined.
Next,  by the   Hardy type inequality \eqref{AHardy}
and  the algebraic  identity  $\frac{w^2}{1-|y|^2}=\frac{|y|^2w^2}{1-|y|^2}+w^2$, 
we conclude that 
\begin{equation}\label{Hardy}
 \int_{B}w^2\frac{\p}{{1-|y|^2}}\y\leq C\int_{B}|\nabla w|^2(1-|y|^2)\p \y+C\int_{B} w^2\p \y.
\end{equation}
Therefore, by observing
$\big|J_{\eta}(w(s))\big|\le C  \int_{B} \big(|\nabla w|^2 +(\partial_sw)^2+w^2\frac{\p}{1-|y|^2}\big)
 {\mathrm{d}}y$, exploiting \eqref{Hardy},
 we deduce that  $J_{\eta}(w(s))$ is well defined. 
 Consequently, $\mathcal{L}_{\eta}(w(s))$, being a combination of these terms as defined in \eqref{Po1bis1}, is also well-defined.

\bigskip
%

The first functional, $M_{\eta}(w(s))$, is derived from a Pohozaev-type identity obtained by multiplying equation \eqref{A1} with $(y \cdot \nabla w)\rho$. The second functional, $J_{\eta}(w(s))$, arises by multiplying equation \eqref{A1} by $w(1-|y|^2)^{-1/2}\rho$. We then introduce the composite functional $\mathcal{L}_{\eta}(w(s)) = M_{1/2+\eta}(w(s)) + bJ_{\eta}(w(s))$, where the constant $b$ is  chosen to cancel the singular term$\int_{B} |\nabla_{\theta} w|^2 \frac{|y|^2 \rho_{\eta}}{1-|y|^2} \mathrm{d}y$ upon estimating the time derivative of $\mathcal{L}_{\eta}(w(s))$. This construction enables us to control the time average of the nonlinear term with a singular weight
representing one of the primary contributions of this paper.
Consequently, the preceding analysis allows us to derive a bound for the time average of the singular integral $\int_{B} |\nabla_{\theta} w|^2 \frac{ \rho_{\eta}}{1-|y|^2} \mathrm{d}y.$

\medskip

We begin by estimating the functional $M_{\eta }(w(s))$ in the following lemma:
\begin{lem}\label{L01}
For all     $s'>s \geq -\ln (T^*(x))$, we have 
 \begin{eqnarray}\label{N0t1gg}
M_{\eta}(w(s'))-M_{\eta}(w(s))
 &=&-\eta \int_s^{s'} \int_{B}|\nabla_{\theta}w|^2
\frac{|y|^2\p}{1-|y|^2}\y{\mathrm{d}}\tau
+ \int_s^{s'}  (\eta-2\alpha(\tau))\int_{B}(y\cdot\grad  w)^2
\p{\mathrm{d}}y{\mathrm{d}}\tau\no\\
&&+\frac{n}2\int_s^{s'}  \int_{B}\Big(|\grad w|^2-(y\cdot\grad  w)^2\Big) \p
{\mathrm{d}}y {\mathrm{d}}\tau-\int_s^{s'} \int_{B}|\grad w|^2 \p
{\mathrm{d}}y{\mathrm{d}}\tau\nonumber\\
&&-\int_s^{s'}(n+2\alpha(\tau))  \int_{B}y \cdot \grad w \partial_{s}w\rho_{\eta} \y{\mathrm{d}}\tau\\
&&-\int_s^{s'} (\frac{2p_{c}+2}{(p_{c}-1)^2}-\gamma(\tau)) \int_{ B} wy\cdot \nabla w \rho_{\eta}\y{\mathrm{d}}\tau \no\\
&&- (n+2\eta) \int_s^{s'} e^{-\frac{2(p_c+1)\tau}{p_c-1}}\tau^{\frac{2a}{p_c-1}} \int_{B}
f(\phi w) \rho_{\eta} \y{\mathrm{d}}\tau\no\\
&&+2\eta \int_s^{s'}  e^{-\frac{2(p_c+1)\tau}{p_c-1}}\tau^{\frac{2a}{p_c-1}}  \int_{B}
f(\phi w) \frac{\rho_{\eta}}{1-|y|^2} \y{\mathrm{d}}\tau\no\\
&&-\frac{n}2 \int_s^{s'}  \int_{B}  (\partial_sw)^2 \p {\mathrm{d}}y{\mathrm{d}}\tau
+\eta \int_s^{s'}  \int_{B}  (\partial_sw)^2 \frac{|y|^2\p}{1-|y|^2}
{\mathrm{d}}y{\mathrm{d}}\tau. \nonumber
\end{eqnarray}
\end{lem}
{\it Proof}: Note that  $\L$ is a differentiable
function and we get,  for all
 $s\ge     \max(-\ln T^*(x),1)$
\begin{equation}
\frac{d}{ds}\L=  \int_{B}y\cdot \grad  w (\partial^2_sw+2y\cdot \grad \partial_s w)\rho_{\eta} \y+ \int_{ B}y\cdot \grad \partial_sw \partial_sw \rho_{\eta} \y.\no
\end{equation}
By virtue of \eqref{A1} and the divergence identity $\text{div}(\rho_{\eta} y) = n\rho_{\eta} - \frac{2\eta |y|^2 \rho_\eta}{1 - |y|^2}$, an integration by parts shows that
\begin{eqnarray}\label{122}
\frac{d}{ds}\L&=&  \int_{B}y\cdot \grad  w \div (\rho_{\eta}\nabla w-\rho_{\eta} (y \cdot \nabla w)y)\y\no\\
&&+(2\eta-2\alpha(s))  \int_{ B} (y\cdot \nabla w)^2  \rho_{\eta} \y
-(\frac{p_c+3}{p_c-1}+2\alpha(s))
\int_{B}y \cdot \grad w \partial_{s}w\rho_{\eta} \y\no\\
&&-(\frac{2p_{c}+2}{(p_{c}-1)^2}-\gamma(s))
\int_{ B} wy\cdot \nabla w \rho_{\eta}\y \no\\
&&-ne^{-\frac{2(p_c+1)s}{p_c-1}}s^{\frac{2a}{p_c-1}}  \int_{B}
f(\phi w) \rho_{\eta} \y+2\eta e^{-\frac{2(p_c+1)s}{p_c-1}}s^{\frac{2a}{p_c-1}}  \int_{B}
f(\phi w) \frac{|y|^2\rho_{\eta}}{1-|y|^2} \y\no\\
&&-\frac{n}2 \int_{B}  (\partial_sw)^2 \p {\mathrm{d}}y
+\eta  \int_{B}  (\partial_sw)^2 \frac{|y|^2\p}{1-|y|^2}
{\mathrm{d}}y. \nonumber
\end{eqnarray}
Additionally, an integration by parts, detailed in Appendix B, leads to the following identity
\begin{align}\label{mport01}
\int_{B}y\cdot\grad  w \Big(\div(\p \nabla w-\p (y\cdot\nabla w)y)\Big) \y  =-\eta \int_{B}|\nabla_{\theta}w|^2
\frac{|y|^2\p}{1-|y|^2}\y
-\eta \int_{B}(y\cdot\grad  w)^2
\p{\mathrm{d}}y\no\\
+\frac{n}2\int_{B}\Big(|\grad w|^2-(y\cdot\grad  w)^2\Big) \p
{\mathrm{d}}y -\int_{B}|\grad w|^2 \p
{\mathrm{d}}y.\nonumber
\end{align}
By combining the above equality and substituting $n = \frac{p_c+3}{p_c-1}$, we arrive at
\begin{eqnarray}\label{122bb}
\frac{d}{ds}\L
 &=&-\eta \int_{B}|\nabla_{\theta}w|^2
\frac{|y|^2\p}{1-|y|^2}\y
+(\eta-2\alpha(s))  \int_{B}(y\cdot\grad  w)^2
\p{\mathrm{d}}y\\\
&&+\frac{n}2\int_{B}\Big(|\grad w|^2-(y\cdot\grad  w)^2\Big) \p
{\mathrm{d}}y -\int_{B}|\grad w|^2 \p
{\mathrm{d}}y\nonumber\\
&&-\frac{n}2 \int_{B}  (\partial_sw)^2 \p {\mathrm{d}}y
+\eta  \int_{B}  (\partial_sw)^2 \frac{|y|^2\p}{1-|y|^2}
{\mathrm{d}}y \nonumber\\
&&-(n+2\alpha(s))
\int_{B}y \cdot \grad w \partial_{s}w\rho_{\eta} \y\no\\
&&
-(\frac{2p_{c}+2}{(p_{c}-1)^2}-\gamma(s))
\int_{ B} wy\cdot \nabla w \rho_{\eta}\y \no\\
&&-(n+2\eta)e^{-\frac{2(p_c+1)s}{p_c-1}}s^{\frac{2a}{p_c-1}}  \int_{B}
f(\phi w) \rho_{\eta} \y\no\\
&&+2\eta e^{-\frac{2(p_c+1)s}{p_c-1}}s^{\frac{2a}{p_c-1}}  \int_{B}
f(\phi w) \frac{\rho_{\eta}}{1-|y|^2} \y.\no
\end{eqnarray}
Integrating  \eqref{122bb} in time yields 
 \eqref{N0t1gg}.
This ends the proof of Lemma \ref{L01}. 
\Box

\bigskip

Furthermore,  by estimating  the time derivative of ${J}_{\eta }(w(s))$, we conclude  the following lemma:
\begin{lem}\label{L01bis}
For all   
  $s'>\taa$, we have 
 \begin{eqnarray}\label{V01t}
J_{\eta}(w(s'))-J_{\eta}(w(s))&=&-\int_s^{s'} 
\frac{1}{\tau^a}\int_{B}   |w|^{p_c+1}\ln^a(\phi^2w^2+2)\frac{\p}{\sqrt{1-|y|^2}}{\mathrm{d}}y{\mathrm{d}}\tau\\
&&+\int_s^{s'}\int_{B}|\grad w_{\theta}|^2\pp{\mathrm{d}}y{\mathrm{d}}\tau
-\int_s^{s'}\int_{B}(\partial_sw)^2\pp {\mathrm{d}}y{\mathrm{d}}\tau\nonumber\\
&&+  \int_s^{s'}(1-2\eta+2\alpha(\tau))\displaystyle\int_{B}
wy\cdot \grad w \frac{\p}{\sqrt{1-|y|^2}}{\mathrm{d}}y{\mathrm{d}}\tau
 \nonumber\\
&&+\int_s^{s'}(\frac{2(p_c+1)}
{(p_c-1)^2}-\gamma(\tau))\displaystyle\int_{B} w^2 \frac{\p}{\sqrt{1-|y|^2}}{\mathrm{d}}y{\mathrm{d}}\tau\nonumber\\
&&-2 \int_s^{s'}\displaystyle\int_{B}\partial_s wy\cdot \grad w \frac{\p}{\sqrt{1-|y|^2}}{\mathrm{d}}y{\mathrm{d}}\tau\nonumber\\
&&+\int_s^{s'}\int_{B}|\grad w_r|^2{\rho_{\frac12+\eta}}{\mathrm{d}}y{\mathrm{d}}\tau+2\int_s^{s'}\alpha(\tau)\displaystyle\int_{B} \partial_sww \frac{\p}{\sqrt{1-|y|^2}}{\mathrm{d}}y{\mathrm{d}}\tau.\nonumber
\end{eqnarray}
\end{lem}

\bigskip

{\it Proof}: Note that   for all
 $\taa$, we have
\begin{eqnarray*}
\frac{d}{ds}{J}_{\eta }(w(s))&=&- \int_{B}(\partial_sw)^2\frac{\p}{\sqrt{1-|y|^2}}{\mathrm{d}}y-
\int_{B}w(\partial^2_{s}w
+2y\cdot \grad \partial_sw) 
\frac{\p}{\sqrt{1-|y|^2}}{\mathrm{d}}y \\
&&-2 \displaystyle\int_{B}\partial_s wy\cdot \grad w \frac{\p}{\sqrt{1-|y|^2}}{\mathrm{d}}y- n \int_{B}w\partial_s w\frac{\p}{\sqrt{1-|y|^2}}
{\mathrm{d}}y.
\end{eqnarray*}
By using  equation (\ref{A1}), 
 we infer
\begin{eqnarray}\label{R0}
\frac{d}{ds}{J}_{\eta }(w(s))&=&- \int_{B}(\partial_sw)^2\pp {\mathrm{d}}y- \int_{B}\div(\p \grad w-\p (y\cdot \grad w)
y)w\frac1{\sqrt{1-|y|^2}}{\mathrm{d}}y \nonumber\\
&&-2(\eta-\alpha(s)) \displaystyle\int_{B}wy\cdot \grad w \frac{\p}{\sqrt{1-|y|^2}}{\mathrm{d}}y
+\big(\frac{2(p_c+1)}
{(p_c-1)^2}-\gamma(s)\big)\displaystyle\int_{B} w^2 \frac{\p}{\sqrt{1-|y|^2}}{\mathrm{d}}y\nonumber\\
&&-\frac{1}{ s^a}\int_{B}   |w|^{p_c+1}\ln^a(\phi^2w^2+2)\frac{\p}{\sqrt{1-|y|^2}}{\mathrm{d}}y+2\alpha(s)\displaystyle\int_{B} \partial_sww \frac{\p}{\sqrt{1-|y|^2}}{\mathrm{d}}y\nonumber\\
&&-2 \displaystyle\int_{B}\partial_s wy\cdot \grad w \frac{\p}{\sqrt{1-|y|^2}}{\mathrm{d}}y.
\end{eqnarray}
Substituting  \eqref{mport1} into  (\ref{R0}),   we deduce
\begin{eqnarray}\label{V01}
\frac{d}{ds}{J}_{\eta }(w(s))&=&- \int_{B}(\partial_sw)^2\pp {\mathrm{d}}y+ \int_{B}|\grad w_{\theta}|^2\pp{\mathrm{d}}y+\int_{B}|\grad w_r|^2{\rho_{\frac12+\eta}}{\mathrm{d}}y
\\
&&(1-2\eta+2\alpha(s)) \displaystyle\int_{B}wy\cdot \grad w \frac{\p}{\sqrt{1-|y|^2}}{\mathrm{d}}y
+\big(\frac{2(p_c+1)}
{(p_c-1)^2}-\gamma(s)\big)\displaystyle\int_{B} w^2 \frac{\p}{\sqrt{1-|y|^2}}{\mathrm{d}}y\nonumber\\
&&-\frac{1}{ s^a}\int_{B}   |w|^{p_c+1}\ln^a(\phi^2w^2+2)\frac{\p}{\sqrt{1-|y|^2}}{\mathrm{d}}y+2\alpha(s)\displaystyle\int_{B} \partial_sww \frac{\p}{\sqrt{1-|y|^2}}{\mathrm{d}}y\nonumber\\
&&-2 \displaystyle\int_{B}\partial_s wy\cdot \grad w \frac{\p}{\sqrt{1-|y|^2}}{\mathrm{d}}y.\no
\end{eqnarray}
Integrating  \eqref{V01} in time yields \eqref{V01t}.
This ends the proof of Lemma \ref{L01bis}. 
\Box

\bigskip

Combining these inequalities with Lemmas \ref{L01} and \ref{L01bis} leads directly to the
 following lemma:
\begin{lem}\label{L01bisbis}
For all     $s'>\taa$, we have 
 \begin{eqnarray}\label{R02bis}
 {\cal{L}}_{\eta }(w(s'))
- {\cal{L}}_{\eta }(w(s)) &=&-(\frac12+\eta)\frac{p_c-1}{p_c+1}\int_s^{s'} 
\frac{1}{\tau^a}\int_{B}   |w|^{p_c+1}\ln^a(\phi^2w^2+2)\frac{\p}{\sqrt{1-|y|^2}}{\mathrm{d}}y{\mathrm{d}}\tau\no\\
&&+(2\eta+1) \int_s^{s'}  e^{-\frac{2(p_c+1)\tau}{p_c-1}}\tau^{\frac{2a}{p_c-1}}  \int_{B}
\big(f_1(\phi w) +f_2(\phi w)\big)\frac{\rho_{\eta}}{\sqrt{1-|y|^2}} \y{\mathrm{d}}\tau\no\\
&&+(\frac{n}2+\eta-\frac12) \int_s^{s'} \int_{B}|\nabla w|^2
\rho_{\frac12+\eta}
\y{\mathrm{d}}\tau  \no\\ 
&&-(\frac{n+1}2+\eta) \int_s^{s'}  \int_{B}  (\partial_sw)^2 \rho_{\frac12+\eta}
{\mathrm{d}}y{\mathrm{d}}\tau\no\\
&&+ \int_s^{s'}  (\frac12+\eta-\frac{n}2-2\alpha(\tau))\int_{B}(y\cdot\grad  w)^2
\rho_{\frac12+\eta}{\mathrm{d}}y{\mathrm{d}}\tau\\
&&-\int_s^{s'}(n+2\alpha(\tau))  \int_{B}y \cdot \grad w \partial_{s}w\rho_{\frac12+\eta} \y{\mathrm{d}}\tau\no\\
&&-\int_s^{s'} (\frac{2p_{c}+2}{(p_{c}-1)^2}-\gamma(\tau)) \int_{ B} wy\cdot \nabla w \rho_{\frac12+\eta}\y{\mathrm{d}}\tau \no\\
&&
+(\frac12+\eta)  \int_s^{s'}(1-2\eta+2\alpha(\tau))\displaystyle\int_{B}
wy\cdot \grad w \frac{\p}{\sqrt{1-|y|^2}}{\mathrm{d}}y{\mathrm{d}}\tau
 \nonumber\\
&&+(\frac12+\eta)\int_s^{s'}(\frac{2(p_c+1)}
{(p_c-1)^2}-\gamma(\tau))\displaystyle\int_{B} w^2 \frac{\p}{\sqrt{1-|y|^2}}{\mathrm{d}}y{\mathrm{d}}\tau\nonumber\\
&&-(1+2\eta) \int_s^{s'}\displaystyle\int_{B}\partial_s wy\cdot \grad w \frac{\p}{\sqrt{1-|y|^2}}{\mathrm{d}}y{\mathrm{d}}\tau\nonumber\\
&&+(1+2\eta)\int_s^{s'}\alpha(\tau)\displaystyle\int_{B} \partial_sww \frac{\p}{\sqrt{1-|y|^2}}{\mathrm{d}}y{\mathrm{d}}\tau\nonumber\\
&&- (n+1+2\eta) \int_s^{s'} e^{-\frac{2(p_c+1)\tau}{p_c-1}}\tau^{\frac{2a}{p_c-1}} \int_{B}
f(\phi w) \rho_{\frac12+\eta} \y{\mathrm{d}}\tau.\no
\end{eqnarray}
\end{lem}

{\it Proof}:
Expanding ${\cal{L}}_{\eta }(w(s))$ according to its definition in \eqref{Po1bis1}, we write
\begin{equation}\label{Leta}
  {\cal{L}}_{\eta }(w(s')-
  {\cal{L}}_{\eta }(w(s))=M_{\frac12+\eta}(w(s')-M_{\frac12+\eta}(w(s)))+(\frac12+\eta )
\big(
J_{\eta}(w(s'))-J_{\eta}(w(s))\big).
\end{equation}
 Thanks to
\eqref{Leta}, the identities   \eqref{N0t1gg},  and \eqref{V01t} we get
 \eqref{R02bis}.
This ends the proof of Lemma \ref{L01bisbis}. 
\Box

Before proceeding to the proof of Proposition \ref{prop123}, we state and prove the following Lemma, which may be can be viewed as a corollary to Lemma \ref{lem21}.
\begin{lem}\label{bws010}
For all  
$\ttaa$, we have
\begin{equation}\label{ws1}
\int_{s}^{s+1}\int_{B}(\partial_s w)^2\frac{\p}{1-|y|^2}\y 
{\mathrm{d}}\tau 
 \le C\int_{s-1}^{s+2}{\cal {N}}_{\eta}(w(\tau)){\mathrm{d}}\tau+Ce^{-2s},
\end{equation}
where ${\cal {N}}_{\eta}(w(\tau))$ is given by \eqref{NN}.
\end{lem}
{\it Proof}:
Let  $\ttaa$,  $s_1=s_1(s)\in [s-1,s]$  
and $s_2=s_2(s)\in [s+1,s+2]$ to be chosen later.  From Lemma \ref{lem21}, and utilizing the inequality $ab \le a^2 + b^2$, it follows that:
\begin{eqnarray}\label{ws2}
\int_{s}^{s+1}\int_{B}(\partial_sw)^2\frac{|y|^2\p}{{1-|y|^2}}\y 
{\mathrm{d}}\tau &\le & C\int_{s_1}^{s_2} {\cal {N}}_{\eta}(w(\tau))
{\mathrm{d}}\tau  +C|E_{\eta }(w(s_1))| \nonumber\\
 &&+C|E_{\eta }(w(s_2))|+Ce^{-2s}.
\end{eqnarray}
We now bound the terms on the right-hand side of \eqref{ws2}. Specifically, by using the expression for $E_{\eta }(w(s))$ from \eqref{Eeta} and exploiting \eqref{equiv1}, we obtain
\begin{equation}\label{ws3}
|E_{\eta }(w(s_i))|\le 
C {\cal {N}}_{\eta}(w(s_i))
+Ce^{-2s_i} ,\qquad i \in \{1,2\}.
%
\end{equation}
By using the Mean Value Theorem, let us choose  $s_1=s_1(s)\in
[s-1,s]$ such that
\begin{equation}\label{ws4}
\displaystyle\int_{s-1}^{s}
{\cal {N}}_{\eta}(w(\tau))
{\mathrm{d}}\tau={\cal {N}}_{\eta}(w(s_1)).
\end{equation}
In view of \eqref{ws3},  and (\ref{ws4}) we write
\begin{equation}\label{ws5}
|E_{\eta }(w(s_1))|\le C\displaystyle\int_{s-1}^{s} 
{\cal {N}}_{\eta}(w(\tau))
{\mathrm{d}}\tau+Ce^{-2s}.
\end{equation}
Similarly, by using the Mean Value Theorem, we choose
$s_2=s_2(s)\in [s+1,s+2]$ such that
\begin{equation}\label{ws7}
\displaystyle\int_{s+1}^{s+2} {\cal {N}}_{\eta}(w(\tau))
{\mathrm{d}}\tau= {\cal {N}}_{\eta}(w(s_2)).
\end{equation}
Hence, by (\ref{ws3}) and \eqref{ws7}, we infer
\begin{equation}\label{ws8}
|E_{\eta }(w(s_2))|\le C\displaystyle\int_{s+1}^{s+2} {\cal {N}}_{\eta}(w(\tau)){\mathrm{d}}\tau+Ce^{-2s}.
\end{equation}
Combining \eqref{ws2}, \eqref{ws5}, and \eqref{ws8}, we deduce
\begin{equation}\label{p4bis24}
\int_{s}^{s+1}\int_{B}(\partial_sw)^2\frac{|y|^2\p}{{1-|y|^2}}\y 
{\mathrm{d}}\tau \le
C\displaystyle\int_{s-1}^{s+2}{\cal {N}}_{\eta}(w(\tau)){\mathrm{d}}\tau+Ce^{-2s}.
\end{equation}
Finally, using the identity $\frac{(\partial_sw)^2}{1-|y|^2} = \frac{|y|^2(\partial_sw)^2}{1-|y|^2} + (\partial_sw)^2$ along with \eqref{p4bis24}, we arrive at \eqref{ws1}. This completes the proof of Lemma \ref{bws010}.
\Box

\bigskip

Having proved Lemma \ref{bws010}, we proceed to Proposition \ref{prop123} by employing an analogous argument.

\bigskip
{\it Proof of Proposition \ref{prop123}:}\\
-Proof of  \eqref{w1eta}: \\
Let  $\tttaa $,  $s_3=s_3(s)\in [s-1,s]$  
and $s_4=s_4(s)\in [s+1,s+2]$ to be chosen later. From Lemma \ref{L01bisbis}, 
by using the   inequality, 
  $ab\le a^2+b^2$,  and the Hardy inequality \eqref{Hardy}, we can write 
\begin{eqnarray}\label{p4bis}
&&\frac{p_c-1}{2p_c+2}\int_{s_3}^{s_4} 
\frac{1}{\tau^a}\int_{B}   |w|^{p_c+1}\ln^a(\phi^2w^2+2)\frac{\p}{\sqrt{1-|y|^2}}{\mathrm{d}}y{\mathrm{d}}\tau\no\\
&\le&\underbrace{ \int_{s_3}^{s_4}  e^{-\frac{2(p_c+1)\tau}{p_c-1}}\tau^{\frac{2a}{p_c-1}}  \int_{B}
\Big(f_1(\phi w) +f_2(\phi w)\Big)\frac{\rho_{\eta}}{\sqrt{1-|y|^2}} \y{\mathrm{d}}\tau}_{A_1(s_3,s_4)}\\
& &+ C\underbrace{ \Big|
 \int_{s_3}^{s_4}  e^{-\frac{2(p_c+1)\tau}{p_c-1}}\tau^{\frac{2a}{p_c-1}}  \int_{B}
f(\phi w) \rho_{\frac12+\eta} \y{\mathrm{d}}\tau\Big|}_{A_2(s_3,s_4)}\no\\
&&+ C \int_{s_3}^{s_4}\!\int_{B}\Big(w^2+|\nabla  w|^2+\frac{(\partial_sw)^2}{1-|y|^2} \Big) \p {\mathrm{d}}y{\mathrm{d}}\tau  \nonumber\\
 && 
 +C|{\cal{L}}_{\eta }(w(s_4))|+C|
{\cal{L}}_{\eta }(w(s_3))|.\nonumber
\end{eqnarray}
Now, we control all the terms on the right-hand side
of the relation (\ref{p4bis}).

For the first term, we use \eqref{equiv2} and \eqref{equiv3}, to obtain
\begin{equation}\label{As1}
A_1(s_3,s_4)\le  \frac{C}{s}\int_{s_3}^{s_4} 
\frac{1}{\tau^a}\int_{B}   |w|^{p_c+1}\ln^a(\phi^2w^2+2)\frac{\p}{\sqrt{1-|y|^2}}{\mathrm{d}}y{\mathrm{d}}\tau
+Ce^{-2s}.
\end{equation}
Furthermore, by   exploiting \eqref{equiv1}, we get
\begin{equation}\label{As2}
A_2(s_3,s_4)\le C\int_{s_3}^{s_4} 
\frac{1}{\tau^a}\int_{B}   |w|^{p_c+1}\ln^a(\phi^2w^2+2)\p
{\mathrm{d}}y{\mathrm{d}}\tau
+Ce^{-2s}.
\end{equation}
In view of definition \eqref{Po1bis1}, it follows from the Cauchy-Schwarz and Hardy \eqref{Hardy} inequalities that
\begin{equation}\label{I01}
|{\cal{L}}_{\eta }(w(\tau))|\le C \displaystyle\int_{B}\Big(
w^2(\tau)+
(\partial_s
w(\tau))^2+|\nabla  w(\tau)|^2\Big) \p {\mathrm{d}}y,\qquad \forall \tau\in [s_3,s_4].
\end{equation}
Applying the Mean Value Theorem, we choose $s_3 \in [s-1, s]$ satisfying 
\begin{equation}\label{q04}
\displaystyle\int_{s-1}^{s} \displaystyle\int_{B}\Big(w^2+(\partial_s
w)^2+|\nabla  w|^2\Big)\p {\mathrm{d}}y{\mathrm{d}}\tau= \displaystyle\int_{B}\Big(w^2(s_3)+(\partial_s
w(s_3))^2+|\nabla  w(s_3)|^2\Big)\p {\mathrm{d}}y.
\end{equation}
 It then follows from \eqref{I01} and \eqref{q04} that 
\begin{equation}\label{I11110}
|{\cal{L}}_{\eta }(w(s_3))|\le C\displaystyle\int_{s-1}^{s} \displaystyle\int_{B}\Big(w^2+(\partial_s
w)^2+|\nabla  w|^2\Big)\p {\mathrm{d}}y{\mathrm{d}}\tau.
\end{equation}
Similarly, the Mean Value Theorem  allows us to choose $s_4 \in [s+1, s+2]$ satisfying \begin{equation}\label{q440}
\displaystyle\int_{s+1}^{s+2} \displaystyle\int_{B}\Big(w^2+(\partial_s
w)^2+|\nabla  w|^2\Big)\p {\mathrm{d}}y{\mathrm{d}}\tau= \displaystyle\int_{B}\Big(w^2(s_4)+(\partial_s
w(s_4))^2+|\nabla  w(s_4)|^2\Big) \p{\mathrm{d}}y.
\end{equation}
Combining \eqref{q440} with \eqref{I01} leads to the estimate 
\begin{equation}\label{I2n}
|{\cal{L}}_{\eta }(w(s_4))|\le C\displaystyle\int_{s+1}^{s+2} \displaystyle\int_{B}\Big(w^2+(\partial_s
w)^2+|\nabla  w|^2\Big) \p {\mathrm{d}}y{\mathrm{d}}\tau.
\end{equation}
Gathering  (\ref{p4bis}), \eqref{As1}, \eqref{As2}, (\ref{I11110}),  \eqref{I2n},  and \eqref{ss00}, 
we conclude  that, for all $\ttaa$, we have 
\begin{eqnarray}\label{p4}
\int_{s}^{s+1} 
\frac{1}{\tau^a}\int_{B}   |w|^{p_c+1}\ln^a(\phi^2w^2+2)\frac{\p}{\sqrt{1-|y|^2}}{\mathrm{d}}y{\mathrm{d}}\tau\qquad\qquad\qquad\no\\
\le C\int_{s-1}^{s+2}\!\int_{B}\frac{(\partial_sw)^2}{1-|y|^2}  \p {\mathrm{d}}y{\mathrm{d}}\tau
+C\int_{s-1}^{s+2} {\cal {N}}_{\eta}(w(\tau))
{\mathrm{d}}\tau+Ce^{-2s}.
\end{eqnarray}
Thanks to   Lemma \ref{bws010},  we have  for all  $\tttaa $, we have 
\begin{equation}\label{ws11}
\int_{s-1}^{s+2}\int_{B}(\partial_s w)^2\frac{\p}{1-|y|^2}\y 
{\mathrm{d}}\tau \le    C\int_{s-2}^{s+3}{\cal {N}}_{\eta}(w(\tau)){\mathrm{d}}\tau+Ce^{-2s},
\end{equation}
where ${\cal {N}}_{\eta}(w(\tau))$ is given by \eqref{NN}.
By exploiting  \eqref{p4}, and \eqref{ws11}, 
and \eqref{ss00}, 
we easily deduce
the desired estimate \eqref{w1eta}.

-Proof of  \eqref{wtangeta}:
Let  $\ttttaa $,  $s_5=s_5(s)\in [s-1,s]$  
and $s_6=s_6(s)\in [s+1,s+2]$ to be chosen later. From Lemma \ref{L01bis}, 
by using the   inequality, 
  $ab\le a^2+b^2$,  and the Hardy inequality \eqref{Hardy}, we can write 
\begin{eqnarray}\label{23j1}
\int_{s_5}^{s_6}\int_{B}|\grad w_{\theta}|^2\pp{\mathrm{d}}y{\mathrm{d}}\tau
&\le&
\int_{s-1}^{s+2} 
\frac{1}{\tau^a}\int_{B}   |w|^{p_c+1}\ln^a(\phi^2w^2+2)\frac{\p}{\sqrt{1-|y|^2}}{\mathrm{d}}y{\mathrm{d}}\tau\nonumber\\
&&+ C \int_{s-1}^{s+2}\!\int_{B}\Big(w^2+|\nabla  w|^2+\frac{(\partial_sw)^2}{1-|y|^2} \Big) \p {\mathrm{d}}y{\mathrm{d}}\tau  \nonumber\\
 && 
 +|J_{\eta }(w(s_5))|+|
J_{\eta }(w(s_6))|.
\end{eqnarray}
From \eqref{w1eta},  we obtain
\begin{align}\label{w1eta23}
&\int_{s-1}^{s+2} 
\frac{1}{\tau^a}\int_{B}   |w|^{p_c+1}\ln^a(\phi^2w^2+2)\frac{\p}{\sqrt{1-|y|^2}}{\mathrm{d}}y{\mathrm{d}}\tau
 \le C\int_{s-3}^{s+4}{\cal {N}}_{\eta}(w(\tau)){\mathrm{d}}\tau+Ce^{-2s}.
\end{align}
Thanks to   \eqref{ws11},  and the  definition of ${\cal {N}}_{\eta}(w(\tau))$  given by \eqref{NN}, we have 
\begin{equation}\label{ws11bb}
\int_{s-1}^{s+2}\int_{B} \Big(w^2+|\nabla  w|^2+\frac{(\partial_sw)^2}{1-|y|^2} \Big) 
(\partial_s w)^2\p\y 
{\mathrm{d}}\tau \le    C\int_{s-2}^{s+3}{\cal {N}}_{\eta}(w(\tau)){\mathrm{d}}\tau+Ce^{-2s}.
\end{equation}
In view of definition \eqref{Po1bis}, it follows from the Cauchy-Schwarz and Hardy \eqref{Hardy} inequalities that
\begin{equation*}
|J_{\eta }(w(\tau))|\le C \displaystyle\int_{B}\Big(
w^2(\tau)+
(\partial_s
w(\tau))^2+|\nabla  w(\tau)|^2\Big) \p {\mathrm{d}}y,\qquad \forall \tau\in [s_5,s_6].
\end{equation*}
Similar to the proof of   \eqref{I11110}
we can   choose $s_5 \in [s-1, s]$ and $s_6 \in [s+1, s+2]$ satisfying 
\begin{equation}\label{I11110j23}
|J_{\eta }(w(s_5))|+
|J_{\eta }(w(s_6))|\le C\displaystyle\int_{s-1}^{s+2} \displaystyle\int_{B}\Big(w^2+(\partial_s
w)^2+|\nabla  w|^2\Big)\p {\mathrm{d}}y{\mathrm{d}}\tau.
\end{equation}
By exploiting   \eqref{23j1},
\eqref{w1eta23}, \eqref{ws11bb},  and \eqref{I11110j23}, we easily deduce  \eqref{wtangeta}.
This ends the proof of Proposition \ref{prop123}'. 
\Box


\section{Almost linear bounds 
for the time average of the $H^{1}$ norm 
  and the  non linear term of $w$ with 
singular weight for  solution of equation 
(\ref{A})}
\label{section3}
Note that we can rewrite equation \eqref{scaling} in $w$ in the following equivalent form:
\begin{align}\label{eqc}
\partial_{s}^2w&=\frac{1}{\rho}\div(\rho \grad w-\rho(y\cdot \grad w)y)
-\frac{2p_c+2}{(p_c-1)^2}w+\gamma(s)w\nonumber\\
&-\Big(\frac{p_c+3}{p_c-1}+2\alpha (s)\Big)\partial_s w
-2y\cdot \grad \partial_{s}w+
\frac{1}{s^a}   |w|^{p_c-1}w\ln^a(\phi^2w^2+2),
\end{align}
with $\rho (y,s)=(1-|y|^2)^{\alpha(s)},$
and where  $\alpha (s), \gamma (s), \ps$ are given respectively  by 
\eqref{alpha}, \eqref{defgamma}
and \eqref{defphi},  respectively.

\bigskip

In this section, we refine the estimates established in Section \ref{section2} by deriving nearly linear bounds. Specifically, we prove that for any $q > 0$, the time-average of the $H^1$ norm of the solution is bounded by $O(s^{1+q})$, providing a control of arbitrarily small super-linear order. 
These results are built upon the exponential bounds derived in the preceding section for the $H^1$ norm, as well as the nonlinear term considered both with and without the singular weight.
  More precisely, this is the aim of this section.

\begin{prop}\label{prop21bb2428}
\noindent Let
 $q>0$. Consider $u $   a solution of ({\ref{gen}}) with
blow-up graph $\Gamma:\{x\mapsto T(x)\}$ and  $x_0$  a non
characteristic point under the condition \eqref{ss00}. Then,  
for all $T_0 \in (0,T(x_{0})]$,  $s\geq 3-\ln T_0,$   and $x\in \R^n,$ such that $|x-x_0|\le \frac{e^{-s}}{\delta_0(x_0)}$,  
we have
\begin{equation}\label{feb19}
\int_{s}^{s+1}\!\ibint \big( w^2+|\grad w|^2 +( \partial_{s}w)^2\frac{\rho}{1-|y|^2}\big)\y \t \leq K_3 s^{q+1}, 
\end{equation}
\begin{equation}\label{feb19alpha}
\frac{1}{s^a}\int_{s+1}^{s}\! \int_{B}   |w|^{p_c+1}\ln^a(\phi^2w^2+2) {\mathrm{d}}y{\mathrm{d}}\tau
\leq K_3 s^{q+1}, 
\end{equation}
where $w=w_{x,T^*(x)}$ is defined in \eqref{scaling}, with
$T^*(x)$ given by \eqref{18dec1},
 and $\delta_{0}(x_{0})$ defined in \eqref{nonchar}.  
\end{prop}
Clearly, by  combining    Proposition \ref{prop21bb2428}
and Proposition \ref{prop123},
we deduce the following 
\begin{cors}\label{corol1alpha}
Let
 $q>0$. Consider $u $   a solution of ({\ref{gen}}) with
blow-up graph $\Gamma:\{x\mapsto T(x)\}$ and  $x_0$  a non
characteristic point under the condition \eqref{ss00}. Then,  
for all $T_0 \in (0,T(x_{0})]$,  $s\geq 6-\ln T_0,$   and $x\in \R^n,$ such that $|x-x_0|\le \frac{e^{-s}}{\delta_0(x_0)}$,  
we have
\begin{equation}\label{14sep1b}
\frac{1}{s^a}\int_s^{s+1}\!\! 
\int_{B}   |w|^{p_c+1}\ln^a(\phi^2w^2+2)\frac{\p}{\sqrt{1-|y|^2}}{\mathrm{d}}y{\mathrm{d}}\tau
 \leq K_4 s^{q+1}, \quad 
\end{equation}
and
\begin{align}\label{14sep2b}
&\int_s^{s+1} \!
\int_{B}   |\nabla_{\theta}w|^2\frac{\p}{\sqrt{1-|y|^2}}{\mathrm{d}}y{\mathrm{d}}\tau \le  K_4 s^{q+1}. 
\quad 
\end{align}
\end{cors}

\bigskip

In order to prove the Proposition \ref{prop21bb2428},  we need to construct a Lyapunov functional for
equation \eqref{eqc}. Accordingly, we start by recalling from \eqref{Ealpha} the following functional
\begin{align*}
  E(w(s),s)=&\iint \Big(\frac{1}{2}(\partial_{s}w)^2+\frac{1}{2}(|\grad  w|^2-(y\cdot \grad w)^2)+\big(\frac{p_c+1}{(p_c-1)^2}-\frac12  \gamma (s)\big)w^2\nonumber\\
  &-e^{-\frac{2(p_c+1)s}{p_c-1}}s^{\frac{2a}{p_c-1}}   f(\phi w)\Big)\rho \y,\no
\end{align*}
where $f$ is given by \eqref{defF}. 
Then, we introduce the following functionals:
\begin{align}
F(w(s),s)=&-\alpha(s) \int_{B} w\partial_{s}w \rho \y +\frac{\alpha(s)n }{2}\int_{ B}w^2\rho  {\mathrm{d}}y,\label{Jpoly}\\
{\cal P}(w(s),s)=&E(w(s),s)+\nu F(w(s),s),\label{Ppoly}
\end{align}
where $\nu\in(0,1)$
will be chosen later, such 
that the energy ${\cal {P}} (w(s),s)$ satisfies  the  following inequality: 
\begin{equation}\label{PestimJ}
\frac{d}{ds}{\cal {P}}(w(s),s)\le - \alpha(s) \ibint (\partial_{s}w)^2\frac{\rho}{1-|y|^2}\y
-\frac{an \nu }{2s}{\cal {P}}(w(s),s)
+ \theta_3e^{-c s},
 \end{equation}
for some $c>0$, which implies  that  the functional
\begin{equation}\label{21juin1}
{\cal {F} }(w(s),s)=s^{\frac{an \nu}{2}} {\cal {P} }(w(s),s)+\theta_4s^{\frac{an \nu}{2}-5},
\end{equation}
where $\theta_4$ is a sufficiently large constant that will be determined later.
We will show that the functional  
${\cal {F} }(w(s),s)$
 is a decreasing 
functional.
\subsection{Classical energy estimates}\label{sec3.1}
In 
this subsection, we introduce two lemmas essential for constructing a Lyapunov functional for equation \eqref{eqc}.
We first establish a bound for the time derivative of $E(w(s), s)$ in the following lemma:
\begin{lem}\label{lem32} For all   
  $\taa$, we have  
\begin{align}\label{Lem32b}
\frac{d}{ds}E(w(s),s)=  &- 2\alpha(s) \int_{ B} (\partial_{s}w)^2\frac{\rho}{1-|y|^2}\y+\Sigma_{3}(s),
\end{align}
where $\Sigma_{3}(s)$ satisfies 
\begin{align}\label{Lem32}
\Sigma_{3}(s)
\le &\frac{C}{s^{a+\frac74}}\ibint |w|^{p_c+1}\ln^a(2+\phi^2 w^2)\rho \y\\
&+C(\frac{\e}{s}+\frac1{s^2}+e^{-\e s})\int_{B} \Big(w^2+(\partial_{s}w)^2+|\grad w|^2-(y\cdot \grad w)^2
\Big)\rho  \y \nonumber\\
&+Ce^{-\e s}
\int_{B} \Big((\partial_{s}w)^2+|\grad w|^2-(y\cdot \grad w)^2
\Big)\frac{\rho}{(1-|y|^2)^{\frac14}}  \y+ Ce^{-2s}, \qquad \forall \e>0.\no
\end{align}
\end{lem}

{\it Proof}:
Multiplying $\eqref{eqc}$ by $\partial_{s} w\rho(y,s) $ and integrating over  $B$, we obtain \eqref{Lem32b}, where
\begin{align}\label{JE12}
\Sigma_{3}(s)= &
\underbrace{
\frac{a}{(p_c+1)s^{a+1}}\ibint |w|^{p_c+1}\ln^a(\phi^2w^2+2)
\rho\y+\frac{2p_c+2}{p_c-1}
e^{-\frac{2(p_c+1)s}{p_c-1}}s^{\frac{2a}{p_c-1}}\ibint f_1(\phi w)
\rho \y}_{\chi_1(s)}\no\\
&+\underbrace{\frac{2p_c+2}{p_c-1}
e^{-\frac{2(p_c+1)s}{p_c-1}}s^{\frac{2a}{p_c-1}}\ibint f_2(\phi w)
\rho \y}_{\chi_2(s)}\no\\
&\underbrace{-\frac{2a}{(p_c-1)s}
e^{-\frac{2(p_c+1)s}{p_c-1}}s^{\frac{2a}{p_c-1}}\ibint\big( f_1(\phi w)+f_2(\phi w)\big)
\rho \y}_{\chi_3(s)}\underbrace{-{\frac{\gamma' (s)}2\ibint w^2\rho\y}}_{\chi_4(s)}\no\\
&+\underbrace{\frac{a}{2(p_c-1)s^2}\ibint \Big((\partial_{s}w)^2+|\grad w|^2-(y\cdot \grad w)^2
\Big)\rho \ln(1-|y|^2) \y} _{\chi_5(s)}\nonumber\\
&+\underbrace{\frac{a}{(p_c-1)s^2}  
(\frac{p_c+1}{(p_c-1)^2}-\frac{\gamma(s)}2)\ibint w^2\rho \ln(1-|y|^2) \y}_{\chi_6(s)} \nonumber\\
&-\underbrace{\frac{a}{(p_c-1)s^2}
\ibint e^{-\frac{2(p_c+1)s}{p_c-1}}s^{\frac{2a}{p_c-1}}   f(\phi w)\rho \ln(1-|y|^2) \y}_{\chi_7(s)},
\end{align}
where  $f_1$, and $f_2$ are given by \eqref{defF2}, and \eqref{defF123}.

Now, we control the terms $\chi_{i}(s)$,  for all $i\in \{1,2,3,4,5,6,7\}$.\\
By   using the definition of $f_1$, we  show that
\begin{align}\label{regroup}
\chi_{1}(s)= \frac{a}{(p_c+1)s^{a+1}}\iint  {|  w|^{p_c+1}}\ln^{{a-1}}(2+\phi^2 w^2  )\Big(\ln (\phi^2w^2+2 )-\frac{4s}{p_c-1} \Big)\rho \y.
 \end{align}
Note that,  the  control of   \eqref{regroup}   needs the  use of  the  fact that $a<0$.
 More precisely,
%
for all   $\taa$, we   divide $B$ into two parts
 \begin{equation}\label{27nov1bis}
\tilde B (s)=\{y \in B\,\,|\,\,  w^2(y,s)\leq  s^{\frac{2a-4}{p_c-1}}\}\,\,{\rm and }\,\, \check{ B}(s)=\{y \in B
\,\,|\,\,  w^2(y,s)\ge  s^{\frac{2a-4}{p_c-1}}\}.
\end{equation}
Accordingly, we write  
$\chi_{1}(s)=\Phi_1(s)+\Phi_2(s)$, where
\begin{align}
\Phi_1(s)=&
\frac{a}{(p_c+1)s^{a+1}}\int_{\tilde B(s)}  {|  w|^{p_c+1}}\ln^{{a-1}}(2+\phi^2 w^2  )\Big(\ln (\phi^2w^2+2 )-\frac{4s}{p_c-1} \Big)\rho\y
,
\nonumber\\
\Phi_2(s)=&
\frac{a}{(p_c+1)s^{a+1}}\int_{\check{ B}(s)}  {|  w|^{p_c+1}}\ln^{{a-1}}(2+\phi^2 w^2  )\Big(\ln (\phi^2w^2+2 )-\frac{4s}{p_c-1}\Big)\rho  \y.
\nonumber
\end{align}
On the one hand, by using 
 the definition of the set $\tilde B(s)$ given  in \eqref{27nov1bis} and   the expression of $\ps$ in  \eqref{defphi},   and  integrate over $\tilde B(s),$ we get, for all   $\taa$,
\begin{equation}\label{93b}
\Phi_1(s) \le \frac{C}{s^2}\iint   w^2  \rho \y.
\end{equation}
On the other hand, by using  the definition of the $\phi$ given by \eqref{defphi}, we write the identity
\begin{equation}\label{16dec102b}
\ln (2+\phi^2 w^2)-\frac{4s}{p_c-1}
 =\ln (2\phi^{-2}+ w^2)-\frac{2a}{p_c-1}\ln s.
\end{equation}
Furthermore,   for all    
  $s \geq 1$, we have   $\phi  (s)\ge 1 $. Therefore,
 by exploiting  \eqref{16dec102b}, we write  for all 
 $\taa$, 
\begin{equation}\label{16dec10bb}
\ln (2+\phi^2 w^2)-\frac{4s}{p_c-1}
 \ge - C\ln s.
\end{equation}
Also, by using the definition of the set $\check{B}(s)$ defined in \eqref{27nov1bis}, we can write 
 for all $\taa,$   if $y\in \check{B}(s)$ , we have 
\begin{equation}\label{16dec12bbg}
\ln(\phi^2w^2+2)\ge \ln(\phi^2s^{\frac{2a-4}{p_c-1}})\ge \frac{4s}{p_c-1}-\frac{4 }{p_c-1}\ln s.
\end{equation}
   Therefore,
 by exploiting \eqref{16dec10bb} and  \eqref{16dec12bbg} we have   for all  $\taa$, 
\begin{align}\label{131bisb}
\Phi_{2}(s)\le 
\frac{C\ln s}{s^{a+2}}\int_{B}  {|w|^{p_c+1}}\ln^{{a}}(2+\phi^2 w^2  )\rho \y.
\end{align}
Adding  \eqref{93b} and \eqref{131bisb}  and  using  the fact 
that
$\chi_{1}(s)=\Phi_1(s)+\Phi_2(s)$,   we get
\begin{equation}\label{13janva1}
\chi_{1}(s)\le 
\frac{C}{s^{a+\frac74}}\iint   {|  w|^{p_c+1}}\ln^{{a}}(2+\phi^2 w^2 )\rho \y+\frac{C}{s^2}\iint   w^2  \rho \y.
\end{equation}
Note from    \eqref{equiv2}  and  \eqref{equiv3}  that
\begin{equation}\label{15dec1}
\frac{1}{s}| f_1(\phi w)|+| f_2(\phi w)|\le   C+\frac{C}{s^2}|\phi z|^{p_c+1} \ln^a(2+\phi^2 z^2).
\end{equation}
By \eqref{15dec1},  we have, for all $\taa$,
\begin{equation}\label{sigma11dec18}
\chi_{2}(s)+
\chi_{3}(s)\le 
\frac{C}{s^{a+2}}\iint |w|^{p_c+1}\ln^a(\phi^2w^2+2)\rho \y+  C e^{-2s}.
\end{equation}
By the  expression of $\gamma(s)$ defined in 
  \eqref{defgamma}, we write the inequality  $|\gamma'(s)|\le \frac{C}{s^2},$  and  we get 
\begin{equation}\label{juin2025}
\chi_{4}(s)\leq \frac{C}{s^2}\int_{B} w^2\rho {\mathrm{d}}y.
\end{equation}

The  control of  the  remaining terms  needs the  use of  the  additional information obtained in Proposition \ref{prop123}.
 More precisely,  for all $\e>0$,
for all   
  $\taa$, we   divide $B$ into two parts
 \begin{equation}\label{271}
B_{1}(s)=\{y \in B\,\,|\,\,|y|\le  \sqrt{1-e^{-8\e s}}\}\,\,{\rm and }
\,\,B_{2}(s)=\{y \in B
\,\,|\,\, |y|\ge \sqrt{1-e^{-8\e s}}\}.
\end{equation}
%
%
Accordingly,  we write  $\chi_5(s)=\chi_5^1(s)+\chi_5^2(s)$, where
\begin{align}
\chi_5^{1}(s)=&
\frac{a}{2(p_c-1)s^2}\int_{B_1(s)} \Big((\partial_{s}w)^2+|\grad w|^2-(y\cdot \grad w)^2
\Big)\rho \ln(1-|y|^2) \y,\nonumber\\
\chi_5^{2}(s)=&
\frac{a}{2(p_c-1)s^2}\int_{B_2(s)} \Big(\frac{1}{2}(\partial_{s}w)^2+|\grad w|^2-(y\cdot \grad w)^2
\Big)\rho \ln(1-|y|^2) \y.\nonumber
\end{align}
On the one hand, by using 
 the definition of the set $B_1(s)$ given  in \eqref{271},   we get,  $-\ln(1-|y|^2)\le 8 \e s$, for all  $y\in B_1(s)$.
Hence,  we write
\begin{align}\label{xi41}
\chi_5^{1}(s)\le 
\frac{C\e }{s }\int_{B} \Big((\partial_{s}w)^2+|\grad w|^2-(y\cdot \grad w)^2
\Big)\rho  \y.
\end{align}
On the other hand, by using  the  fact that 
the function  $y\mapsto (1-|y|^2)^{\frac18}  \ln(1-|y|^2)$
 is uniformly bounded on $B$, 
 and  the definition of the set $B_2(s)$ given  in \eqref{271},   we get,  $(1-|y|^2)^{\frac18}\le e^{-\e s}$, for all  $y\in B_2(s)$.
Therefore,  
 \begin{align}\label{xi42}
\chi_5^{2}(s)\le Ce^{-\e s}
\int_{B} \Big((\partial_{s}w)^2+|\grad w|^2-(y\cdot \grad w)^2
\Big)\frac{\rho}{(1-|y|^2)^{\frac14}}  \y.
\end{align}
 Using \eqref{xi41} together with \eqref{xi42}, we see that
\begin{align}\label{xi6}
\chi_5(s)\le &
\frac{C \e}{s }\int_{B} \Big((\partial_{s}w)^2+|\grad w|^2-(y\cdot \grad w)^2
\Big)\rho  \y\\
&+Ce^{-\e s}
\int_{B} \Big((\partial_{s}w)^2+|\grad w|^2-(y\cdot \grad w)^2
\Big)\frac{\rho}{(1-|y|^2)^{\frac14}}  \y.\nonumber
\end{align}
 Similar to the proof of \eqref{xi6},  we easily obtain
 \begin{align}\label{xi51}
\chi_6(s)\le 
\frac{C\e }{s }\int_{B} w^2\rho  \y+ Ce^{-\e s}
\int_{B} w^2\frac{\rho}{(1-|y|^2)^{\frac14}}  \y.
\end{align}
For the particular case where $\eta=\frac34+\alpha(s)$, the Hardy inequality $\eqref{AHardy}$ yields
 \begin{align}\label{xi52b}
\int_{B} w^2\frac{\rho}{(1-|y|^2)^{\frac14}}  \y\le C
\int_{B} w^2\rho  \y+C
\int_{B} |\grad w|^2 (1-|y|^2) \rho \y.
\end{align}
Thus, by  \eqref{xi51}, and  \eqref{xi52b}, we deduce
\begin{align}
\chi_6(s)\le C(\frac{\e}{ s}+e^{-\e s})
\int_{B} w^2\rho  \y+Ce^{-\e s}
\int_{B} |\grad w|^2 (1-|y|^2) \rho \y.
\label{xi5}
\end{align}
Adding the fact  that $\chi_7(s)\le 0,$  
\eqref{13janva1},  \eqref{sigma11dec18}, \eqref{juin2025}, \eqref{xi6}, and \eqref{xi5}, 
we derive \eqref{Lem32}, which ends the proof of Lemma \ref{lem32}.
\Box

 We are now going to prove the following estimate for the functional $F(w(s),s)$.
\begin{lem}\label{lem3.3} For all   
  $\taa$, we have  
\begin{align}\label{6bb}
\frac{d}{ds}
F(w(s),s)
=&-\alpha(s)\int_{B} (\partial_{s} w)^2\rho {\mathrm{d}}y+\alpha(s)
\int_{B}(|\nabla w|^2-(y \cdot \nabla w)^2) \rho {\mathrm{d}}y\\
&+\frac{2p_{c}+2}{(p_{c}-1)^2}\alpha(s)\int_{B}w^2 \rho {\mathrm{d}}y
-\frac{\alpha(s)}{s^{a}}\ibint |w|^{p_c+1}\ln^a(2+\phi^2 w^2)\rho \y+\Sigma_4(s),\no
\end{align}
where $\Sigma_4(s)$ satisfies
\begin{eqnarray}\label{6}
\Sigma_4(s)
&\le&(\frac{\alpha(s)}{2\nu}+\frac{C}{s^2})\int_{B}(\partial_{s} w)^2 \frac{\rho}{1-|y|^2} {\mathrm{d}}y
+(20\nu\alpha(s)+Ce^{-\e s}) \int_{B} |\nabla w|^2(1-|y|^2)\rho {\mathrm{d}}y\no\\
&&+\frac{C}{s^2}\int_{B} w^2\rho  {\mathrm{d}}y
+Ce^{-\e s}\int_{B} \frac{(\partial_{s}w)^2}{(1-|y|^2)^{\frac14}}  \y.
\end{eqnarray}
\end{lem}
{\it Proof}: Note that $F(w(s),s)$ is a differentiable function, by using equation \eqref{eqc} and integrating by part, for all  $\taa$ we have \eqref{6b}, where
\begin{eqnarray}
\label{K20}
\Sigma_4(s)
&=&\underbrace{-2\alpha(s) \int_{B}\partial_{s} w y \cdot \nabla w\rho {\mathrm{d}}y}_{\Sigma_4^1(s)}+\underbrace{ \alpha^2(s) \int_{B}\partial_{s} w w\frac{|y|^2\rho}{1-|y|^2} {\mathrm{d}}y}_{\Sigma_4^2(s)}
\no\\
&&+\underbrace{ (\frac{n\alpha'(s)}{2}
-\gamma (s)\alpha(s)) \int_{B}w^2 \rho {\mathrm{d}}y
-(\alpha'(s)+2\alpha^2 (s)) \int_{B} w\partial_{s}w \rho \y } _{\Sigma_4^3(s)}\no\\
&&\underbrace{-\alpha (s) {\alpha}'(s) \ibint \big(w\partial_{s}w-\frac{n}{2}w^2
\big)\rho \ln(1-|y|^2) \y} _{\Sigma_{4}^4(s)}.
\end{eqnarray}
We now estiamte the terms of $\Sigma_4(s)$. From   the following basic inequality
\begin{equation}\label{cs}
ab\le \mu a^2+\frac{1}{4\mu}b^2, \qquad \forall \mu>0,
\end{equation}
we infer
\begin{equation}\label{juin12}
\Sigma_{4}^1(s) \le
\frac{\alpha(s)}{4\nu}\int_{B}(\partial_{s} w)^2 \frac{\rho}{1-|y|^2} {\mathrm{d}}y+4 \nu \alpha(s)\int_{B} |\nabla w|^2(1-|y|^2)\rho {\mathrm{d}}y,
\end{equation}
and
\begin{equation}\label{juin128}
\Sigma_{4}^2(s) \le
\frac{\alpha(s)}{4\nu}\int_{B}(\partial_{s} w)^2 \frac{\rho}{1-|y|^2} {\mathrm{d}}y+16\nu\alpha^3(s) \int_{B} w^2\frac{|y|^2\rho}{1-|y|^2} {\mathrm{d}}y.
\end{equation}
Now, we use   the   Hardy-Sobolev inequality   \eqref{AHardy} to write  
\begin{equation}\label{juin12bn}
\int_{B} w^2\frac{|y|^2\rho}{1-|y|^2} {\mathrm{d}}y\le  \frac1{\alpha^2 (s)}
\int_{B}|\nabla w|^2(1-|y|^2)\rho  {\mathrm{d}}y+\frac{n}{\alpha (s)}
\int_{B} w^2\rho  {\mathrm{d}}y.
\end{equation}
Thus, it follows from  \eqref{juin128}, \eqref{juin12bn} that
\begin{equation}\label{chi8}
\Sigma_4^2(s)\le
\frac{\alpha(s)}{4\nu}\int_{B}(\partial_{s} w)^2 \frac{\rho}{1-|y|^2} {\mathrm{d}}y+
  16\nu\alpha(s) 
\int_{B}|\nabla w|^2(1-|y|^2)\rho  {\mathrm{d}}y+16n\nu \alpha^2 (s)
\int_{B} w^2\rho  {\mathrm{d}}y.
\end{equation}
By using the inequality $|\alpha(s)|+|s\alpha'(s)| +|\gamma(s)|\le \frac{C}{s},$ and \eqref{cs},  we get 
\begin{align}\label{xi9}
\Sigma_{4}^3(s) \le 
\frac{C }{s^2 }\int_{B} \Big(w^2+(\partial_{s}w)^2\Big)\rho  \y.
\end{align}
Proceeding similarly as for the estimates  \eqref{xi5}, we easily deduce
\begin{align}
\Sigma_{4}^4(s) \le& \frac{C}{ s^2}
\int_{B} \Big(w^2+(\partial_{s}w)^2\Big)\rho  \y+Ce^{-\e s}
\int_{B} |\grad w|^2 (1-|y|^2) \rho \y\no\\
&+Ce^{-\e s}
\int_{B}(\partial_{s}w)^2 \frac{1
}{(1-|y|^2)^{\frac14}}  \y.
\label{xi10}
\end{align}
Finally by using \eqref{juin12}, \eqref{chi8}, \eqref{xi9} and $\eqref{xi10}$, and the fact that $\nu\in(0,1)$,  we have easily the estimate
\eqref{6}, which ends the proof of Lemma \ref{lem3.3}.
\Box

\subsection{Existence of a decreasing functional for equation (\ref{A})}
In this subsection, we establish the monotonicity of the functional associated with equation \eqref{eqc} by utilizing Lemmas \ref{lem32} and \ref{lem3.3}. For convenience, we recall the definition of the functional from \eqref{21juin1}:
\begin{equation*}
\mathcal{F}(w(s),s) = s^{\frac{an\nu}{2}} \mathcal{P}(w(s),s) + \theta_4 s^{\frac{an\nu}{2}-5},
\end{equation*}
where $\theta_4 > 0$ is a sufficiently large constant to be determined later. We demonstrate that the functional $\mathcal{F}(w(s),s)$ satisfies the following proposition:
\begin{prop}\label{proplyap} For all $\nu>0$, 
there exists $\lambda_1>0$ such that for all $\ttttaa$, 
we have 
\begin{align}\label{DEC101} 
{\cal {F}} (w(s+1),s+1)
 -{\cal {F}} (w(s),s)
\le  - \lambda_1 \int_{s}^{s+1}\tau^{\frac{an \lambda-2}{2}}\int_{ B} (\partial_{s}w)^2\frac{\rho}{1-|y|^2}\y{\mathrm{d}}\tau \no\\
- \lambda_1 \int_{s}^{s+1} \tau^{\frac{an \lambda-2}{2}}\int_{B}(|\nabla w|^2-(y \cdot \nabla w)^2) \rho {\mathrm{d}}y{\mathrm{d}}\tau
  - \lambda_1\int_{s}^{s+1}\tau^{\frac{an \lambda-2}{2}}
\int_{B}w^2 \rho {\mathrm{d}}y{\mathrm{d}}\tau\no\\
- \lambda_1\int_{s}^{s+1}\tau^{\frac{an \lambda-2}{2}}\int_{ B} |w|^{p_c+1}\ln^a(2+\phi^2 w^2)
\rho \y{\mathrm{d}}\tau.\qquad\qquad
\end{align}
Moreover,  for all $\ttttaa$, 
we have 
 \begin{equation}\label{posi1}
{\cal {F} }(w(s),s)\geq 0.
\end{equation}
\end{prop}
{\it Proof}: 
It follows from \eqref{ss00} that $\alpha(s) < \frac{3}{4}$ for all  $\taa$.
From  the definition of ${\cal P} (w(s),s)$ given  in \eqref{Ppoly},   Lemmas  \ref{lem32}, \ref{lem3.3}, 
we can write for all
  $\taa$, we have  
\begin{align}\label{PP1}
\frac{d}{ds}{\cal P}(w(s),s)\le  &\frac{\alpha(s)(p_c+3)\nu}2{\cal P}(w(s),s)- \Big(\frac32\alpha(s)-\frac{C_0  }{s^2}-C_0 e^{-\e s}\Big) \int_{ B} (\partial_{s}w)^2\frac{\rho}{1-|y|^2}\y\no\\
&-\Big(\alpha(s)\frac{(p_c+5)\nu}2   -\frac{C_0\e}{s }-\frac{C_0}{s^2}-C_0e^{-\e s}\Big)\int_{B} (\partial_{s} w)^2\rho {\mathrm{d}}y\no\\
&-\Big(\alpha(s)\nu (\frac{p_c-1}2-20\nu)   
 -\frac{C_0\e}{s }-\frac{C_0}{s^2}-C_0e^{-\e s}\Big)
\int_{B}(|\nabla w|^2-(y \cdot \nabla w)^2) \rho {\mathrm{d}}y\no\\
&-\Big(\alpha(s)\frac{(p_c+1)\nu}{2(p_c-1)}  -\frac{C_0\e}{s }-\frac{C_0}{s^2}-C_0e^{-\e s}\Big)
\int_{B}w^2 \rho {\mathrm{d}}y\no\\
&-\Big(\alpha(s)\frac{(p_c-1)\nu}{2(p_c+1)}-\frac{C_0}{s^{\frac74}}\Big)\frac1{s^{a}}\ibint |w|^{p_c+1}\ln^a(2+\phi^2 w^2)\rho \y\no\\
&+\underbrace{\frac{\alpha(s) (p_c+3)\nu}{2}
e^{-\frac{2(p_c+1)s}{p_c-1}}s^{\frac{2a}{p_c-1}}  \ibint  \Big(f_1(\phi w) +f_2(\phi w)\Big)\rho\y}_{\chi_{8}(s)} \no\\
&+C_0e^{-\e s}
\int_{B} \Big(|\grad w|^2-(y\cdot \grad w)^2
\Big)\frac{1}{(1-|y|^2)^{\frac14}}  \y+C_0e^{-2s},\no
\end{align}
where $C_0$  stands for some universal constant depending only on $n,$ and $a$.

By using \eqref{equiv2}, and \eqref{equiv3},  we infer  
\begin{equation}\label{xsi8}
\chi_{8}(s)\le 
\frac{\tilde C_0}{s^{a+2}}\ibint |w|^{p_c+1}\ln^a(2+\phi^2 w^2)\rho \y+\tilde C_0e^{-2s}.
\end{equation}
Utilizing \eqref{ss00}, we deduce that for all $\taa$, the following inequality holds
\begin{eqnarray*}
\frac12\alpha(s)-C_0e^{-\e s}-\frac{C_0}{s^2} \ge 0, \qquad \frac{C_0  }{s^2}+C_0 e^{-\e s}\le  \frac{C_0 \e }{s},\qquad
\alpha(s)\frac{(p_c-1)\nu}{4(p_c+1)}-\frac{C_0}{s^{\frac74}}-\frac{\tilde C_0}{s^2}\ge 0.
\end{eqnarray*}
Furthermore, we choose $\e= \min \left\{ -\frac{(p_c+5)a\nu}{8C_0(p_c-1)}, -\frac{a\nu}{16C_0}, -\frac{(p_c+1)a\nu}{8C_0(p_c-1)^2} \right\}$ and $\nu \in \left( 0, \frac{p_c-1}{160} \right)$. It follows that for all $\taa$, we have:
\begin{align}\label{PP1bis}
\frac{d}{ds}{\cal P} (w(s),s)\le  &\frac{\alpha(s)(p_c+3)\nu}2{\cal P}(w(s),s)- \alpha(s)\int_{ B} (\partial_{s}w)^2\frac{\rho}{1-|y|^2}\y\\
&-\alpha(s)\frac{(p_c+5)\nu}4   \int_{B} (\partial_{s} w)^2\rho {\mathrm{d}}y
-\alpha(s)\frac{(p_c-1)\nu}8  
\int_{B}(|\nabla w|^2-(y \cdot \nabla w)^2) \rho {\mathrm{d}}y\no\\
&-\alpha(s)\frac{(p_c+1)\nu}{4(p_c-1)} 
\int_{B}w^2 \rho {\mathrm{d}}y-\alpha(s)\frac{(p_c-1)\nu}{4(p_c+1)s^{a}}\ibint |w|^{p_c+1}\ln^a(2+\phi^2 w^2)\rho \y\no\\
& +C_0e^{\frac{a\nu}{8C_0} s}
\int_{B} \Big(|\grad w|^2-(y\cdot \grad w)^2
\Big)\frac{1}{(1-|y|^2)^{\frac14}}  \y
+ C_1e^{-2s}.\no
\end{align}
By using the definition of  ${\cal {F} }(w(s),s)$ given  in  \eqref{21juin1} together with  the estimate \eqref{PP1bis}, and the fact that 
$\frac{\alpha(s)(p_c+3)\nu}2=-\frac{an\nu}{2s}$,
 we   derive, for all    $\taa$
\begin{eqnarray}\label{M11}
\frac{d}{ds}
{\cal {F} }(w(s),s)
&\le & - \lambda_1s^{\frac{an \nu-2}{2}} \int_{ B} (\partial_{s}w)^2\frac{\rho}{1-|y|^2}\y\\
&&- \lambda_1s^{\frac{an \nu-2}{2}}  \int_{B}(|\nabla w|^2-(y \cdot \nabla w)^2) \rho {\mathrm{d}}y
  - \lambda_1s^{\frac{an \nu-2}{2}}
\int_{B}w^2 \rho {\mathrm{d}}y\no\\
&&- \lambda_1s^{\frac{an \nu-2}{2}-a}\ibint |w|^{p_c+1}\ln^a(2+\phi^2 w^2)
\rho \y\no\\
&&+C_0e^{\frac{a\nu}{8C_0} s}s^{\frac{an\nu}{2}}
\int_{B} \Big(|\grad w|^2-(y\cdot \grad w)^2
\Big)\frac{1}{(1-|y|^2)^{\frac14}}  \y\no\\
&&+  C_1 s^{\frac{an\nu}{2}}e^{-2s}
-\theta_4(5-\frac{an\nu}{2})s^{\frac{an\nu}{2}-6},\no
\end{eqnarray}
where $ \lambda_1= \inf (-\frac{a}{p_c-1},- \frac{a\nu}{4} ,
 - \frac{a(p_c+1)\nu}{4(p_c-1)^2}
,-\frac{a\nu}{p_c+1})$.
By  integrating  in time between $s$ and $s+1$
the  inequality
(\ref{M11}),  we   prove that for all    $\taa$
\begin{align}\label{step1}
{\cal {F} }(w(s+1),s+1)
 -{\cal F }(w(s),s)
\le  - \lambda_1 \int_{s}^{s+1}\tau^{\frac{an \nu-2}{2}}\int_{ B} (\partial_{s}w)^2\frac{\rho}{1-|y|^2}\y{\mathrm{d}}\tau \no\\
- \lambda_1 \int_{s}^{s+1} \tau^{\frac{an \nu-2}{2}}\int_{B}(|\nabla w|^2-(y \cdot \nabla w)^2) \rho {\mathrm{d}}y{\mathrm{d}}\tau
  - \lambda_1\int_{s}^{s+1}\tau^{\frac{an \nu-2}{2}}
\int_{B}w^2 \rho {\mathrm{d}}y{\mathrm{d}}\tau\no\\
- \lambda_1\int_{s}^{s+1}\tau^{\frac{an \nu-2}{2}}\int_{ B} |w|^{p_c+1}\ln^a(2+\phi^2 w^2)
\rho \y{\mathrm{d}}\tau\qquad \qquad \qquad\quad\ \  \no\\
+Ce^{\frac{a\nu}{8C_0} s}s^{\frac{an\nu}{2}}
\int_{s}^{s+1}\int_{B} \Big(|\grad w|^2-(y\cdot \grad w)^2
\Big)\frac{1}{(1-|y|^2)^{\frac14}}  \y{\mathrm{d}}\tau\qquad\qquad \no\\
+  \big( C_1 e^{-2s}
-\theta_4(5-\frac{an\nu}{2})\frac{1}{s^6}\big)s^{\frac{an\nu}{2}}.\qquad \qquad\qquad\qquad\qquad
\qquad\qquad\quad
\end{align}
Now we will use  the exponential \del{rought} \add{rough} estimates obtained in  Proposition \ref{prop21}, and   Corollary \ref{corol1}.
More precisely,  by \ref{feb19eta}, (\ref{feb19etaNL}, and \eqref{14sep2},
 we get, ,  
\begin{equation}\label{es1}
C
\int_{s}^{s+1}\int_{B} \Big(|\grad w|^2-(y\cdot \grad w)^2
\Big)\frac{1}{(1-|y|^2)^{\frac14}}  \y{\mathrm{d}}\tau\le C_2
e^{-\frac{a\nu}{16C_0} s},\qquad \forall \ttttaa.
\end{equation}
We now choose $\theta_4 $,  such that  we have,  
 \begin{equation}\label{theta2}
C_1 e^{-2s}+C_2
e^{\frac{a\nu}{16C_0} s} 
-\theta_4(5-\frac{an\nu}{2})\frac{1}{s^6}
 \le0, \qquad \forall \ttttaa.
\end{equation}
Clearly, by  combining   \eqref{step1}, \eqref{es1} and \eqref{theta2}
we deduce \eqref{DEC101}.\\

The  proof of \eqref{posi1}   is the same as the similar part in  Proposition \ref{prop2.2}.
This concludes the proof of   Proposition \ref{proplyap}.

\Box

\bigskip

Now, we are in position to prove  Proposition \ref{prop21bb2428}
\medskip

{\it Proof of  Proposition \ref{prop21bb2428} }
According to the Proposition \ref{proplyap}, we obtain the following corollary which summarizes the principal properties of ${\cal {P}}(w(s),s)$ and a polynomial  
bound result  of the time-average of the $H^1$ norm  and the nonlinear term with  the weight $\rho$.

\begin{cors}\label{cor01sec3}
For all $\nu>0$,  
for all $T_0 \in (0,T(x_{0})]$, 
  $x\in \R^n,$ such that $|x-x_0|\le \frac{T_0}{\delta_0(x_0)}$,  and $s\geq 3-\ln (T^*(x)),$  
we have
\begin{equation}
-C\leq {\cal {P} }(w(s),s) \leq \Big(\theta_4 +{\cal {P} }(w(\tilde s_1),\tilde s_1) \Big)s^{\frac{-an \nu}{2}} ,
\end{equation}
\begin{equation}\label{cor01Asec3}
\int_{s}^{s+1}\int_{B}(\partial_{s} w)^2\frac{\rho}{1-|y|^2}{\mathrm{d}}y{\mathrm{d}}\tau\leq
C  \Big(\theta_4 +{\cal {P} }(w(\tilde s_1),\tilde s_1) \Big) 
 s^{1-\frac{an \nu}{2}}, 
\end{equation}
\begin{equation}  \label{cor01Bsec3}
\int_{s}^{s+1}
\int_{B}(w^2+|\nabla w|^2-(y\cdot \nabla w)^2)\rho {\mathrm{d}}y{\mathrm{d}}\tau\leq C\Big(\theta_4 +{\cal {P} }(w(\tilde s_1),\tilde s_1 ) \Big)s^{1-\frac{an \nu}{2}}, 
\end{equation}
\begin{equation} \label{cor01Csec3}
\frac{1}{ s^a} \int_{s}^{s+1} 
\int_{B}   |w|^{p_c+1}\ln^a(\phi^2w^2+2)\rho {\mathrm{d}}y{\mathrm{d}}\tau \leq
C\Big(\theta_4 +{\cal {P} }(w(\tilde s_1),\tilde s_1 ) \Big)s^{1-\frac{an \nu}{2}}, 
\end{equation}
where $w=w_{x,T^*(x)}$ is defined in \eqref{scaling}, and where   $\tilde s_1= 3-\ln (T^*(x))$.
\end{cors}
From Corollary \ref{cor01sec3}, we are in position to prove the   Proposition \ref{prop21bb2428}.\\

{\it Proof  of the Proposition \ref{prop21bb2428}:}
The proof is 
the same as the similar part in the Proposition \ref{prop21}. Indeed,  thanks to the covering technique we obtain easily the required estimates, for $q=-\frac{an\nu}{2}$.  This concludes the proof
  of the Proposition \ref{prop21bb2428}.
\Box
\section{Proof of Theorem \ref{t1} }\label{section4}

In this section, we refine the estimates obtained in Lemma \ref{lem32} using the results derived in Section \ref{section3}, namely Proposition \ref{prop21bb2428}, and Corollary \ref{corol1alpha}. These improved bounds are instrumental in constructing a Lyapunov functional and establishing the proof of Theorem \ref{t1}. Indeed, we show that  $L(w(s),s)$  is a  Lyapunov functional by analyzing the energy functional $G(w(s),s)$. Specifically, we demonstrate that $G(w(s),s)$ satisfies the following differential inequality:
\begin{equation}\label{A56}
\frac{d}{ds}G(w(s),s) \le - \alpha(s) \iint \frac{(\partial_{s}w)^2}{1-|y|^2} \rho , dy + \frac{p_c+3}{2s\sqrt{s}}G(w(s),s) + \frac{\tilde \theta_1}{s^{5/2}}.
\end{equation}
This inequality provides a sharper  control than the estimate for $\mathcal{P}(w(s),s)$ established in Section \ref{section3}, effectively improving the decay from almost linear to linear. 
 This section is organized as follows:
\begin{itemize}
    \item In Subsection \ref{4.1}, we provide the derivation of the refined estimates, improving upon the preliminary results of Subsection \ref{sec3.1}.
    \item In Subsection \ref{4.2}, we state and prove a generalized version of Theorem \ref{t1} that holds uniformly for $x$ in a neighborhood of $x_0$. 
    Furthermore, we prove the non-negativity of $L$ in Lemma \ref{L56}.
\end{itemize}
We recall the functionals $J(w(s),s)$, and $G(w(s),s)$:
\begin{align}
J(w(s),s)=&-\frac{1}{s\sqrt{s}}\int_{B} w\partial_{s}w \rho \y +\frac{n }{2s\sqrt{s}}\int_{ B}w^2\rho  {\mathrm{d}}y,\label{Jpolynew}\\
G(w(s),s)=&E(w(s),s)+J(w(s),s).\label{Ppolynew}
\end{align}
\subsection{Refined  energy estimates}\label{4.1}
In this subsection, we state two lemmas that are crucial for deriving the differential inequality \eqref{A56}. More specifically, we first establish the following result:
\begin{lem}\label{lem41} For all   
  $\taa$, we have  
  \begin{equation}\label{0Lem41}
\frac{d}{ds}E(w(s),s)=- 2\alpha(s) \int_{ B} (\partial_{s}w)^2\frac{\rho}{1-|y|^2}\y+\Sigma_3(s),
\end{equation}
where
\begin{align}\label{Lem41}
\Sigma_3(s)\le 
&\frac{C}{s^{a+\frac74}}\ibint |w|^{p_c+1}\ln^a(2+\phi^2 w^2)\rho \y\\
&+\frac{C\ln s}{s^2}\int_{B} \Big(w^2+(\partial_{s}w)^2+|\grad w|^2-(y\cdot \grad w)^2
\Big)\rho  \y \nonumber\\
&+\frac{C}{s^4}
\int_{B} \Big((\partial_{s}w)^2+|\grad w|^2-(y\cdot \grad w)^2
\Big)\frac{\rho}{(1-|y|^2)^{\frac18}}  \y+Ce^{-2s}.\no
\end{align}
\end{lem}

{\it Proof}:
Let us first recall that \eqref{0Lem41} is given by Lemma \ref{lem32}, and 
\begin{equation}\label{JE12new}
\Sigma_3(s)= \sum_{i=1}^7\chi_i(s),
\end{equation}
where $\chi_{i}(s)$,  for all $i\in \{1,2,3,,4,5,6,7\}$ is given by \eqref{JE12}.
 Let us mention that the estimates  related to  $\chi_{1}(s)$, $\chi_{2}(s)$, $\chi_{3}(s)$, and $\chi_{4}(s)$, given by 
\eqref{13janva1},
\eqref{sigma11dec18}, and
 \eqref{juin2025}  and the nonpositivity $\chi_{7}(s)$
 are acceptable and does not need any improvement.

Now,
by using the additional information obtained    in Corollary \ref{corol1alpha}, we are going to refine the estimate related to $\chi_{5}(s)$,  and $\chi_{6}(s)$
 defined in \eqref{JE12}.
 More precisely, 
for all   
  $\taa$, we   divide $B$ into two parts
 \begin{equation}\label{271new}
B_{3}(s)=\{y \in B\,\,|\,\,|y|\le  \sqrt{1-\frac1{s^{64}} }\}\,\,{\rm and }
\,\,B_4(s)=\{y \in B
\,\,|\,\, |y|\ge \sqrt{1-\frac1{s^{64}}}\}.
\end{equation}
%
Accordingly,  we write  $\chi_5(s)=\tilde \chi_5^1(s)+\tilde \chi_5^2(s)$, where
\begin{align}
\tilde \chi_5^{1}(s)=&
\frac{a}{2(p_c-1)s^2}\int_{B_3(s)} \Big((\partial_{s}w)^2+|\grad w|^2-(y\cdot \grad w)^2
\Big)\rho \ln(1-|y|^2) \y\nonumber\\
\tilde\chi_5^{2}(s)=&
\frac{a}{2(p_c-1)s^2}\int_{B_4(s)} \Big(\frac{1}{2}(\partial_{s}w)^2+|\grad w|^2-(y\cdot \grad w)^2
\Big)\rho \ln(1-|y|^2) \y.\nonumber
\end{align}
On the one hand, by using 
 the definition of the set $B_3(s)$ given  in \eqref{271new},   we get,  $-\ln(1-|y|^2)\le 64\ln s $, for all  $y\in B_3(s)$.
Hence,  we write
\begin{align}\label{xi41new}
\tilde \chi_5^{1}(s)\le 
\frac{C\ln s }{s^2 }\int_{B} \Big((\partial_{s}w)^2+|\grad w|^2-(y\cdot \grad w)^2
\Big)\rho  \y.
\end{align}
On the other hand, by using  the  fact that 
the function  $y\mapsto (1-|y|^2)^{\frac1{16}}  \ln(1-|y|^2)$
 is uniformly bounded on $B$, 
 and  the definition of the set $B_4(s)$ given  in \eqref{271new},   we get,  $(1-|y|^2)^{\frac1{16}}  \le  \frac1{s^4}$, for all  $y\in B_4(s)$.
Therefore,  
 \begin{align}\label{xi42new}
\tilde \chi_5^{2}(s)\le \frac{C}{s^4}
\int_{B} \Big((\partial_{s}w)^2+|\grad w|^2-(y\cdot \grad w)^2
\Big)\frac{\rho}{(1-|y|^2)^{\frac18}}  \y.
\end{align}
 Using \eqref{xi41new} together with \eqref{xi42new}, we see that
\begin{align}\label{xi55}
\chi_5(s)\le &
\frac{C \ln s}{s^2 }\int_{B} \Big((\partial_{s}w)^2+|\grad w|^2-(y\cdot \grad w)^2
\Big)\rho  \y\\
&+\frac{C}{s^4}
\int_{B} \Big((\partial_{s}w)^2+|\grad w|^2-(y\cdot \grad w)^2
\Big)\frac{\rho}{(1-|y|^2)^{\frac18}}  \y.\nonumber
\end{align}
 Similarly,  we write  $\chi_6(s)=\tilde\chi_6^1(s)+\tilde\chi_6^2(s)$, where
\begin{align}
\tilde\chi_6^{1}(s)=&\frac{a}{(p_c-1)s^2}  
(\frac{p_c+1}{(p_c-1)^2}-\frac{\gamma(s)}2)\int_{B_3(s)} w^2\rho \ln(1-|y|^2) \y,\nonumber\\
\tilde \chi_6^{2}(s)=&\frac{a}{(p_c-1)s^2}  
(\frac{p_c+1}{(p_c-1)^2}-\frac{\gamma(s)}2)\int_{B_4(s)} w^2\rho \ln(1-|y|^2) \y.\nonumber
\end{align}
By using 
 the fact that  $-\ln(1-|y|^2)\le 64\ln  s$, for all  $y\in B_3(s)$,  we conclude
\begin{align}\label{xi51new}
\tilde\chi_6^{1}(s)\le 
\frac{C\ln s }{s^2 }\int_{B} w^2\rho  \y.
\end{align}
As before,  by  combining the fact that  
the function  $y\mapsto (1-|y|^2)^{\frac1{16}}  \ln(1-|y|^2)$
 is uniformly bounded on $B$, and  $(1-|y|^2)^{\frac1{16}}\le \frac1{s^4}$, for all  $y\in B_4(s)$, we infer  
 \begin{align}\label{xi52new}
\tilde \chi_6^{2}(s)\le \frac{C}{s^4}
\int_{B} w^2\frac{\rho}{(1-|y|^2)^{\frac18}}  \y.
\end{align}
For the particular case where $\eta=\frac78+\alpha(s)$, the Hardy inequality $\eqref{AHardy}$ yields
 \begin{align}\label{xi52bnew}
\int_{B} w^2\frac{\rho}{(1-|y|^2)^{\frac18}}  \y\le C
\int_{B} w^2\rho  \y+C
\int_{B} |\grad w|^2 (1-|y|^2) \rho \y.
\end{align}
Thus, by  \eqref{xi51new}, \eqref{xi52new}, \eqref{xi52bnew}, we deduce
\begin{align}
\chi_6(s)\le C\frac{\ln s}{ s^2}
\int_{B} w^2\rho  \y+\frac{C}{s^4}
\int_{B} |\grad w|^2 (1-|y|^2) \rho \y.
\label{xi5new}
\end{align}
The result \eqref{Lem41} derives immediately from 
\eqref{13janva1},
\eqref{sigma11dec18}, 
 \eqref{juin2025},  
 \eqref{xi55},  and \eqref{xi5new} and the nonpositivity $\chi_{7}(s)$,
which ends the proof of Lemma \ref{lem41}.
\Box

\bigskip

 We are now going to prove the following estimate for the functional $J(w(s),s)$.
\begin{lem}\label{lem4.2} For all   
  $\taa$, we have  
\begin{eqnarray}\label{6new}
\frac{d}{ds}J(w(s),s)
&=&-\frac1{s\sqrt{s}} \int_{B} (\partial_{s} w)^2\rho {\mathrm{d}}y+\frac1{s\sqrt{s}} 
\int_{B}(|\nabla w|^2-(y \cdot \nabla w)^2) \rho {\mathrm{d}}y \\
&&+\frac{2p_{c}+2}{(p_{c}-1)^2s\sqrt{s}} \int_{B}w^2 \rho {\mathrm{d}}y
-\frac{ 1}{s^{a+1}\sqrt{s}}\ibint |w|^{p_c+1}\ln^a(2+\phi^2 w^2)\rho \y +\Sigma_5(s),\no
\end{eqnarray}
where
\begin{equation}\label{6new30}
\Sigma_5(s)\le
\frac{C}{s^{\frac54}}\int_{B}(\partial_{s} w)^2 \frac{\rho}{1-|y|^2} {\mathrm{d}}y+
  \frac{C}{s^\frac74}
\int_{B}|\nabla w|^2(1-|y|^2)\rho  {\mathrm{d}}y+\frac{C}{s^\frac{5}2}
\int_{B} w^2\rho  {\mathrm{d}}y.
\end{equation}
\end{lem}
{\it Proof}: Note that $J(w(s),s)$ is a differentiable function, by using equation \eqref{eqc} and integrating by part, we have \eqref{6new}, where
\begin{eqnarray}
\label{K20new}
 \Sigma_5(s)
&=&\underbrace{- \frac2{s\sqrt{s}}\int_{B}\partial_{s} w y \cdot \nabla w\rho {\mathrm{d}}y}_{ \Sigma^1_5(s)}+\underbrace{ \frac{4\alpha(s)}{s\sqrt{s}} \int_{B}\partial_{s} w w\frac{|y|^2\rho}{1-|y|^2} {\mathrm{d}}y}_{ \Sigma^2_5(s)}
\no\\
&&\underbrace{- (\frac{3n}{4s}
+\gamma (s))\frac1{s\sqrt{s}} \int_{B}w^2 \rho {\mathrm{d}}y
+(\frac3{2s}+2\alpha (s))\frac1{s\sqrt{s}} \int_{B} w\partial_{s}w \rho \y } _{ \Sigma^3_5(s)}\no\\
&&\underbrace{-\frac{\alpha'(s)}{s\sqrt{s}}  \ibint \big(w\partial_{s}w-\frac{n}{2}w^2
\big)\rho \ln(1-|y|^2) \y} _{ \Sigma^4_5(s)}.\no
\end{eqnarray}
By applying the basic inequality $2ab \le a^2 + b^2$, we obtain the following estimates:
\begin{equation}\label{juin12new}
 \Sigma^1_5(s)\le
\frac{1}{s^{\frac54}}\int_{B}(\partial_{s} w)^2 \frac{\rho}{1-|y|^2} {\mathrm{d}}y+\frac1{s^{\frac74}}\int_{B} |\nabla w|^2(1-|y|^2)\rho {\mathrm{d}}y,
\end{equation}
and
\begin{equation}\label{juin128new}
 \Sigma^2_5(s)\le
\frac{1}{s^{\frac54}}\int_{B}(\partial_{s} w)^2 \frac{\rho}{1-|y|^2} {\mathrm{d}}y+\frac{C}{s^{\frac{15}4}} \int_{B} w^2\frac{|y|^2\rho}{1-|y|^2} {\mathrm{d}}y.
\end{equation}
Furthermore, invoking the estimate $|\alpha(s)| + |\gamma(s)| \le \frac{C}{s}$, we deduce
\begin{align}\label{xi9new}
 \Sigma^3_5(s)\le
\frac{C }{s^{\frac52} }\int_{B} \big(w^2+(\partial_{s}w)^2\big)\rho  \y.
\end{align}
By using the bound $|\alpha'(s)| \le \frac{C}{s^2}$, we arrive at
\begin{align}\label{xi9new30}
 \Sigma^4_5(s)\le
\frac{C }{s^{\frac72} }\int_{B} \big(w^2+(\partial_{s}w)^2\big)\rho(-\ln(1-|y|^2))  \y.
\end{align}
Proceeding similarly as for the estimates  \eqref{xi5new}, we easily deduce
\begin{align}
 \int_{B} w^2\rho(-\ln(1-|y|^2))  \y\le C\ln s
\int_{B} w^2\rho  \y+\frac{C}{s^2}
\int_{B} |\grad w|^2 (1-|y|^2) \rho \y.
\label{xi5new30}
\end{align}
Moreover, utilizing the fact that $-\ln(1-|y|^2)(1-|y|^2) \le C$ for all $y \in B$, we obtain
\begin{align}\label{xi9new30b}
 \int_{B} (\partial_{s}w)^2\rho(-\ln(1-|y|^2))  \y\le C \int_{B} (\partial_{s}w)^2\frac{\rho}{1-|y|^2} \y.
\end{align}
Thus, it follows from  \eqref{xi9new30}, \eqref{xi5new30}, and \eqref{xi9new30b} that 
\begin{align}\label{xi9new30bis}
 \Sigma^4_5(s)\le  \frac{C\ln s}{s^{\frac72}}
\int_{B} w^2\rho  \y+\frac{C}{s^{\frac{11}2}}
\int_{B} |\grad w|^2 (1-|y|^2) \rho \y+
  \frac{C}{s^{\frac72}} \int_{B} (\partial_{s}w)^2\frac{\rho}{1-|y|^2} \y.
\end{align}
Finally by using \eqref{juin12new}, \eqref{juin128new},  \eqref{xi9new},  \eqref{xi9new30bis}, and   the   Hardy-Sobolev inequality   \eqref{juin12bn},
  we have  the estimate \eqref{6new30}, which ends the proof of Lemma \ref{lem4.2}.
\Box

\bigskip

\subsection{A Lyapunov functional
for equation (\ref{A})} \label{4.2}
In this subsection, our aim is to construct a Lyapunov functional for equation \eqref{A}. Recalling that
\begin{equation}\label{21juin1bisBnew}
L(w(s),s)=\exp\left(\frac{p_c+3}{\sqrt{s}}\right) G(w(s),s)+\frac{\theta_1}{ s},
\end{equation} 
where $\theta_1$ is a sufficiently large constant that will be determined later. Indeed, with Lemmas \ref{lem41} and \ref{lem4.2}, we are in a position to state and prove Theorem \ref{t1}', which is a uniform version of Theorem \ref{t1} for $x$ near $x_0$.
We shall demonstrate that the functional $L(w(s),s)$ satisfies the following:

\noindent {\bf{Theorem} \ref{t1}'} {\it (Existence  of a Lyapunov functional for equation }
{\eqref{A}})\\
\label{t1bis}
{\it
Consider   $u $    a solution of ({\ref{gen}}) with blow-up graph
$\Gamma:\{x\mapsto T(x)\}$ and  $x_0$  a non characteristic point under the condition \eqref{ss00}.
Then, for all $T_0\in  (0,T(x_0)]$,  for all  $s\ge 6 -\log(T_0)$  and $x\in \R^n$, where $|x-x_0|\le \frac{e^{-s}}{\delta_0(x_0)}$,
 we have
 \begin{equation}\label{vv}
L(w(s+1),s+1)  -L(w(s),s) 
 \leq -\int_{s}^{s+1}\!\alpha(\tau) \exp\Big(\frac{p_c+3}{\sqrt{\tau}}\Big) \iint (\partial_{s}w)^2 \frac{\rho}{1-|y|^2}\y{\mathrm{d}}\tau,
 \end{equation}
where  $w=w_{x,T^*(x)}$ and $T^*(x)$  is defined in \eqref{18dec1}.}

\medskip

{\it Proof}: 
From  the definition of $G (w(s),s)$ given  in \eqref{Ppolynew},   Lemmas  \ref{lem41},  and \ref{lem4.2}, 
  using \eqref{0Lem41} , \eqref{Lem41}, \eqref{6new},  and \eqref{6new30}, we can write for all
  $s \geq -\ln (T^*(x))$, we have    
\begin{align}\label{PP1new}
\frac{d}{ds}G (w(s),s) \le  &   \frac{p_c+3}{2s\sqrt{s}} G(w(s),s)- \Big(2\alpha(s)-\frac{C_1}{s^{\frac54}}\Big) \int_{ B} (\partial_{s}w)^2\frac{\rho}{1-|y|^2}\y\no\\
&-\Big(\frac{p_c+7}{4\,s\sqrt{s}}   -\frac{C_1\ln s}{s^2}-\frac{C_1}{s^3}\Big)\int_{B} (\partial_{s} w)^2\rho {\mathrm{d}}y\no\\
&-\Big(\frac{p_c-1}{4\,s\sqrt{s}}  
 -\frac{C_1\ln s}{s^2 }-\frac{C_1}{s^{\frac74}}\Big)
\int_{B}(|\nabla w|^2-(y \cdot \nabla w)^2) \rho {\mathrm{d}}y\\
&-\Big(\frac{p_c+1}{2(p_c-1)s\sqrt{s}}  -\frac{C_1\ln s}{s^2 }-\frac{C_1}{s^{\frac52}}-\frac{C_1}{s^3}\Big)
\int_{B}w^2 \rho {\mathrm{d}}y\no\\
&-\Big(\frac{(p_c-1)}{2(p_c+1)s\sqrt{s}}-\frac{C_1}{s^{\frac74}}\Big)\frac1{s^{a}}\ibint |w|^{p_c+1}\ln^a(2+\phi^2 w^2)\rho \y\no\\
&+\underbrace{\frac{p_c+3}{2s\sqrt{s}}
e^{-\frac{2(p_c+1)s}{p_c-1}}s^{\frac{2a}{p_c-1}}  \ibint  \Big(f_1(\phi w) +f_2(\phi w\Big)\rho\y}_{\chi_{9}(s)} \no\\
&+\frac{C_1}{s^4}
\int_{B} \Big(|\grad w|^2-(y\cdot \grad w)^2
\Big)\frac{\rho}{(1-|y|^2)^{\frac18}}  \y
+Ce^{-2s},\no
\end{align}
where $C_1$  stands for some universal constant depending only on $n,$ and $a$.\\
\note{I recomputed the coefficients of $\int(\partial_s w)^2\rho$ and of the angular gradient: with the $\frac12$ factors in $E$ and $\frac{d}{ds}J$ from \eqref{6new} I get $1+\frac{p_c+3}{4}=\frac{p_c+7}{4}$ and $1-\frac{p_c+3}{4}=-\frac{p_c-1}{4}$. All coefficients stay negative, so \eqref{PP1newB} and the rest of the proof are unaffected, but please recheck.}
By using \eqref{equiv2}, and \eqref{equiv3},  we infer  
\begin{equation}\label{xsi11new}
\chi_{9}(s)\le 
\frac{C}{s^{a+\frac52}}\ibint |w|^{p_c+1}\ln^a(2+\phi^2 w^2)\rho \y+Ce^{-2s}.
\end{equation}
By exploiting \eqref{ss00}, we easily write that   for all   
  $s \geq -\ln (T^*(x))$, we have 
\begin{align}\label{PP1newB}
\frac{d}{ds}G (w(s),s)\le  &\frac{p_c+3}{2s\sqrt{s}} G(w(s),s)-\alpha(s)
  \int_{ B} (\partial_{s}w)^2\frac{\rho}{1-|y|^2}\y\no\\
&+\frac{C}{s^4}
\int_{B} \Big(|\grad w|^2-(y\cdot \grad w)^2
\Big)\frac{\rho}{(1-|y|^2)^{\frac18}}  \y
+ Ce^{-2s}.
\end{align}
It follows from \eqref{21juin1bisBnew} and a direct computation that
  \begin{equation}\label{17dec2new}
  \frac{d}{ds}
L(w(s),s) =\exp\Big(\frac{p_c+3}{\sqrt{s}}\Big)\Big(\frac{d}{ds}
G(w(s),s)-\frac{p_c+3}{2s\sqrt{s}}G(w(s),s)\Big)-\frac{\theta_1}{s^{2}}.
\end{equation}
Therefore, estimates \eqref{PP1newB} and \eqref{17dec2new} lead to the following 
\begin{align}\label{PP1newBH}
 \frac{d}{ds}
L(w(s),s) 
 &\leq -\alpha(s) \exp\Big(\frac{p_c+3}{\sqrt{s}}\Big) \iint (\partial_{s}w)^2 \frac{\rho}{1-|y|^2}\y\\
 &+\frac{C}{s^4} \exp\Big(\frac{p_c+3}{\sqrt{s}}\Big)
\int_{B} \Big(|\grad w|^2-(y\cdot \grad w)^2
\Big)\frac{\rho}{(1-|y|^2)^{\frac18}}  \y\no\\
&+C\exp\Big(\frac{p_c+3}{\sqrt{s}}\Big)e^{-2s}
-\frac{\theta_1}{s^{2}}.\no
\end{align}
A simple integration between $s$ and $s+1$ implies
\begin{eqnarray}\label{PP1newBI}
L(w(s+1),s+1)-L(w(s),s) 
 \leq -\int_{s}^{s+1}\alpha(\tau) \exp\Big(\frac{p_c+3}{\sqrt{\tau}}\Big) \iint (\partial_{s}w)^2 \frac{\rho}{1-|y|^2}\y{\mathrm{d}}\tau\no\\
 + \frac{C_3}{s^4} \exp\Big(\frac{p_c+3}{\sqrt{s}}\Big)
 \int_{s}^{s+1}\int_{B} \Big(|\grad w|^2-(y\cdot \grad w)^2
\Big)\frac{\rho}{(1-|y|^2)^{\frac18}}  \y{\mathrm{d}}\tau\no\\
+C_3\exp\Big(\frac{p_c+3}{\sqrt{s}}\Big)e^{-2s}
-\frac{\theta_1}{(s+1)^2}.\qquad\qquad\qquad\qquad\qquad\qquad\ \quad
\end{eqnarray}
Now we will use  the polynomial rough estimates obtained in  Proposition \ref{prop21bb2428}, and   Corollary \ref{corol1alpha}.
More precisely,  by  (\ref{feb19}, and \eqref{14sep2b}
in the particular case where $q =  \frac12$, we obtain  
\begin{equation}\label{es1v}
C_3\int_{s}^{s+1}
\int_{B} \Big(|\grad w|^2-(y\cdot \grad w)^2
\Big)\frac{1}{(1-|y|^2)^{\frac18}}  \y{\mathrm{d}}\tau\le C_4s^{\frac32}, \qquad \forall s\ge6-\ln (T^*(x)).
\end{equation}
Since we have $1\leq \exp\Big(\frac{p_c+3}{\sqrt{s}}\Big)$, 
We now choose $\theta_1>0$,  such that  we have,  
 \begin{equation}\label{theta3}
C_4\exp\Big(\frac{p_c+3}{\sqrt{s}}\Big)\frac{1}{s^{\frac52}}
+C_3\exp\Big(\frac{p_c+3}{\sqrt{s}}\Big)e^{-2s}
-\frac{\theta_1}{(s+1)^2}
\le0, \qquad \forall \taa
\end{equation}
Clearly, by  combining   \eqref{PP1newBI}, \eqref{es1v} and \eqref{theta3},
we deduce for all  $s\ge6-\ln (T^*(x))$,
\begin{equation}\label{nb1}
L(w(s+1),s+1)-L(w(s),s) 
 \leq -\int_{s}^{s+1}\alpha(\tau) \exp\Big(\frac{p_c+3}{\sqrt{\tau}}\Big) \iint (\partial_{s}w)^2 \frac{\rho}{1-|y|^2}\y{\mathrm{d}}\tau.
 \end{equation}
Inequality \eqref{nb1} is exactly \eqref{vv}.
This concludes the proof
of  Theorem \ref{t1}'.
\Box

\begin{cors}[Pointwise form of \eqref{vv}]\label{c:pointwise}
Taking $x=x_0$ in Theorem~$1'$ gives $T^*(x_0)=T_0$, hence $w=w_{x_0,T_0}$. On $[s,s+1]$ the functions $\tau\mapsto\alpha(\tau)=-\frac{a(n-1)}{4\tau}$ and $\tau\mapsto\exp(\frac{p_c+3}{\sqrt\tau})$ are positive and decreasing, so $\alpha(\tau)e^{\frac{p_c+3}{\sqrt\tau}}\ge\alpha(s+1)e^{\frac{p_c+3}{\sqrt{s+1}}}=-\frac{a(n-1)}{4(s+1)}e^{\frac{p_c+3}{\sqrt{s+1}}}$. Since the integrand in \eqref{vv} is non-negative, \eqref{vv} implies the inequality in Theorem~\ref{t1}.
\end{cors}
\bigskip

We now claim the following lemma:

\begin{lem}\label{L56}
Let $x_0$, $T_0$ and $w=w_{x_0,T_0}$ be as in Theorem~\ref{t1}. Then $L(w(s),s)\ge0$ for every $s\ge6-\ln T_0$.
\end{lem}
{\it Proof}: We follow the argument of Proposition~\ref{prop2.2}, replacing $\mathcal E_\eta$ by $L$. The only two properties of the Lyapunov functional used there are a decrease along a sequence of times and a lower bound with $\liminf\ge0$. We check both for $L$.

Assume, by contradiction, that $L(w(s_1),s_1)<0$ for some $s_1\ge 6-\ln T_0$.

\smallskip\noindent\emph{Step 1: decrease along integer steps.} By \eqref{vv} (with $x=x_0$), $L(w(s+1),s+1)\le L(w(s),s)$ for $s\ge6-\ln T_0$. Hence $L(w(1+s_1),1+s_1)<0$ and, with $s_k:=1+s_1+k$, $L(w(s_k),s_k)\le L(w(1+s_1),1+s_1)<0$ for all $k\in\mathbb N$. (Note that \eqref{vv} gives a decrease of $L$ only from $s$ to $s+1$, not pointwise in $s$, since in \eqref{PP1newBH} the term with the weight $(1-|y|^2)^{-1/8}$ is controlled by \eqref{es1v} only after integration over a unit interval.)

\smallskip\noindent\emph{Step 2: translated solutions.} For $\delta>0$ let $\tilde w^\delta:=w_{x_0,T_0-\delta}$, given by \eqref{blowupbefore}. It is the function of Theorem~$1'$ with $T_0$ replaced by $T_0-\delta$ (note $T^*(x_0)=T_0$), so \eqref{vv} holds for $\tilde w^\delta$ for $s\ge 6-\ln(T_0-\delta)$, in particular for $s\ge1+s_1$ if $\delta$ is small. Moreover, $\tilde w^\delta$ is defined on $B\times[1+s_1,\infty)$ for $\delta$ small, and $\delta\mapsto L(\tilde w^\delta(1+s_1),1+s_1)$ is continuous. Fix $\delta_1>0$ small with $L(\tilde w^{\delta_1}(1+s_1),1+s_1)=:-c_1<0$. Step~1 applied to $\tilde w^{\delta_1}$ gives
\begin{equation}\label{eq:negseq}
L(\tilde w^{\delta_1}(s_k),s_k)\le -c_1<0\qquad\text{for all }k\in\mathbb N .
\end{equation}

\smallskip\noindent\emph{Step 3: lower bound.} Let $\tilde w=\tilde w^{\delta_1}$. By $ab\le\frac12(a^2+b^2)$ and $n\ge2$,
\[
J(\tilde w(s),s)\ge-\frac1{2s\sqrt s}\int_B(\partial_s\tilde w)^2\rho\,\mathrm{d}y+\frac{n-1}{2s\sqrt s}\int_B\tilde w^2\rho\,\mathrm{d}y\ge-\frac1{2s\sqrt s}\int_B(\partial_s\tilde w)^2\rho\,\mathrm{d}y .
\]
In $E$ (see \eqref{Ealpha}), the terms $\frac12(\partial_s\tilde w)^2$, $\frac12(|\nabla\tilde w|^2-(y\cdot\nabla\tilde w)^2)$ and $(\frac{p_c+1}{(p_c-1)^2}-\frac{\gamma(s)}{2})\tilde w^2$ are non-negative for $s$ large (recall $|\gamma(s)|\le C/s$). Hence for $s\ge S_1$,
\[
L(\tilde w(s),s)\ \ge\ -\,e^{\frac{p_c+3}{\sqrt s}}\,e^{-\frac{2(p_c+1)s}{p_c-1}}s^{\frac{2a}{p_c-1}}\int_B f(\phi(s)\tilde w)\,\rho\,\mathrm{d}y .
\]
Since $e^{\frac{p_c+3}{\sqrt s}}\to1$ and $0\le\rho\le1$, the computation leading to \eqref{d2} in the proof of Proposition~\ref{prop2.2} (use of \eqref{equiv6}, the change of variables \eqref{blowupbefore}, and the local boundedness of $\|w(s)\|_{L^{p_c+1}(B)}$ since $p_c<p_S$) applies unchanged and gives
$\liminf_{s\to\infty}L(\tilde w(s),s)\ge0$.

\smallskip\noindent This contradicts \eqref{eq:negseq}. Hence $L(w(s),s)\ge0$ for all $s\ge6-\ln T_0$. \Box

\medskip

\textit{Proof of Theorem \ref{t1}.}
The decrease inequality is Corollary~\ref{c:pointwise}; the non-negativity of $L$ is Lemma~\ref{L56}, which is used with $w=w_{x_0,T_0}$ defined for all $s\ge-\ln T_0$ (the point $x_0$ is non-characteristic). \hfill \(\Box\)

\bigskip

Although Section \ref{section4} establishes a Lyapunov functional, the weak dissipation inherent in the logarithmic sub-conformal structure implies that the methodology of \cite{HZjmaa2020, HZ2022} yields only linear estimates for the $H^1 \times L^2$ time-average. We derive these results in Subsection \ref{section51}. As noted, overcoming this limitation constitutes a central novelty of our analysis, a significant analytical challenge necessitated by the criticality of the problem.

Indeed, by utilizing the linear estimate obtained in Subsection \ref{section51}, we sharpen these results by introducing a conformal energy functional. This yields a logarithmic estimate for the time-average of $\|\partial_s w\|_{L^2}$ and facilitates the improvement of our linear bounds. Finally, we demonstrate that the time-average of $\|\partial_s w\|_{L^2}$ is bounded. Subsequently, applying the methodology of \cite{HZjmaa2020, HZ2022} ensures that all terms within the energy functional remain uniformly bounded, thereby establishing Theorem \ref{t2}.

\section{Proof of Theorem \ref{t2}}\label{section5}

This section is devoted to the proof of Theorem \ref{t2}. We proceed in three parts:
\begin{enumerate}
    \item In Subsection \ref{section51}, leveraging the framework established in Theorem \ref{t1}, we derive linear estimates for the time-averaged $H^1$ norm and the singular weighted nonlinear term associated with the solution to \eqref{eqc}. 

    \item In Subsection \ref{section52}, we refine these results by exploiting the energy functional related to the conformal problem. This approach allows us to circumvent the limitations of the previous bounds and derive sharper estimates for $\partial_s w$ in $L^2(\mathbb{R}^n \times [s, s+1])$.

    \item Finally, in Subsection \ref{section53}, we establish the uniform boundedness of the $H^1 \times L^2$ norm for the solution of \eqref{A}, thereby completing the proof of Theorem \ref{t2}.
\end{enumerate}

\subsection{Linear estimates  for the time average of the $H^{1}$ norm 
   for  solution of equation (\ref{A})}\label{section51}
Based on Theorem \ref{t1}, we derive linear estimates for the time-average of the $H^1$ norm and the singular weighted nonlinear term of $w$ for the solution of \eqref{eqc}. This result improves upon the almost linear bounds established for these same quantities 
in    Proposition \ref{prop21bb2428},
and Corollary \ref{corol1alpha}.
  More precisely, this is the aim of this subsection.
\begin{prop}\label{prop21bb}
\noindent  Consider $u $   a solution of ({\ref{gen}}) with
blow-up graph $\Gamma:\{x\mapsto T(x)\}$ and  $x_0$  a non
characteristic point under the condition \eqref{ss00}. Then,  for all $T_0\in(0,T(x_{0})],$
   $s\geq 7-\ln T_0$ and $x\in \R^n,$ where $|x-x_0|\le \frac{e^{-s}}{\delta_0(x_0)}$, 
we have
\begin{equation}\label{4feb19}
\int_{s}^{s+1}\!\ibint \big( w^2(y,\tau )+|\grad w(y,\tau )|^2 +( \partial_{s}w(y,\tau ))^2\frac{\rho}{1-|y|^2}\big)\y \t \leq K_5 s,
\end{equation}
\begin{equation}\label{4feb19alphab}
\frac{1}{s^a}\int_{s+1}^{s}\! \int_{B}   |w|^{p_c+1}\ln^a(\phi^2w^2+2) {\mathrm{d}}y{\mathrm{d}}\tau
\leq K_5 s,
\end{equation}
where $w=w_{x,T^*(x)}$ is defined in \eqref{scaling}, with
$T^*(x)$ given by \eqref{18dec1},
 and $\delta_{0}(x_{0})$ defined in \eqref{nonchar}.
\end{prop}
Clearly, by  combining    Proposition \ref{prop21bb} and Proposition \ref{prop123},
we deduce the following  estimates:
\begin{cors}\label{4corol1alpha}
Consider $u $   a solution of ({\ref{gen}}) with
blow-up graph $\Gamma:\{x\mapsto T(x)\}$ and  $x_0$  a non
characteristic point under the condition \eqref{ss00}. Then,  for all $T_0\in(0,T(x_{0})],$
   $s\geq 10-\ln T_0$ and $x\in \R^n,$ where $|x-x_0|\le \frac{e^{-s}}{\delta_0(x_0)}$, 
we have
\begin{equation}\label{414sep1}
\int_s^{s+1} 
\frac{1}{\tau^a}\int_{B}   |w|^{p_c+1}\ln^a(\phi^2w^2+2)\frac{\p}{\sqrt{1-|y|^2}}{\mathrm{d}}y{\mathrm{d}}\tau
 \leq K_6 s, \quad \forall \eta\in(0,1),
\end{equation}
and
\begin{align}\label{414sep2b}
&\int_s^{s+1} 
\int_{B}   |\nabla_{\theta}w|^2\frac{\p}{\sqrt{1-|y|^2}}{\mathrm{d}}y{\mathrm{d}}\tau \le  K_6 s, \qquad \forall \eta\in(0,1).
\end{align}
\end{cors}

In order to establish Proposition \ref{prop21bb}, we begin by applying Theorem \ref{t1}', which allows us to derive the following

\begin{cors}\label{cor01sec4}
For all $T_{0}\in (0,T(x_{0})]$,  $x\in \R^n$, such that
 $|x-x_0|\le \frac{T_0}{\delta_0(x_0)}$, and 
 $s \geq  6-\ln (T^*(x))$
\begin{equation}\label{L1212}
-C\leq G(w(s),s)  \leq  C \Big(\theta_1 +G(w(\tilde s_2),\tilde s_2) \Big),
\end{equation}
\begin{equation}\label{cor01Asec4}
\int_{s}^{s+1}\int_{B}(\partial_{s} w)^2\frac{\rho}{1-|y|^2}{\mathrm{d}}y{\mathrm{d}}\tau\leq
 C\Big(\theta_1 +G(w(\tilde s_2),\tilde s_2) \Big)s,
 \end{equation}
where $w=w_{x,T^*(x)}$ is defined in \eqref{scaling},    $\tilde s_2= 6-\ln (T^*(x))$.
\end{cors}

\medskip

We are now in a position to prove Proposition \ref{prop21bb}. By exploiting the preceding Corollary, the argument proceeds similarly to those in \cite{MZajm03,MZimrn05} except for the treatment of the nonlinear terms and of the perturbation terms.

-Proof of \eqref{4feb19alphab}:
For $s\ge 7-\ln (T^*(x))$, let us work with time integrals between $s_{10}$ and  $s_{11}$ where $s_{10}\in [s-1,s]$
and $s_{11}\in [s+1,s+2]$.
 By integrating the expression \eqref{Ppolynew}
 of $G(w(s),s)$ in time between $s_{10}$ and $s_{11}$, where $s_{11}>s_{10}>6-\ln (T^*(x))$, we obtain:
\begin{align}\label{et}
\int_{s_{10}}^{s_{11}}\!G(w(\tau),\tau)\t=&\displaystyle\int_{s_{10}}^{s_{11}}\iint\!\!\Big ( \frac{1}{2}(\partial_{s}w)^2
+\frac{p_c+1}{(p_c-1)^2}w^2
-e^{-\frac{2(p_c+1)\tau}{p_c-1}}\tau^{\frac{2a}{p_c-1}}   f(\phi w)\Big )\rho {\mathrm{d}}y\t\nonumber\\
&+\frac{1}{2}\displaystyle\int_{s_{10}}^{s_{11}}\!\!\iint\!\!  (|\grad w|^2-(y\cdot \grad w)^2)\rho {\mathrm{d}}y\t-\frac12\int_{s_{10}}^{s_{11}}\!\gamma(\tau)\!\!\displaystyle\iint\!w^2\rho {\mathrm{d}}y\t\no\\
&-\int_{s_{10}}^{s_{11}}\!\! \frac1{\tau\sqrt{\tau}}\!\!\displaystyle\iint\!\!w\partial_s w\rho {\mathrm{d}}y\t+\int_{s_{10}}^{s_{11}}\!\! \frac{n}{2\tau\sqrt{\tau}}\!\!\displaystyle\iint\!\!w^2\rho {\mathrm{d}}y\t.
\end{align}
By multiplying the equation (\ref{eqc}) by $w\rho$ and integrating both in time and in space over $B\times [s_{10},s_{11}]$  we obtain the following identity, after some
integration by parts:
\begin{eqnarray}\label{et1}
&&\Big [\iint\!\!\Big (w\partial_{s}w-\frac n2w^2\Big ) \rho{\mathrm{d}}y\Big ]_{s_{10}}^{s_{11}}=
\int_{s_{10}}^{s_{11}}\!\!\iint\!\!(\partial_{s}w)^2\rho{\mathrm{d}}y\t\\
&&-\int_{s_{10}}^{s_{11}}\!\!\iint\!\!\big(|\grad w|^2-(y\cdot \grad w)^2\big)\rho{\mathrm{d}}y\t
-\frac{2p_c+2}{(p_c-1)^2}\int_{s_{10}}^{s_{11}}\!\!\iint\!\!w^2\rho{\mathrm{d}}y\t\nonumber\\
&&+\int_{s_{10}}^{s_{11}}\!\frac1{\tau^a}\iint\!\!|w|^{p_c+1}\ln(2+ \phi^2 w^2)\rho{\mathrm{d}}y\t-4\!\int_{s_{10}}^{s_{11}}\!\alpha(\tau)\iint\!\!w\partial_{s}w
\frac{|y|^2\rho}{1-|y|^2}{\mathrm{d}}y\t\nonumber\\
&&+2\!\!\int_{s_{10}}^{s_{11}}\!\!\iint\!\!y\cdot \grad w\partial_{s} w \rho\y\t
+\!\!\int_{s_{10}}^{s_{11}}\!\!\iint\!\!\gamma(\tau) w^2\rho\y\t
-2\int_{s_{10}}^{s_{11}}\alpha (\tau)\iint \partial_s w w\rho \y\t.\nonumber
\end{eqnarray}
Note that, by using the identity  \eqref{defF123}, we get
\begin{eqnarray}\label{23jan12}
\frac1{2s^a}|w|^{p_c+1}\ln(2+ \phi^2 w^2)- e^{-\frac{2(p_c+1)s}{p_c-1}}s^{\frac{2a}{p_c-1}}f(\phi w)=\frac{p_c-1}{2(p_c+1)s^a}|w|^{p_c+1}\ln(2+ \phi^2 w^2)\no\\
-e^{-\frac{2(p_c+1)s}{p_c-1}}s^{\frac{2a}{p_c-1}}\Big(f_1(\phi w)+
f_2(\phi w)\Big).
\end{eqnarray}
By combining the identities (\ref{et}),  (\ref{et1}) and exploiting  \eqref{23jan12}, we obtain
\begin{eqnarray}\label{et44}
&&\frac{p_c-1}{2(p_c+1)}\int_{s_{10}}^{s_{11}}\!\frac1{\tau^a}\iint 
|w|^{p_c+1}\ln(2+ \phi^2 w^2)\rho {\mathrm{d}}y\t=\frac12
\Big [\iint\!\!\big (w\partial_{s}w-\frac n2w^2\big ) \rho{\mathrm{d}}y\Big ]_{s_{10}}^{s_{11}}
\nonumber\\
&&-\int_{s_{10}}^{s_{11}}\!\!\iint\!\!(\partial_{s}w)^2\rho{\mathrm{d}}y\t+\int_{s_{10}}^{s_{11}}\!G(w(\tau),\tau)\t
-\!\!\int_{s_{10}}^{s_{11}}\!\!\iint\!\!y\cdot \grad w\partial_{s} w \rho{\mathrm{d}}y\t\nonumber\\
&&
+\underbrace{2\!\int_{s_{10}}^{s_{11}}\alpha(\tau)\!\iint\!\!w\partial_{s}w
\frac{|y|^2\rho}{1-|y|^2}{\mathrm{d}}y\t}_{D_1}
\underbrace{-\frac12 \int_{s_{10}}^{s_{11}}\!\!\iint\!\!\gamma(\tau) w^2\rho\y\t}_{D_2}
+\underbrace{\int_{s_{10}}^{s_{11}}\iint \alpha(\tau)\partial_s w w\rho \y\t}_{D_3}\no \\
&&+\underbrace{\int_{s_{10}}^{s_{11}}\!\! \frac1{\tau\sqrt{\tau}}\!\!\displaystyle\iint\!\!\big(w\partial_s w-\frac n2 w^2\big)\rho {\mathrm{d}}y\t}_{D_4}+\underbrace{\int_{s_{10}}^{s_{11}}\!\!\iint e^{-\frac{2(p_c+1)\tau}{p_c-1}}\tau^{\frac{2a}{p_c-1}} f_1(\phi w)\rho {\mathrm{d}}y\t}_{D_5}\nonumber\\
&&+\underbrace{\int_{s_{10}}^{s_{11}}\!\!\iint e^{-\frac{2(p_c+1)\tau}{p_c-1}}\tau^{\frac{2a}{p_c-1}} f_2(\phi w)\rho {\mathrm{d}}y\t}_{D_6}.
\end{eqnarray}
We claim that 
the proof of \eqref{4feb19alphab}
follows from the following Lemma where we   control all the terms on the right-hand side of the relation
 (\ref{et44}) in terms of
the left-hand  side:
\begin{lem}\label{g}
 For all  $s \ge 7-\ln (T^*(x))$, 
 for all $\nn >0$,
\begin{equation}\label{control300}
\int_{s_{10}}^{s_{11}}\!\!\!\iint\!|\grad w|^2(1-|y|^2) \rho{\mathrm{d}}y\t
\le K_{7}s +K_{7}\int_{s_{10}}^{s_{11}}\iint\! e^{-\frac{2(p_c+1)\tau}{p_c-1}}\tau^{\frac{2a}{p_c-1}} f(\phi w)\rho{\mathrm{d}}y\t,\qquad
\end{equation}
\begin{equation}\label{control3}
\int_{s_{10}}^{s_{11}}\!\!\!\iint\!|y\cdot \grad w\partial_s w| \rho{\mathrm{d}}y\t
\le \frac{K_{7}}{\nn}s +K_{7}\nn\int_{s_{10}}^{s_{11}}\iint\! e^{-\frac{2(p_c+1)\tau}{p_c-1}}\tau^{\frac{2a}{p_c-1}} f(\phi w)\rho{\mathrm{d}}y\t,\qquad
\end{equation}
\begin{equation}\label{control1}
\sup_{\tau\in [s_{10},s_{11}]}\iint \!\!w^2(y,\tau)\rho{\mathrm{d}}y\le \frac{K_{7}}{\nn} s+K_{7}\nn\int_{s_{10}}^{s_{11}}\iint\! e^{-\frac{2(p_c+1)\tau}{p_c-1}}\tau^{\frac{2a}{p_c-1}} f(\phi w)\rho{\mathrm{d}}y\t,
\end{equation}
\begin{align}\label{control4}
\iint|w\partial_{s}w|\rho {\mathrm{d}}y\le&
\iint (\partial_{s}w)^2\rho {\mathrm{d}}y+
 \frac{K_{7}}{\nn} s+K_{7}\nn\int_{s_{10}}^{s_{11}}\iint\! e^{-\frac{2(p_c+1)\tau}{p_c-1}}\tau^{\frac{2a}{p_c-1}} f(\phi w)\rho{\mathrm{d}}y\t,
\end{align}
\begin{equation}\label{control5}
\iint \big( (\partial_{s}w(y,s_{10}))^2+(\partial_{s}w(y,s_{11}))^2\big)\rho {\mathrm{d}}y
\le K_{7}s,
\end{equation}
\begin{equation}\label{control30}
|D_1|
\le \frac{K_{7}}{\nn}s +K_{7}\nn
\int_{s_{10}}^{s_{11}}\iint\! e^{-\frac{2(p_c+1)\tau}{p_c-1}}\tau^{\frac{2a}{p_c-1}} f(\phi w)\rho{\mathrm{d}}y\t,\qquad
\end{equation}
\begin{equation}\label{A19}
|D_2|+|D_3|+|D_4|
\le 
\frac{K_{7}}{\nn}s +K_{7}\nn\int_{s_{10}}^{s_{11}}\iint\! e^{-\frac{2(p_c+1)\tau}{p_c-1}}\tau^{\frac{2a}{p_c-1}} f(\phi w)\rho{\mathrm{d}}y\t,
\end{equation}
\begin{equation}\label{A30}
|D_5|+|D_6|
\le C+ \frac{C}{ s}\int_{s_{10}}^{s_{11}}\iint\! e^{-\frac{2(p_c+1)\tau}{p_c-1}}\tau^{\frac{2a}{p_c-1}} f(\phi w)\rho{\mathrm{d}}y\t.
\end{equation}
\begin{equation}\label{26control1}
\int_{s_{10}}^{s_{11}}\!\iint\! e^{-\frac{2(p_c+1)\tau}{p_c-1}}\tau^{\frac{2a}{p_c-1}} f(\phi w)\rho{\mathrm{d}}y\t
\le C\int_{s_{10}}^{s_{11}}
\frac{1}{\tau^a} \!\!\int_{B} \!\!  |w|^{p_c+1}\ln^a(\phi^2w^2+2)\rho {\mathrm{d}}y{\mathrm{d}}\tau
+Ce^{-2s},
\end{equation}
\begin{equation}\label{b26control1}
\int_{s_{10}}^{s_{11}}
\frac{1}{\tau^a} \!\int_{B} \!\!  |w|^{p_c+1}\ln^a(\phi^2w^2+2)\rho {\mathrm{d}}y{\mathrm{d}}\tau
\le C\int_{s_{10}}^{s_{11}}\!\iint\! e^{-\frac{2(p_c+1)\tau}{p_c-1}}\tau^{\frac{2a}{p_c-1}} f(\phi w)\rho{\mathrm{d}}y\t
+Ce^{-2s},
\end{equation}
for some 
$s_{10}$ and $s_{11}$, such that  $s_{10}\in [s-1,s]$
and $s_{11}\in [s+1,s+2]$.
\end{lem}

\bigskip

Indeed, from (\ref{et44}) and  Lemma \ref{g},
we deduce that
\begin{eqnarray}\label{gg}
\int_{s_{10}}^{s_{11}}
\frac{1}{\tau^a} \!\int_{B} \!\!  |w|^{p_c+1}\ln^a(\phi^2w^2+2)\rho {\mathrm{d}}y{\mathrm{d}}\tau
&\!\le&\! C_5(K_{7}\nn +\frac{1}{s})
\int_{s_{10}}^{s_{11}}\frac{1}{\tau^a} \!\int_{B} \!\!  |w|^{p_c+1}\ln^a(\phi^2w^2+2)\rho {\mathrm{d}}y{\mathrm{d}}\tau\no\\
&&+C_5 \frac{K_{7}}{\nn} s+C_5e^{-2s}.
\end{eqnarray}
Now, we use \eqref{ss00} to satisfy
$$\frac{C_5}{7 - \ln(T(x_0))} \le \frac{1}{4}.$$
Using the fact that $s \ge 7- \ln(T^*(x)) \ge 7 - \ln(T(x_0)),$ it follows that $$\frac{C_5}{s} \le \frac{1}{4}.$$
Furthermore, by choosing $\nu_0$ small enough such that $C_5 K_7 \nu_0 \le \frac{1}{4}$. Therefore, \eqref{gg}
implies
\begin{align*}
\int_{s_{10}}^{s_{11}}
\frac{1}{\tau^a} \!\int_{B} \!\!  |w|^{p_c+1}\ln^a(\phi^2w^2+2)\rho {\mathrm{d}}y{\mathrm{d}}\tau
\le C{K_{7}}s, \qquad \forall 
s\ge 7-\ln (T^*(x)).
\end{align*}
Since $[s,s+1]\subset  [s_{10},s_{11}]$, and $s-s_{10}\le 1$,
 we derive  
\begin{equation}\label{feb19alphabnvb}
\frac{1}{s^a}\int_{s}^{s+1}\! \int_{B}   |w|^{p_c+1}\ln^a(\phi^2w^2+2)\rho {\mathrm{d}}y{\mathrm{d}}\tau
\leq K_{8} s.
\end{equation}
 Thanks to the covering technique (we refer the reader to Merle
 and Zaag \cite{MZimrn05} (pure power case) and Hamza and Zaag in  Lemma 2.8 in \cite{HZjhde12}),  
 the estimate \eqref{feb19alphabnvb} can be easily extended to the case where $\rho$ is replaced by $1$, showing that estimate \eqref{4feb19alphab} holds.\\

It remains to prove Lemma \ref{g}.\\

  {\it{Proof of  Lemma \ref{g}:}}
 Thanks to \eqref{L1212}, and   \eqref{cor01Asec4}, we can adapt with no difficulty
 the proof  in the unperturbed case \cite{MZajm03,MZimrn05}  (up to some very minor changes),
in order to get   the proof of the estimates \eqref{control300}, \eqref{control3}, \eqref{control1},  \eqref{control4},   \eqref{control5}, and \eqref{A19}. Furthermore,
by using   \eqref{cor01Asec4}, we conclude 
\begin{equation}\label{aa1}
|D_1|\le \frac{C\tilde K_{5}}{\nn}s +\frac{C\nn}{s^2}\!\int_{s_1}^{s_8}\!\iint\!\!w^2
\frac{|y|^2\rho}{1-|y|^2}{\mathrm{d}}y\t.
\end{equation}
 Also, by using   the   Hardy-Sobolev inequality   \eqref{juin12bn},  and the estimates  \eqref{control300}, and \eqref{control1},     we  conclude 
\eqref{control30}.\\
Finally,
the results \eqref{A30}, \eqref{26control1}, and \eqref{b26control1},
 follows  immediately from 
 \eqref{equiv1},    \eqref{equiv2}  and  \eqref{equiv3}, 
This concludes the proof of Lemma \ref{g} and the estimate \eqref{4feb19alphab}  too.\\

-Proof of \eqref{4feb19}:
Let us use \eqref{control300},
\eqref{control1},    and  \eqref{26control1}   to deduce 
\begin{equation}\label{feb19v0}
\int_{s}^{s+1}\!\ibint \big( w^2+|\grad w|^2(1-|y|^2) \big)\rho\y \t \leq K s.
\end{equation}
Thanks to the covering technique and the estimate  \eqref{cor01Asec4}, we infer \eqref{4feb19}. This concludes the proof of  Proposition \ref{prop21bb}.

\Box
%

%
%
%
%


\subsection{Refined bounds for $\partial_s w$ in $L^{2}(\mathbb{R}^n \times [s, s+1])$}\label{section52}

In this subsection, we refine the $L^2$-estimate for $\partial_s w$ on the interval $[s, s+1]$, sharpening the linear bound obtained in Proposition \ref{prop21bb} to a logarithmic one. Our main result is as follows:
\begin{prop}\label{psec5}
For all $T_{0}\in (0,T(x_{0})]$,  for all
 $s \geq 11-\ln (T_{0})$ and
 $x\in \R^n$, where $|x-x_0|\le \frac{e^{-s}}{\delta_0(x_0)}$,
we have
\begin{equation}\label{L1212sec5}
\int_{s}^{s+1}\int_{\partial B}(\partial_{s} w)^2{\mathrm{d}}\sigma {\mathrm{d}}\tau\leq
 K_9\ln s,
 \end{equation}
 \begin{equation}\label{L1212sec5b}
\int_{s}^{s+1}\int_{B}\Big (\partial_s w(y,\tau
)-\lambda(\tau ,s) w(y,\tau )\Big )^2dy d\tau\le K_9\ln s,
\end{equation}
where $0\le \lambda(\tau ,s) \le C(\delta_0)$, for all $\tau \in
[s,s+1]$.
where $w=w_{x,T^*(x)}$ is defined in \eqref{scaling}.
\end{prop}

\bigskip
 We begin with  the following lemma
\begin{lem}\label{lem51} For all  $s\ge -\ln (T^*(x))$,
 \begin{equation}\label{lem5e1}
\frac{d}{ds}E_0(w(s),s)=\!
-\int_{\partial B}(\partial_{s} w)^2{\mathrm{d}}\sigma -2\alpha(s) \int_{B}\partial_{s} w 
 y\cdot \nabla w {\mathrm{d}}y
-2\alpha(s) \ibint (\partial_{s}w)^2\y+\Sigma_{5}(s),
\end{equation}
where $\Sigma_{5}(s)$ satisfies
    \begin{equation}\label{lem5e2}
\Sigma_{5}(s)\leq 
\frac{C}{s^{a+1}}\ibint |w|^{p_c+1}\ln^a(\phi^2w^2+2)\y
+ \frac{C}{s^2}\int_{B}w^{2} {\mathrm{d}}y
+C e^{-2s}.
\end{equation}
\end{lem}
{\it Proof}: 
Multiplying $\eqref{Ac}$ by $\partial_{s} w$ and integrating over  $B$, we obtain \eqref{lem5e1}, where
\begin{align}\label{0lem5e1}
\Sigma_{5}(s)= &
\underbrace{
\frac{a}{(p_c+1)s^{a+1}}\ibint |w|^{p_c+1}\ln^a(\phi^2w^2+2)
\y}_{\Sigma_5^1(s)}\underbrace{-{\frac{\gamma' (s)}2\ibint w^2\y}}_{\Sigma_5^2(s)}\\
&+\underbrace{\big(\frac{2p_c+2}{p_c-1}
e^{-\frac{2(p_c+1)s}{p_c-1}}-\frac{2a}{(p_c-1)s}\big)
e^{-\frac{2(p_c+1)s}{p_c-1}}s^{\frac{2a}{p_c-1}}\ibint\big( f_1(\phi w)+f_2(\phi w)\big)
\y}_{\Sigma_5^3(s)}.\no
\end{align}
Note from    \eqref{equiv2},  and  \eqref{equiv3}  that
\begin{equation}\label{sigma51}
\Sigma_5^1(s)+\Sigma_5^3(s)\le 
\frac{C}{s^{a+1}}\iint |w|^{p_c+1}\ln^a(\phi^2w^2+2) \y+  C e^{-2s}.
\end{equation}
Furthermore, by the inequality  $|\gamma'(s)|\le \frac{C}{s^2},$    we infer 
\begin{equation}\label{juin2026}
\Sigma_5^2(s)\leq \frac{C}{s^2}\int_{B} w^2 {\mathrm{d}}y.
\end{equation}
From \eqref{sigma51}, and \eqref{juin2026}, we conclude 
\eqref{lem5e2},
which ends the proof of Lemma \ref{lem51}.
\Box

\medskip

The control of the term $E(w(s),s)-E_0(w(s),s)$ for large $s$ relies on $1-\rho$ being small if,  $|y|\in [-\sqrt{1-s^{-32}}, \sqrt{1-s^{-32}}]$. In light of Proposition \ref{prop21bb} and Corollary \ref{cor01sec4},
 we are in position to prove the following:
\begin{lem}\label{ws01}
For all  $s\ge 10-\ln (T^*(x))$,  we have
\begin{equation}\label{sec5dif2}
\int_{s}^{s+1}\Big| E(w(\tau),\tau)-E_0 (w(\tau),\tau)\Big|
{\mathrm{d}}\tau \le K_{10}\ln s, 
\end{equation}
and
\begin{equation}\label{sec5dif1}
\int_{s}^{s+1}\Big| G(w(\tau),\tau)-E (w(\tau),\tau)\Big|
{\mathrm{d}}\tau \le \frac{K_{10}}{\sqrt{s}},
\end{equation}
where $E(w(\tau),\tau)$,   $E_0(w(\tau),\tau)$, and $G(w(\tau),\tau)$  are given by \eqref{Ealpha},  \eqref{E}, and \eqref{Lalpha}.
\end{lem}
{\it Proof}:\\  
-Proof of \eqref{sec5dif2}: 
For all  $s \ge 10-\ln (T^*(x))$,  we   divide $B$ into two parts
 \begin{equation}\label{271newf}
B_{5}(s)=\{y \in B\,\,|\,\,|y|\le  \sqrt{1-{s^{-32}} }\}\,\,{\rm and }
\,\,B_6(s)=\{y \in B
\,\,|\,\, |y|\ge \sqrt{1-{s^{-32}}}\}.
\end{equation}
Therefore,  we write  
\begin{equation} \label{17AAA}
\int_{s}^{s+1}\Big| E(w(\tau),\tau)-E_0 (w(\tau),\tau)\Big|
{\mathrm{d}}\tau\ \le\ \int_{s}^{s+1}\big|\tilde \chi_5(\tau)\big|\t+\int_{s}^{s+1}\big|\tilde \chi_6(\tau)\big|\t,
\end{equation}
where
\begin{align}
\tilde \chi_{i}(\tau)=&
\int_{B_i(\tau)} \Big(\frac{1}{2}(\partial_{s}w)^2+\frac{1}{2}|\grad w|^2-\frac12(y\cdot \grad w)^2\no\\
&+(\frac{p_c+1}{(p_c-1)^2}-\frac{\gamma(\tau)}2)w^2-e^{-\frac{2(p_c+1)\tau}{p_c-1}}\tau^{\frac{2a}{p_c-1}}   f(\phi w)\Big)(1-\rho) \y, 
\end{align}
for $i\in \{5,6\}$.
Clearly
 \begin{equation}\label{271newfx}
  0\le 1-\rho(y,\tau)\le \frac{C \ln \tau}{\tau}, \qquad \forall y\in  B_5(\tau).
  \end{equation}
On the one hand, by using    \eqref{271newfx}, and \eqref{equiv1}, we get for all  $\tau \ge 10-\ln (T^*(x))$, 
\begin{align*}
|\tilde \chi_{5}(\tau)|\le \frac{C\ln s}{s}
\int_{B} \Big((\partial_{s}w)^2+|\grad w|^2
+w^2+\frac{1}{\tau^{a}} |w|^{p_c+1}\ln^a(2+\phi^2 w^2)\Big) \y+C e^{-2s}.
\end{align*}
Therefore,
\begin{equation}\label{22A}
\int_{s}^{s+1}
|\tilde \chi_{5}(\tau)|\t\! \le\! \frac{C\ln s}{s}\int_{s}^{s+1}
\int_{B} \Big((\partial_{s}w)^2+|\grad w|^2
+w^2+\frac{1}{\tau^{a}} |w|^{p_c+1}\ln^a(2+\phi^2 w^2)\Big) \y\t
+C e^{-2s}.
\end{equation}
It follows from \eqref{4feb19} and \eqref{4feb19alphab} that
\begin{align}\label{17AA}
\int_{s}^{s+1}
|\tilde \chi_{5}(\tau)|\t \le K\ln s.
\end{align}
On the other hand,   by using  \eqref{equiv1},  and the fact that  
  $(1-|y|^2)^{\frac1{8}}\le \frac1{s^4}$, for all  $y\in B_6(s)$, we get for all  $\tau \ge 10-\ln (T^*(x))$, 
  \begin{align*}
|\tilde \chi_{6}(\tau)|\le \frac{C}{s^4}
\int_{B} \frac{\Big((\partial_{s}w)^2+|\grad w|^2-(y\cdot\grad w)^2
+w^2+\frac{1}{\tau^{a}} |w|^{p_c+1}\ln^a(2+\phi^2 w^2)\Big)}{ 
(1-|y|^2)^{\frac1{8}}}
\y+C e^{-2s}.
\end{align*}
Thus, by   \eqref{4feb19}, \eqref{4feb19alphab},
\eqref{414sep1}, \eqref{414sep2b}
and the Hardy inequality \eqref{xi52bnew}, we deduce
\begin{align}\label{17A}
\int_{s}^{s+1}
|\tilde \chi_{6}(\tau)|\t \le \frac{K}{s^3}.
\end{align}
In light of \eqref{17AAA}, \eqref{17AA}, and \eqref{17A}, the bound \eqref{sec5dif2}  holds.\\
-Proof of \eqref{sec5dif1}: 
 Recalling the definitions of $E(w(s),s),$  $J(w(s),s),$ and $G(w(s),s)$ in \eqref{Ealpha}, \eqref{Jalpha}, and  \eqref{Lalpha}, and applying Young's inequality, we obtain
 \begin{equation}\label{0dif1}
\int_{s}^{s+1}\Big| G(w(\tau),\tau)-E (w(\tau),\tau)\Big|
{\mathrm{d}}\tau \le \frac{C}{s\sqrt{s}}\int_{s}^{s+1}\!\int_{B}\big(w^2+(\partial_sw)^2 \big)  {\mathrm{d}}y{\mathrm{d}}\tau.
\end{equation}
The estimate \eqref{sec5dif1} then follows from \eqref{4feb19}, concluding the proof of Lemma \ref{ws01}.
\Box

\bigskip

In view of Lemmas \ref{lem51} and \ref{ws01}, we  prove Proposition \ref{psec5}.

\medskip

{\it Proof of  Proposition \ref{psec5}}\\
-Proof of  \eqref{L1212sec5}: 
For $s\ge 11-\ln (T^*(x))$, let us work with time integrals between $s_{12}$ and $s_{13}$ where $s_{12}\in [s-1,s]$
and $s_{13}\in [s+1,s+2]$.  Thanks to Lemma \ref{lem51}, we write
 \begin{eqnarray}\label{sec5v1}
\int_{s_{12}}^{s_{13}}\int_{\partial B}(\partial_{s} w)^2{\mathrm{d}}\sigma\t &=&E_0(w(s_{12}),s_{12})- E_0(w(s_{13}),s_{13})
 \underbrace{-2\int_{s_{12}}^{s_{13}}\alpha(\tau) \int_{B}\partial_{s} w 
 y\cdot \nabla w {\mathrm{d}}y}_{R_1(s_{12},s_{13})}\t\no \\
&&\underbrace{-2\int_{s_1}^{s_{13}}\alpha(\tau) \ibint (\partial_{s}w)^2\y\t}_{R_2(s_{12},s_{13})}+\int_{s_{12}}^{s_{13}}\Sigma_{5}(\tau)\t,
\end{eqnarray}
where $\Sigma_{5}(s)$ satisfies \eqref{lem5e2}.
Now, we control all the terms on the right-hand side
of the relation (\ref{sec5v1}). Note that, 
by using the Mean Value Theorem, let us choose  $s_{12}=s_{12}(s)\in
[s-1,s]$  and $s_{13}=s_{13}(s)\in [s+1,s+2]$ such that
\begin{equation}\label{ws4s5}
E_0(w(s_{12}),s_{12})=\displaystyle\int_{s-1}^{s} E_0(w(\tau),\tau)\t,  \qquad 
E_0(w(s_{13}),s_{13})=\displaystyle\int_{s+1}^{s+2} E_0(w(\tau),\tau)\t,  
\end{equation}
In view of \eqref{ws4s5}, we have 
\begin{equation}\label{ws5s5} 
\big|E_0(w(s_{12}),s_{12})- E_0(w(s_{13}),s_{13})\big| \le C \int_{s-1}^{s+2} \big|E_0(w(\tau),\tau)\big|  \mathrm{d}\tau. \end{equation}
 Combining \eqref{ws5s5} with the estimates \eqref{sec5dif2} and \eqref{sec5dif1}, it follows that \begin{equation}\label{1ws5s5b} 
 \big|E_0(w(s_{12}),s_{12})- E_0(w(s_{13}),s_{13})\big| \le C \int_{s-1}^{s+2} \big|G(w(\tau),\tau)\big| \mathrm{d}\tau + K \ln s. 
 \end{equation}
 Moreover, \eqref{L1212} yields the uniform bound 
 \begin{equation}\label{ws5s5b1}
   \int_{s-1}^{s+2} \big|G(w(\tau),\tau)\big| , \mathrm{d}\tau \le K. 
 \end{equation} 
    Thus, substituting \eqref{ws5s5b1} into \eqref{1ws5s5b}, we arrive at the logarithmic estimate
 \begin{equation}   \label{ws5s5b26} 
 \big|E_0(w(s_{12}),s_{12})- E_0(w(s_{13}),s_{13})\big| \le K \ln s. 
 \end{equation}
Furthermore, 
it follows from \eqref{4feb19}  that
 \begin{equation}\label{sec5v1bb}
 R_1(s_{12},s_{13})+R_2(s_{12},s_{13})\le \frac{C}{s}
\int_{s-1}^{s+2}\ibint \big((\partial_s w)^2
+ |\nabla  w|^2\big)\y\t\le K.
\end{equation}
In view of \ref{lem5e2}, \eqref{4feb19},  and \eqref{4feb19alphab}, we have
 \begin{eqnarray}\label{sec5v1bbg}
 \int_{s_{12}}^{s_{13}}\Sigma_{5}(\tau)\t\le K.
\end{eqnarray}
 Finally, combining \eqref{sec5v1} with \eqref{ws5s5b26}--\eqref{sec5v1bbg}, we conclude that \eqref{L1212sec5} holds.\\

-Proof of  \eqref{L1212sec5b}
By exploiting \eqref{L1212sec5}, where the space integration is done over the unit sphere, one can use the averaging technique of Proposition 4.2 (page 1147) in   \cite{MZimrn05} to get
\begin{equation}\label{cor00001}
\int_{s}^{s+1}\int_{B}\Big (\partial_s w(y,\tau
)-\lambda(\tau ,s) w(y,\tau )\Big )^2dy d\tau\le K\ln s,
\end{equation}
where $0\le \lambda(\tau ,s) \le C(\delta_0)$, for all $\tau \in
[s,s+1]$.

 This  ends the proof of   Proposition \ref{psec5}.
\Box

\begin{nb}
It is important to note that we lack control over $\partial_s w$ when the singular weight is applied. Specifically, we cannot obtain an estimate for$$\int_s^{s+1}\int_{B} (\partial_{s}w(y,\tau ))^2 \frac{\rho}{1-|y|^2} dy d\tau.$$ Consequently, the strategy employed in Proposition \ref{prop21bb} cannot be adapted to prove Proposition \ref{prop21bbb2} by replacing \eqref{cor01Asec4} with estimate \eqref{L1212sec5b}. This technical obstacle also appeared in the conformal case  and was overcome  in \cite{MZimrn05,MZma05, HZjhde12}.
\end{nb}

\subsection{Uniform boundedness in $H^1 \times L^2$ for the solution of \eqref{A}}\label{section53}

By applying \eqref{L1212sec5b}, we refine the linear bound established in Proposition \ref{prop21bb} into a logarithmic estimate; specifically, we establish the following:

\begin{prop}\label{prop21bbb2}
\noindent  Consider $u $   a solution of ({\ref{gen}}) with
blow-up graph $\Gamma:\{x\mapsto T(x)\}$ and  $x_0$  a non
characteristic point under the condition \eqref{ss00}. Then,  for all $T_0\in(0,T(x_{0})],$
   $s\geq 12-\ln T_0$ and $x\in \R^n,$ where $|x-x_0|\le \frac{e^{-s}}{\delta_0(x_0)}$, 
we have
\begin{equation}\label{feb19b2}
\int_{s}^{s+1}\!\ibint \big( w^2(y,\tau )+|\grad w(y,\tau )|^2 +( \partial_{s}w(y,\tau ))^2\big)\y \t \leq K_{11}\ln s,
\end{equation}
\begin{equation}\label{feb19alphab2}
\frac{1}{s^a}\int_{s+1}^{s}\! \int_{B}   |w|^{p_c+1}\ln^a(\phi^2w^2+2) {\mathrm{d}}y{\mathrm{d}}\tau
\leq K_{11} \ln s,
\end{equation}
where $w=w_{x,T^*(x)}$ is defined in \eqref{scaling}, with
$T^*(x)$ given by \eqref{18dec1}.
\end{prop}
{\it Proof of  Proposition \ref{prop21bbb2}:}
Proposition \ref{prop21bbb2} is proved by exploiting the fact that $G(w(s),s)$ is a bounded functional, as shown in \eqref{L1212}, and the estimate \eqref{L1212sec5b}. This approach follows the strategy developed for the conformal case in \cite{MZimrn05,MZma05, HZjhde12}. Naturally, the perturbative and nonlinear terms are handled using the same techniques as those employed in the proof of Proposition \ref{prop21bb}. This concludes the proof of Proposition \ref{prop21bbb2}.\Box
\\

Thanks to Proposition \ref{prop21bbb2} and Corollary \ref{cor01sec4}, we are now in a position to improve the first estimate shown in Lemma \ref{ws01} and, accordingly, improve Proposition \ref{psec5}. More precisely, we state the following:
\begin{cors}\label{csec5}
For all $T_{0}\in (0,T(x_{0})]$,  for all
 $s \geq 12-\ln (T_{0})$ and
 $x\in \R^n$, where $|x-x_0|\le \frac{e^{-s}}{\delta_0(x_0)}$,
we have
\begin{equation}\label{dif2bis}
\int_{s}^{s+1}\Big| E(w(\tau),\tau)-E_0 (w(\tau),\tau)\Big|
{\mathrm{d}}\tau \le K_{12}\frac{(\ln s)^2}{s}, 
\end{equation}
\begin{equation}\label{L1212sec5bis}
\int_{s}^{s+1}\int_{\partial B}(\partial_{s} w)^2{\mathrm{d}}\sigma {\mathrm{d}}\tau\leq
 K_{12},
 \end{equation}
 \begin{equation}\label{L1212sec5bbis}
\int_{s}^{s+1}\int_{B}\Big (\partial_s w(y,\tau
)-\lambda(\tau ,s) w(y,\tau )\Big )^2dy d\tau\le K_{12},
\end{equation}
where $0\le \lambda(\tau ,s) \le C(\delta_0)$, for all $\tau \in
[s,s+1]$.
\end{cors}

{\it Proof}:\\ 
-Proof of  \eqref{dif2bis}:  In view of the sharpened estimates \eqref{feb19b2} and \eqref{feb19alphab2}, the bound \eqref{22A}   yields 
\begin{align}\label{17AAbis}
\int_{s}^{s+1}
|\tilde \chi_{5}(\tau)|\t \le K\frac{(\ln s)^2}{s}.
\end{align}
The result $\eqref{dif2bis}$ then follows by combining $\eqref{17AAbis}$, $\eqref{17AAA}$, 
and $\eqref{17A}$.\\

-Proof of  \eqref{L1212sec5bis}: 
We follow the strategy employed in the proof of \eqref{L1212sec5}, with one primary modification: we utilize the refined estimate \eqref{dif2bis} in place of \eqref{sec5dif2}. More precisely, by applying \eqref{ws5s5}, \eqref{dif2bis}, \eqref{sec5dif1}, and \eqref{ws5s5b1}, we arrive at the following uniform bound:
\begin{equation} \label{ws5s5b26bb} \big|E_0(w(s_{12}),s_{12}) - E_0(w(s_{13}),s_{13})\big| \le K. \end{equation}
Finally, by combining \eqref{sec5v1} with \eqref{sec5v1bb}, \eqref{sec5v1bbg}, and the estimate in \eqref{ws5s5b26bb}, we conclude that \eqref{L1212sec5bis} holds.\\

-Proof of  \eqref{L1212sec5bbis}:
By exploiting \eqref{L1212sec5bis},  one can use the averaging technique of Proposition 4.2 (page 1147) in   \cite{MZimrn05} to get \eqref{L1212sec5bbis}.
 This  ends the proof of   Corollary \ref{csec5}.
\Box

\bigskip

By utilizing the boundedness of $G(w(s),s)$ from \eqref{L1212} and the estimate in \eqref{L1212sec5bbis}, we arrive at a new  estimate for the solution of \eqref{A}. This allows us to establish the boundedness of the time-average of the $H^1$ norm, as follows:
\begin{prop}\label{prop21bbb255}
\noindent  Consider $u $   a solution of ({\ref{gen}}) with
blow-up graph $\Gamma:\{x\mapsto T(x)\}$ and  $x_0$  a non
characteristic point under the condition \eqref{ss00}. Then,  for all $T_0\in(0,T(x_{0})],$
   $s\geq 13-\ln T_0$ and $x\in \R^n,$ where $|x-x_0|\le \frac{e^{-s}}{\delta_0(x_0)}$, 
we have
\begin{equation}\label{feb19b25}
\int_{s}^{s+1}\!\ibint \big( w^2(y,\tau )+|\grad w(y,\tau )|^2 +( \partial_{s}w(y,\tau ))^2\big)\y \t \leq K_{13},
\end{equation}
\begin{equation}\label{feb19alphab25}
\frac{1}{s^a}\int_{s+1}^{s}\! \int_{B}   |w|^{p_c+1}\ln^a(\phi^2w^2+2) {\mathrm{d}}y{\mathrm{d}}\tau
\leq K_{13},
\end{equation}
where $w=w_{x,T^*(x)}$ is defined in \eqref{scaling}, with
$T^*(x)$ given by \eqref{18dec1}.
 \end{prop}

\medskip

{\it Proof}:
Following the strategy for the conformal case established in \cite{MZimrn05, MZma05}, we exploit the boundedness of $G(w(s), s)$ in \eqref{L1212} and the estimate \eqref{L1212sec5bbis}. Since the perturbative and nonlinear components can be handled by the same arguments used for Proposition \ref{prop21bb}, the  result follows. This concludes the proof of Proposition \ref{prop21bbb255}.
 \Box

\medskip

{{\it {Proof of Theorem \ref{t2}}}} :
 The boundedness of the $H^1(B)$ norm for the solution to \eqref{A} follows from the boundedness of the functional $G(w(s),s)$ established in \eqref{L1212}, together with Proposition \ref{prop21bbb255}, by following the strategy developed for the conformal case in \cite{MZimrn05,MZma05}. The perturbative and nonlinear terms are controlled using the estimates \eqref{A30}, \eqref{26control1},  \eqref{b26control1}, and \eqref{equiv6}.  This completes the proof of upper bound of  \eqref{main1}.\\

Regarding the lower bound in \eqref{main1}, the proof in the case of pure nonlinearity follows from well-posedness in the energy space, the finite speed of propagation, the fact that $x_0$ is not characteristic, and scaling arguments. In the present case, while we lose the scaling, the argument extends naturally to the setting where $a < 0$. Indeed, since $|f(u)| \le C(|u|^p + |u|)$, the nonlinear term is controlled by the same Strichartz norms as in the pure power case. For sufficiently small initial data, the additional logarithmic factor does not affect the contraction argument. This completes the proof of Theorem \ref{t2}. $\Box$

 \appendix

 \section{Some elementary lemmas.}
Let the functions $\phi$, $f$, $f_1$, and $f_2$ be defined as in \eqref{defphi}, \eqref{defF}, \eqref{defF2}, and \eqref{defF123}, respectively.
Clearly, we have 
\begin{lem} \ {\  }Let $r>1$,\\ 
\begin{align}
\int_0^u|v|^{r-1}v\ln^{{a}}(2+v^2 )\v  \sim&  \frac{| u|^{r+1}}{r+1}\ln^{{a}}(u^2+2  ),\quad  \text{ as } \;\; |u| \to \infty,
\label{estF0}\\
f(u)  \sim & \frac{| u|^{p_c+1}}{p_c+1}\ln^{{a}}(u^2+2  )  \quad  \text{ as } \;\; |u| \to \infty,\label{estF}\\
f_2(u)\sim & C| u|^{p_c+1}\ln^{{a-2}}(u^2+2) \quad  \text{ as } \;\; |u| \to \infty.\label{estF3}
\end{align}
\end{lem}
\begin{proof} 
 See Lemma A.1 
 in \cite{HZjmaa2020}.
\end{proof}
\Box

Thanks to \eqref{estF0}, \eqref{estF} and \eqref{estF3}, we can
  state and prove  the following estimates: 
\begin{lem}\label{lemm:esth}
For all  $s \geq 1$,  and $z\in \R$, we have
\begin{align}
 C^{-1}|\phi z|^{p_c+1} \ln^a(\phi^2 z^2+2)\le C+f(\phi z)\le C +C|\phi z|^{p_c+1} \ln^a(2+\phi^2 z^2),\label{equiv1}\\
f_1(\phi z)\le C+\frac{C}{s}|\phi z|^{p_c+1} \ln^a(2+\phi^2 z^2),\quad\quad\label{equiv2}\\
f_2(\phi z)\le C+\frac{C}{s^2}|\phi z|^{p_c+1} \ln^a(2+\phi^2 z^2),\quad\quad\label{equiv3}\\
 \frac{1}{s^a} |z|^{p_c+1} \ln^a(2 + \phi^2 z^2) 
\le  C| z|^{ p_c+1}+Ce^{-s}, \quad \textrm{if}\quad  a<0.\quad\label{equiv6}
\end{align}
Furthermore,    there exists $\tilde{S}_0 > 1$ such that for all $s \ge \tilde{S}_0$,  $\varepsilon > 0$, and $z \in \mathbb{R}$, the following estimate holds:
\begin{equation}\label{equiv5}
z^2 \le \frac{\varepsilon}{s^a} |z|^{p_c+1} \ln^a(2 + \phi^2(s) z^2) + C(\varepsilon), \quad \textrm{if}\quad  a<0.
\end{equation}
\end{lem}

\begin{proof} 
The estimate \eqref{equiv1} follows immediately from \eqref{estF}. The derivation of \eqref{equiv2} and \eqref{equiv3} is obtained by partitioning the domain into $\{z^2\phi \geq 4\}$ and $\{z^2\phi \leq 4\}$ and applying \eqref{estF0}--\eqref{estF3}.
Furthermore, the assumption $a < 0$ yields \eqref{equiv6} via a similar argument. We prove the inequality \eqref{equiv5} by distinguishing two cases regarding the size of $|z|$:
\begin{itemize}
\item \textbf{Case 1:} $|z| \le \phi(s)$: Since the function $\d \xi \mapsto \ln^a(2+\xi)$ is strictly decreasing on $[2,\infty)$, it follows that
\begin{equation}\label{equiv8}
\frac1{s^a} \ln^a(2+\phi^2(s)z^2) \ge \frac1{s^a}\ln^a(2+\phi^4(s)).
\end{equation}
By the asymptotic property $\d \lim_{s \to \infty} \frac{1}{s^a} \ln^a(2 + \phi^4(s)) = \frac{8^a}{(p_c-1)^a}  > 0$, there exists $\tilde{S}_0 > 1$ such that
\begin{equation}\label{equiv8b}
\frac{1}{s^a} \ln^a(2 + \phi^4(s)) \ge \frac{8^a}{2(p_c-1)^a}, \quad \forall s \ge \tilde{S}_0.
\end{equation}
Combining \eqref{equiv8} and \eqref{equiv8b}, we obtain
\begin{equation}\label{equiv8c}\frac1{s^a}\ln^a(2+\phi^2(s)z^2) \ge  \frac{8^a}{2(p_c-1)^a}, \quad \forall s \ge \tilde{S}_0.
\end{equation}
Applying Young’s inequality in the form $z^2 \le \nu |z|^{p_c+1} + C(\nu)$ with the choice $\nu = \frac{\varepsilon 8^a}{2(p_c-1)^a}$, we arrive at \eqref{equiv5}.\item \textbf{Case 2:} $|z| \ge \phi(s)$.Using the monotonicity of the logarithm, we have$$\frac1{s^a}\ln^a(2+\phi^2(s)z^2) \ge \frac1{s^a}\ln^a(2+z^4).$$Standard asymptotic analysis implies that for any $\varepsilon > 0$, there exists a constant $C(\varepsilon) > 0$ such that
\begin{equation}\label{aux_est}z^2 \le \varepsilon |z|^{p_c+1} \ln^a(2 + z^4) + C(\varepsilon).
\end{equation}
Combining \eqref{aux_est} with the observation that $s^{-a} \ge 1$ for $s \ge \tilde{S}_0$, we conclude the proof of \eqref{equiv5}.
\end{itemize}
This concludes the proof of Lemma \ref{lemm:esth}. \Box
\end{proof}

 \section{ Some elementary 
  identities  }
  In this section we prove two identities   which play an instrumental role
in our main argument.
First evaluate  the term
related to the Pohozaev multiplier given by:
\begin{equation}\label{AJ0}
{\cal{A}_{\eta}}(s)=\int_{B}y\cdot\grad  w \Big(\div(\p \nabla w-\p (y\cdot\nabla w)y)\Big) \y,
\end{equation}
where $\eta>0$.  More precisely, we prove   
the following identity:
\begin{lem}\label{L1} 
For all $w\in {\cal  H}$ it holds that
\begin{eqnarray}\label{mport01}
{\cal{A}_{\eta}}(s)  &=&-\eta \int_{B}|\nabla_{\theta}w|^2
\frac{|y|^2\p}{1-|y|^2}\y
-\eta \int_{B}(y\cdot\grad  w)^2
\p{\mathrm{d}}y\noindent\\
&&+\frac{n}2\int_{B}\Big(|\grad w|^2-(y\cdot\grad  w)^2\Big) \p
{\mathrm{d}}y -\int_{B}|\grad w|^2 \p
{\mathrm{d}}y\cdot\nonumber
\end{eqnarray}
\end{lem}
{\it Proof}: 
We divide ${\cal{A}_{\eta}}(s)  $ into three terms:
${\cal{A}_{\eta}}(s)  ={\cal{B}_{\eta}}(s)  +{\cal{C}_{\eta}}(s)  +{\cal{D}_{\eta}}(s),$ where
\begin{equation*}\label{}
{\cal{B}}_{\eta}(s)=\int_{B}(y\cdot\grad  w) \Delta w\p \y,
\end{equation*}
\begin{equation*}\label{A2}
{\cal{C}}_{\eta}(s)=-\int_{B}(y\cdot \grad  w) \div( (y\cdot \nabla w)y)\p \y,
\end{equation*}
and
\begin{equation*}\label{A0}
{\cal{D}_{\eta}}(s)  =\int_{B}(y\cdot\grad  w) \Big( \nabla w-(y\cdot \nabla w)y \Big)\cdot \nabla \p \y .
\end{equation*}

To estimate ${\cal{A}}_{\eta}(s)$, we start
observe the immediate  identity
\begin{equation}\label{A3}
  (y\cdot\grad w) \Delta w= \sum_{i,j}y_i\partial_i w
\partial^2_jw.
\end{equation}
By  integrating by parts, exploiting \eqref{A3} and  the fact that  $\ds{\sum_{i,j}\delta_{i,j}\partial_i w
\partial_jw=|\grad w|^2}$,      we can write
\begin{eqnarray}\label{A4}
{\cal{B}}_{\eta}(s) &=&-\frac12\sum_{i,j}\int_{B}y_i\partial_i((
\partial_j w)^2)\p \y-\int_{B}|\grad w|^2
\p\y \no\\ &&-\sum_{i,j}\int_{B}y_i\partial_i w
\partial_j w\partial_j\p \y.\nonumber
\end{eqnarray}
By using the identity  $\partial_j\p=- \frac{2\eta y_j}{1-|y|^2}\p$, 
\eqref{A4} and  integrating by part one has that
\begin{eqnarray}\label{A5}
{\cal{B}}_{\eta}(s)  =\frac12\int_{B}|\grad w|^2\div (\p y)
{\mathrm{d}}y-\int_{B}|\grad w|^2 \p
{\mathrm{d}}y+2\eta \int_{B}(y\cdot \grad w)^2\frac{\p}{1-|y|^2}
 \y.
\end{eqnarray}
Then,  from \eqref{A5}, and by using the identity  $\div (\p y)=n\p+y\cdot \nabla \p$,   we get 
\begin{eqnarray}\label{A6}
{\cal{B}}_{\eta}(s)  =-\eta \int_{B}|\grad w|^2\frac{|y|^2\p}{1-|y|^2}
{\mathrm{d}}y+2\eta \int_{B}(y\cdot \grad w)^2\frac{\p}{1-|y|^2}\y+\frac{n-2}2\int_{B}|\grad w|^2 \p
{\mathrm{d}}y .\no
\end{eqnarray}
To estimate ${\cal{C}}_{\eta}(s)$, we start with the use of the classical identity
\begin{equation}\label{A11}
\div( (y\cdot\nabla w)y)= n (y\cdot\nabla w)+ \grad (y\cdot\nabla w)\cdot y,
\end{equation}
and  integrating by part, to obtain
\begin{equation}\label{A12}
{\cal{C}}_{\eta}(s) =-n\int_{B}(y\cdot\grad
w)^2\p{\mathrm{d}}y +\frac12 \int_{B}(y\cdot \grad w)^2 \div (
\p y) {\mathrm{d}}y.
\end{equation}
By using  the basic identity $\div (
\p y) =n\p-2\eta\frac{|y|^2}{1-|y|^2}\p$, 
we write
\begin{eqnarray}\label{A13}
{\cal{C}}_{\eta}(s)&=&-\frac{n}2\int_{B}(y\cdot \grad  w)^2\p
{\mathrm{d}}y-\eta \int_{B}(y\cdot \grad  w)^2
\frac{|y|^2\p}{1-|y|^2} {\mathrm{d}}y.
\end{eqnarray}
Furthermore, by exploiting  the identity $\grad \p=-\frac{2\eta}{1-|y|^2}y,$  we conclude easily
\begin{equation}\label{A14}
{\cal{D}}_{\eta}(s)
=
-2\eta\int_{B}(y\cdot\grad  w)^2 \p\y.
\end{equation}
By combining  \eqref{A6}, \eqref{A13} and  \eqref{A14} 
we deduce easily
\begin{eqnarray}
{\cal{A}}(s)  &=&-\eta \int_{B}(|\grad w|^2-(y\cdot\grad  w)^2)
\frac{|y|^2\p}{1-|y|^2}
{\mathrm{d}}y\\
&&+\frac{n}2\int_{B}(|\grad w|^2-(y\cdot\grad  w)^2)\p
{\mathrm{d}}y-\int_{B}|\grad w|^2 \p
{\mathrm{d}}y.\nonumber
\end{eqnarray}
From the identity $|\grad w|^2-(y\cdot\grad  w)^2)=|\grad_{\theta} w|^2+(1-|y|^2)|\grad_r  w|^2$, the above estimate implies
\eqref{mport1}, which ends the proof of Lemma \ref{L1}. 
\Box

\bigskip

Then, we evaluate  the term
\begin{equation}\label{A9}
{\cal{I}_{\eta}}(s)=- \int_{B}\div(\p \grad w-\p (y\cdot\grad w)
y)w\frac1{\sqrt{1-|y|^2}}{\mathrm{d}}y,
\end{equation}
  More precisely, we prove   
the following identity:
\begin{lem}\label{L52} 
For all $w\in {\cal  H}$ it holds that
\begin{eqnarray}\label{mport1}
{\cal{I}_{\eta}}(s)
 &=& \int_{B}|\grad w_{\theta}|^2\pp{\mathrm{d}}y+\int_{B}|\grad w_r|^2{\rho_{\eta+\frac12}}{\mathrm{d}}y
\\
&&+\int_{B}w y\cdot\grad w
 \frac{\p}{\sqrt{1-|y|^2}}{\mathrm{d}}y.\nonumber
\end{eqnarray}
\end{lem}
{\it Proof}: 
By integrating by parts, and using the fact that
$$ \int_{\partial B}y\cdot (\p \grad w-\p (y\cdot\grad w)
y)w\frac1{\sqrt{1-|y|^2}}{\mathrm{d}}\sigma= \int_{\partial B}(y\cdot \grad w) 
w{\sqrt{1-|y|^2}}\p {\mathrm{d}}\sigma=0,$$
 we write
\begin{equation}\label{AA1}
{\cal{I}_{\eta}}(s)
 =\int_{B}(\p \grad w-\p (y\cdot \grad w)
y)\nabla (w\frac1{\sqrt{1-|y|^2}}){\mathrm{d}}y.
\end{equation}
A straightforward computation yields
 the identity
\begin{equation}\label{AA0}
\nabla (w\frac1{\sqrt{1-|y|^2}})=\frac1{\sqrt{1-|y|^2}} \nabla w+ \frac{1}{(1-|y|^2)^{\frac32}}wy.
\end{equation}
Substituting \eqref{AA0} into \eqref{AA1}, we have that
\begin{eqnarray}\label{V1}
{\cal{I}_{\eta}}(s)
 &=&\int_{B}\Big(|\grad w|^2-(y\cdot\grad w)^2\Big)\frac{\p}{\sqrt{1-|y|^2}}{\mathrm{d}}y\nonumber\\
&&+\int_{B}w y\cdot\grad w
 \frac{\p}{\sqrt{1-|y|^2}}{\mathrm{d}}y.
\end{eqnarray}
By  using \eqref{12nov6}, we write easily
\begin{equation}\label{V2}
\int_{B}\Big(|\grad w|^2-(y\cdot\grad w)^2\Big)\frac{\p}{\sqrt{1-|y|^2}}{\mathrm{d}}y
  =\int_{B}|\grad w_{\theta}|^2\pp{\mathrm{d}}y+\int_{B}|\grad w_r|^2{\rho_{\eta+\frac12}}{\mathrm{d}}y.
\end{equation}
Substituting  \eqref{V2} into \eqref{V1}, we get \eqref{mport1}. 
This ends the proof of Lemma \ref{L52}. 
\Box

\section*{Acknowledgments}

The author would like to express his deep gratitude to Professor Hatem Zaag for the fruitful discussions that have greatly enriched this work.

\section*{Declarations}
\begin{itemize}
    \item \textbf{Conflict of interest:} The author declares no conflict of interest.
    \item \textbf{Data availability:} This manuscript has no associated data.
\end{itemize}


\noindent{\bf Address}:\\
 Department of Basic Sciences, Deanship of Preparatory and Supporting Studies,
  Imam Abdulrahman Bin Faisal University,
P.O. Box 1982 Dammam, Saudi Arabia.\\
\vspace{-7mm}
\begin{verbatim}
e-mail:  mahamza@iau.edu.sa
\end{verbatim}

\end{document}